\documentclass[12pt]{amsart}

\usepackage{xcolor}
\usepackage{pdflscape}
\newcommand{\note}[1]{{\color{red}#1}}

\usepackage[margin=1.4in]{geometry}
\usepackage{microtype}
\usepackage{hyperref}
\usepackage{tikz}
\usepackage{amssymb,amsmath,amsfonts,epsfig,latexsym, xypic, hyperref}

\def\Rt{\mathbb{R}\{\!\{t\}\!\}}
\def\kt{k\{\!\{t\}\!\}}

\DeclareMathOperator{\Tr}{Tr}
\DeclareMathOperator{\Trop}{Trop}
\DeclareMathOperator{\mult}{mult}
\DeclareMathOperator{\val}{val}
\DeclareMathOperator{\ini}{in}
\DeclareMathOperator{\supp}{supp}
\DeclareMathOperator{\GW}{GW}
\def\Qtype{\mathrm{Qtype}_{L_{\infty}}}
\DeclareMathOperator{\ddiv}{div}

\newcommand{\legendre}[2]{\ensuremath{\left( \frac{#1}{#2} \right) }}

\newtheorem{theorem}{Theorem}
\numberwithin{theorem}{section}

\newtheorem{conjecture}[theorem]{Conjecture}
\newtheorem{proposition}[theorem]{Proposition}

\newtheorem{lemma}[theorem]{Lemma}
\newtheorem{corollary}[theorem]{Corollary}

\newtheorem*{proposition*}{Proposition}
\newtheorem*{theorem*}{Theorem}
\newtheorem*{lemma*}{Lemma}
\newtheorem*{claim*}{Claim}
\newtheorem*{conjecture*}{Conjecture}

\theoremstyle{definition}
\newtheorem{definition}[theorem]{Definition}

\theoremstyle{remark}
\newtheorem{example}[theorem]{Example}
\newtheorem{remark}[theorem]{Remark}
\newtheorem*{example*}{Example}

\usepackage[boxed,vlined,lined,algoruled]{algorithm2e}
\newcommand{\Subsets}{\operatorname{Subsets}}
\newcommand{\expec}{\ensuremath{\operatorname{exp}}}
\newcommand{\noNones}{\ensuremath{\operatorname{true}}}
\newcommand{\new}{\ensuremath{\operatorname{new}}}

\renewcommand{\algorithmautorefname}{Algorithm}

\title[Bitangents to plane quartics via tropical geometry]{Bitangents to plane quartics via tropical geometry: rationality, $\mathbb{A}^1$-enumeration, and real signed count} 
\newtheorem{innercustomthm}[theorem]{Theorem}

\author {Hannah Markwig}
\address {Hannah Markwig, Eberhard Karls Universit\"at T\"ubingen, Fachbereich Mathematik, Auf der Morgenstelle 10, 72076 T\"ubingen, Germany }
\email {hannah@math.uni-tuebingen.de}
\author {Sam Payne}
\address {Department of Mathematics, University of Texas at Austin, 2515 Speedway, PMA 8.100, Austin, TX 78712, USA}
\email {sampayne@utexas.edu}
\author {Kris Shaw}
\address {Department of Mathematics, University of Oslo, Postboks 1053, Blindern, 0316 Oslo, Norway}
\email {krisshaw@math.uio.no}


\begin{document}
\maketitle

\begin{abstract}
We explore extensions of tropical methods to arithmetic enumerative problems such as $\mathbb{A}^1$-enumeration with values in the Grothendieck-Witt ring and rationality over Henselian valued fields, using bitangents to plane quartics as a test case. We consider quartic curves over valued fields whose tropicalizations are smooth and satisfy a mild genericity condition.  We then express obstructions to rationality of bitangents and their points of tangency in terms of twisting of edges of the tropicalization; the latter depends only on the tropicalization and the initial coefficients of the defining equation modulo squares. 
 We also show that the $\GW$-multiplicity of a tropical bitangent, i.e., the multiplicity with which its lifts contribute to the $\mathbb{A}^1$-enumeration of bitangents as defined by Larson and Vogt \cite{LV21}, can be computed from the tropicalization of the quartic together with the initial coefficients of the defining equation. As an application, we show that the four lifts of most tropical bitangent classes contribute $2\mathbb{H}$, twice the class of the hyperbolic plane, to the $\mathbb{A}^1$-enumeration.%, providing evidence for a conjecture of Larson and Vogt on the $\mathbb{A}^1$-enumeration of bitangents over the reals. 
 These results rely on a degeneration theorem relating the Grothendieck-Witt ring of a Henselian valued field to the Grothendieck-Witt ring of its residue field, in residue characteristic not equal to two.

\end{abstract}

\tableofcontents

\section{Introduction}

Tropical geometry is well-known for its applications to complex and real enumerative geometry, especially for plane curves.  Here, we explore extensions of tropical methods to arithmetic enumerative problems such as $\mathbb{A}^1$-enumeration with values in the Grothendieck-Witt ring, and rationality over Henselian valued fields, using bitangents to plane quartics as a test case.

If the tropicalization of a plane quartic is smooth (locally isomorphic to the tropicalization of a linear space) then it has exactly 7 deformation classes of tropical bitangents. These are in natural bijection with the odd tropical theta characteristics and each contains the tropicalization of 4 algebraic bitangents \cite{Chan2, JensenLen18, LM17}.  Cueto and the first author have given an exhaustive classification of the combinatorial types of these bitangent classes \cite{CM20}. As an application, when the ground field $K$ is real Puiseux series $\Rt$ and the tropicalization satisfies the genericity constraints from \cite[Remark~2.10]{CM20}, they showed that the number of $K$-rational bitangents tropicalizing into a bitangent class is either 0 or 4. In other words, roughly speaking, the obstruction to rationality over $\Rt$ is independent of the choice of bitangent in a given bitangent class. 

\medskip 

Throughout, let $K$ be a Henselian valued field of residue characteristic not equal to 2, with 2-divisible value group, and with a section of the valuation $\sigma \colon \val(K^\times) \to K^\times$. Let $k$ be the residue field of $K$. Let $Q \in K[x,y,z]$ be a homogeneous polynomial of degree 4 such that the projective plane algebraic curve $C = V(Q)$ is smooth and $\Trop(C)$ is tropically smooth. Assume furthermore that $\Trop(C)$ satisfies the genericity constraints from \cite[Remark~2.10]{CM20}, the details of which we recall in Section~\ref{sec:prelimsbitangents}.

Note that $K$ could be the Puiseux series field $\kt$ over any field $k$ of characteristic not equal to $2$. Such fields are commonly considered in tropical geometry.  Other possibilities include mixed characteristic fields such as $\mathbb{Q}_p(p^{1/2^{\infty}}) = \bigcup_n \mathbb{Q}_p (p^{1/2^{n}})$, for $p \neq 2$.  Our arguments are not sensitive to distinctions between pure and mixed characteristic; any reader who prefers to do so may safely assume that $K = \kt$ is a field of Puiseux series.

\addtocontents{toc}{\protect\setcounter{tocdepth}{1}}
\subsection*{Lifting of tropical bitangents to general ground fields}
Our work begins with the observation that the algebraic calculations used to prove the lifting result from \cite{CM20} over the reals are only mildly sensitive to the ground field. 

\begin{theorem} [Theorem \ref{thm:twistlift}]\label{thm:twistliftIntro}
 Let $C = V(Q)$ be a quartic curve defined over $K$ with $\Trop(C)$ a generic tropicalized quartic and suppose $\Lambda$ is a liftable tropical bitangent to $C$. Then whether or not $\Lambda$ lifts over $K$ is determined by $\Trop(C)$  and the equivalence classes of  initials of the coefficients of $Q$ in $k^\times/(k^\times)^2$. 
\end{theorem}

When the residue field $k = \mathbb{R}$, then $\mathbb{R}^\times/(\mathbb{R}^\times)^2 \cong \mathbb{Z}/2\mathbb{Z}$. Knowing the tropicalization of $C$ together with the equivalence classes of initials of the coefficients of $Q$ modulo squares is precisely the  information required to carry out Viro's patchworking \cite{IMS09, Viro}. 
 Motivated by this connection, we express obstructions to rationality in the case of arbitrary fields in terms of \emph{twisting} of edges  of the tropicalization, thus extending to general fields the concept of twisting coming from amoebas over the real numbers, see  \cite{BIMS}.

\begin{theorem} [Theorem \ref{thm-allornone}] \label{thm:rationalbitangents}
The number of $K$-rational bitangents to $C$ that tropicalize into each bitangent class of tropical bitangents is either 0 or 4.
\end{theorem}

When $K$ is $\Rt$, Cueto and the first author showed that the points of tangency of a $K$-rational bitangent are $K$-rational \cite[Theorem~7.1]{CM20}. Again, we observe that this is only mildly sensitive to the ground field.

\begin{theorem} [Corollary \ref{cor-totallyreal}] \label{thm:rationaltangencies}
Assume that $2$ and $3$ are squares in $k^\times$.  Then the points of tangency of a $K$-rational bitangent to $C$ are $K$-rational.
\end{theorem}
\noindent Note that the hypotheses of Theorem~\ref{thm:rationaltangencies} imply that the residue characteristic of $K$ is also not equal to 3.

The proofs of Theorems~\ref{thm:rationalbitangents} and \ref{thm:rationaltangencies} involve solving systems of equations over the residue field to determine the initials of the coefficients of the defining equation of the bitangents and points of tangency in suitable local coordinates, and then lifting via Hensel's lemma. This is carried out case-by-case, according to the Cueto-Markwig classification from \cite{CM20}. Details of the case analysis are provided in the Appendix.

Using the lifting theorems discussed above, we also prove that the  liftability of tropical bitangents over different fields are related  when the groups of the residue fields modulo squares are isomorphic, see Theorem \ref{thm-comparinglifting} and Corollary \ref{cor:compareliftingRk}.

\addtocontents{toc}{\protect\setcounter{tocdepth}{1}}
\subsection*{The tropical approach to the $\mathbb{A}^1$-enumeration of bitangents to a plane quartic, and its foundations}
Even when it is possible to count $K$-rational solutions to a geometric problem, the naive enumeration may not be invariant under deformations. Familiar examples include counts of real rational plane curves of given degree and genus, passing through an appropriate number of specified points in general position. In such cases, there are often \emph{refined} enumerative counts, such as Welschinger invariants for real rational plane curves, where objects are counted with multiplicities in such a way that the result is invariant under deformation.  More recently, techniques from arithmetic geometry and $\mathbb{A}^1$-homotopy theory have produced more general \emph{$\mathbb{A}^1$-enumerative invariants} taking values in Grothendieck-Witt rings of quadratic forms that specialize to classical enumerative invariants over $\mathbb{C}$, to Welschinger invariants and related real enumerative invariants over $\mathbb{R}$, and produce new invariants of interest over other fields \cite{Hoyois14, KW19, Levine20}.  For the $\mathbb{A}^1$-enumeration of the 27 lines on cubic surfaces and the 28 bitangents to plane quartics, see \cite{KW20} and \cite{LV21}, respectively. We recall the relevant definitions in Section~\ref{sec-GW}.

We write $t^v:=\sigma(v)$ for $v\in \val(K^{\times})$, where $\sigma$ denotes the section of the valuation on our valued field $K$. For $A\in K$, we let $\ini(A)$ be the image of $At^{-\val(A)}$ in the residue field and call it the initial of $A$. If $K$ is the field of Puiseux series, $\ini(A)$ is the leading coefficient of the series $A$.

\begin{theorem}[Theorem \ref{thm-2H}] \label{thm:tropicalenumeration}
The $\mathbb{A}^1$-enumerative multiplicity of a bitangent $L$ to $C$ depends only on the tropicalization of $C$ and on the initials of its coefficients modulo squares.
\end{theorem}

\noindent Moreover, we give an effective algorithm for computing the $\mathbb{A}^1$-enumerative count of all bitangents to $C$ in each of the 7 deformation classes, and hence the $\mathbb{A}^1$-enumerative count of all 28 bitangents, as elements of the Grothendieck-Witt ring $\GW(K)$ (see Remark \ref{rem-computation}).  

Theorem~\ref{thm:tropicalenumeration} relies on the following observation relating the Grothendieck-Witt ring of $K$ to that of its residue field.  Recall our standing assumption that $K$ is a Henselian valued field of residue characteristic not equal to 2, with 2-divisible value group, and with a section of the valuation $\sigma \colon \val(K^\times) \to K^\times$.

\begin{theorem}[Theorem \ref{thm-GWKGWk}]\label{thm-GWKGWkintro}
Let $K$ be a Henselian valued field of residue characteristic not equal to 2, with 2-divisible value group, and with a section of the valuation.  Then there is an isomorphism of Grothendieck-Witt rings $$\GW(K) \xrightarrow{\sim} \GW(k), \ \ \langle A \rangle \mapsto \langle \ini(A) \rangle.$$ This isomorphism does not depend on the section $\sigma$.
\end{theorem}

The above theorem is a variant of Springer's theorem on the Witt group of a Henselian discretely valued field, as generalized to arbitrary value groups, e.g., in \cite{ET}, in the special case where the value group is 2-divisible. Theorems of this form are a natural starting point for the application of tropical methods to $\mathbb{A}^1$-enumerative geometry. 
In forthcoming work, Jaramillo Puentes  and Pauli  define  enriched tropical intersection multiplicities for $0$-dimensional intersections of tropical hypersurfaces \cite{JPP}. As a consequence they prove a tropical version of a quadratically  enriched  B\'ezout theorem from \cite{McKean} and  also introduce an enriched  Bernstein-Khovanskii-Kushnirenko theorem. 

\begin{remark} \label{rem:2-divis}
Our hypothesis that the value group $\val(K^\times)$ is 2-divisible is a simplifying assumption that is necessary for Theorems~\ref{thm:tropicalenumeration} and \ref{thm-GWKGWkintro}, as stated. Nevertheless, this framework, and all of the computations that we carry out for bitangents to plane quartics, can be extended to the case where $K$ is Henselian of residue characteristic not equal to 2 with arbitrary value group, but then one must also keep track of the valuations of the coefficients, modulo valuations of squares.  The calculations needed for our case analysis of bitangents are already sufficiently complicated when the value group is 2-divisible that we have chosen to compromise generality in this way. This work, along with that of Jaramillo Puentes and Pauli \cite{JPP}, demonstrates that tropical methods are relevant and useful to $\mathbb{A}^1$-enumerative geometry.
\end{remark}

If an $\mathbb{A}^1$-enumerative multiplicity equals the hyperbolic plane $\mathbb{H} = \langle 1 \rangle \oplus \langle -1 \rangle$, then it contributes $2$ to the corresponding complex count and $0$ to the corresponding signed count over the real numbers. 
We prove that for many types of tropical bitangent classes in the classification, the $\mathbb{A}^1$-enumerative multiplicities of the four algebraic bitangents that tropicalize into this class add up to $2\mathbb{H}$. The exceptional cases for which this does not hold true are listed in Appendix \ref{app-tableQtypes}.
To give an impression which tropical bitangent classes contribute $2\mathbb{H}$ to the $\mathbb{A}^1$-enumerative count, we list some sufficient conditions here. For the complete list in terms of the classification of tropical bitangent classes, see Theorem \ref{thm-2Hdetails}.

\begin{theorem}[Theorems \ref{thm-2H} and \ref{thm-2Hdetails}] \label{thm:2Hintro}
Let $S$ be a tropical bitangent class that is compact.  Then the contribution of $S$ to the $\mathbb{A}^1$-enumerative count of bitangents to $C$ is $2 \mathbb{H}$.
\end{theorem}

\noindent We also show that the contributions of  tropical bitangents classes  to the $\mathbb{A}^1$-enumerative counts  
over different fields are related if the groups of the residue fields modulo squares are isomorphic, see Theorem \ref{thm:qtypescomparison}  and Corollary \ref{cor:LegendreRfinite}.

\addtocontents{toc}{\protect\setcounter{tocdepth}{1}}
\subsection*{The real signed count of bitangents to plane quartics}

In \cite{LV21}, Larson and Vogt investigated the $\mathbb{A}^1$-enumeration of bitangents to a plane quartic, relative to a fixed line at infinity, and the resulting signed count over reals. Over the reals, they show that the signed count is equal to 4 if the real locus of the quartic does not meet the line at infinity. They also show that the signed count is nonnegative and conjectured that it is bounded above by 8.\footnote{Kummer and McKean have now announced a proof of this conjecture \cite{KM23}.} The field of real Puiseux series $\mathbb{R}\{\!\{t\}\!\}$ also has its Grothendieck-Witt ring generated by the two elements $\langle 1 \rangle$ and $\langle -1 \rangle$. The $\mathbb{A}^1$-enumeration of bitangents can thus be expressed as $a\cdot \langle 1 \rangle+b\cdot \langle -1 \rangle$ for some $a,b$, and the signed count is $a-b$. 

\begin{theorem}[Theorem \ref{thm-realQtypes}] \label{thm-realQtypesintro}
Let $C=V(Q)$ be a quartic over $\mathbb{R}\{\!\{t\}\!\}$ whose tropicalization is smooth and generic. Then the signed count of bitangents of $C$ is either $0$, $2$ or $4$.
\end{theorem}

As many, but not all, of the tropical bitangent classes contribute $2\mathbb{H}$ to the $\mathbb{A}^1$-enumeration of bitangents, we deduce that many tropical bitangents give a total contribution of $0$ to the signed count when working over the reals. The total signed count of $2$ and $4$ can only be reached for those tropicalized quartics which admit exceptional tropical bitangent classes as listed in Appendix \ref{app-tableQtypes}. The strategy to prove Theorem \ref{thm-realQtypes} is to investigate how the dual motifs (see Definition \ref{def-dualmotif}) of such exceptional tropical bitangents can fit into the dual Newton subdivision of a tropicalized quartic and show that, at most, there can be either one that contributes 4, or one or two that contribute 2 to the signed count, with the rest contributing 0.

\addtocontents{toc}{\protect\setcounter{tocdepth}{1}}
\subsection*{Organization of this paper}

This paper is organized as follows. In Section~\ref{sec-prelim}, we discuss preliminaries. We fix our convention for the fields we study, introduce tropicalizations of plane quartics and discuss tropical bitangents and their bitangent classes. 

In Section \ref{sec-Lifting}, we discuss obstructions to rationality. Section \ref{subsec-locallift} reviews known techniques to solve so-called local lifting equations which provide local obstructions for a tropical bitangent to have a $K$-rational lift. In Section \ref{subsec-twisting}, we generalize the concept of twisted edges from the reals to more general fields. We also generalize to situations with so-called relative twisting, in which the relative behaviour of a tropical bitangent line and the tropicalized quartic play a role. We express obstructions to $K$-rationality in terms of twisting of edges. In Section~\ref{subsec-comparelift}, we sum up our results on $K$-rationality of tropical bitangent classes, and compare the behaviour over different fields. 

Section \ref{sec-GW} discusses aspects of $\mathbb{A}^1$-enumeration. Section~\ref{subsecGWintro} recalls the definition of the Grothendieck-Witt ring. Section~\ref{subsec:GWisom} contains our result (Theorem \ref{thm-GWKGWkintro}  (\ref{thm-GWKGWk})) relating the Grothendieck-Witt ring of $K$ to the Grothendieck-Witt ring of its residue field. This result lays the foundation for the infusion of tropical methods into the study of $\mathbb{A}^1$-enumerative geometry, and we apply it in Section~\ref{subsec-tropQtype} to investigate the $\mathbb{A}^1$-enumerative geometry of tropical bitangent classes to tropicalized quartics. Our main result here is Theorem~\ref{thm-2H}, which sums up the statements of Theorems \ref{thm:tropicalenumeration} and \ref{thm:2Hintro} in the introduction. Also in Section~\ref{subsec-tropQtype}, we provide the methods to prove that many tropical bitangent classes contribute $2\mathbb{H}$ to the $\mathbb{A}^1$-enumeration of bitangents to quartics, and spell out the details for one case of the classification of tropical bitangent classes. The remaining cases are revisited in Appendix \ref{app:2H}.
Before we turn to the  $\mathbb{A}^1$-enumerative geometry of tropical bitangents, in Section~\ref{subsec-GWbitangent}, we review the results of \cite{LV21} on the $\mathbb{A}^1$-enumeration of bitangents to quartics. Section~\ref{subsec-tropQtype} presents the results 
from the tropical approach to the $\mathbb{A}^1$-enumeration of bitangents to quartics and Section~\ref{subsec-comparingQtype} explores relations of the  $\mathbb{A}^1$ types of lifts of tropical bitangents over different fields. 

In Section \ref{sec-real}, we consider tropical $\mathbb{A}^1$-enumeration of bitangents to quartics over real closed fields. We prove Theorem \ref{thm-realQtypesintro} (\ref{thm-realQtypes}) providing evidence in support of Larson and Vogt's conjecture on the possible values for the signed count of bitangents over the reals. We also point out that this result can be obtained without referring to the techniques of $\mathbb{A}^1$-enumeration. The signed count of bitangents to a real quartic has a geometric interpretation (see Section~\ref{subsec-GWbitangent}), and one can thus also use the method of Viro's patchworking to associate the correct sign to a tropical bitangent.

The Appendix contains details of the classification of tropical bitangents. Part \ref{app:classification} spells out the details about all tropical bitangent classes and their dual motifs (see Definition \ref{def-dualmotif}) up to $\mathbb{S}_2$-symmetry. The original classification in \cite{CM20} is up to $\mathbb{S}_3$-symmetry, but we have to break some of the symmetry to fix the line at infinity. Part \ref{app:2H} adds the details to Theorem \ref{thm-2H} (see Theorem \ref{thm-2Hdetails}) by going through the classification and showing which bitangent classes precisely contribute $2\mathbb{H}$ to the $\mathbb{A}^1$-enumeration of bitangents to quartics. Part \ref{app-tableQtypes} finally lists the exceptional tropical bitangent classes and their  $\mathbb{A}^1$-enumerative multiplicity, which is given in terms of the initials of the coefficients of the defining equation of the quartic, as claimed in Theorem \ref{thm:tropicalenumeration}.

\addtocontents{toc}{\protect\setcounter{tocdepth}{1}}
\subsection*{Acknowledgements}
The authors would like to thank Mar\'ia Ang\'elica Cueto, Alheydis Geiger, Andr\'es Jaramillo Puentes,  Danny Krashen, Mario Kummer, Hannah Larson, Marta Panizzut, Sabrina Pauli, Eugenii Shustin, Isabel Vogt, and Kirsten Wickelgren for valuable discussions related to this work and helpful comments on earlier drafts.
The first author acknowledges support by DFG-grant MA 4797/9-1. The second author acknowledges support from 
NSF grants DMS-2001502 and DMS-2053261. 
The third author acknowledges support of the Trond Mohn Foundation project ``Algebraic and topological cycles in complex and tropical geometry" and the Centre of Advanced study Young Fellows project 	``Real structures in discrete, algebraic, symplectic, and tropical geometries". The project was initiated during a stay of the second author in T\"ubingen funded by the Max-Planck Humboldt-Medal 2018. We thank the Max-Planck society and the Humboldt society for support. We thank two anonymous referees for useful comments on an earlier version of this paper.

\medskip 

Data sharing is not applicable to this article as no new data were created or analyzed in this study.

\section{Preliminaries}\label{sec-prelim}
\addtocontents{toc}{\protect\setcounter{tocdepth}{2}}
\subsection{Field conventions}
 Throughout the paper, let $K$ be a Henselian valued field whose value group $\val(K^\times)$ is $2$-divisible, and let $\sigma$ be a section of the valuation, i.e., a group homomorphism $$\sigma\colon \val(K^\times) \to K^\times$$ such that $\val \circ\,\sigma$ is the identity.  Let $R \subset K$ denote the valuation ring $\mathfrak{m} \subset R$ the maximal ideal, and $k = R/\mathfrak{m}$ the residue field.  Throughout, we assume the characteristic of $k$ is not 2.  

\begin{example}
Such fields exist with any given residue field $k$; for instance, one could take the Puiseux series field $k\{\!\{t\}\!\}$, the generalized power series field $k(\!(t^\mathbb{R})\!)$, or the completion of either. The hypotheses are also satisfied by mixed characteristic fields such as $\mathbb{Q}_p (p^{1/2}, p^{1/4}, p^{1/8}, \ldots )$, where $p^{1/2}$ is a square root of $p$, $p^{1/4}$ is a square root of $p^{1/2}$, and so on; a section is given by $\sigma(v) = p^{v}$.
\end{example}

\begin{definition}
We write $ \sigma(v):=t^v$, for $v \in \val(K^\times)$.  
\end{definition}

\noindent This notation intentionally emphasizes the analogy with generalized power series. 
\begin{definition}[Initials]
Let $A \in K^\times$.  Then the initial of $A$, $\ini(A)$, is the image of $A t^{-\val(A)}$ in $k^\times$.
\end{definition}

\noindent In case $K$ is a generalized power series field, $\ini(A)$ is the ``initial coefficient.''

\subsection{Tropicalizations of plane quartics}

Tropicalization can be viewed as degeneration of algebraic varieties defined over $K$. Here, we introduce it only for plane curves.

\begin{definition}[Tropicalization]
Let $C\subset \mathbb{P}^2_K$ be a plane curve and assume that $K$ is algebraically closed. The \emph{tropicalization} is defined by taking $-\val$ componentwise to the torus points:
$$\Trop(C):=\overline{\{(-\val x,-\val y)\;|\; (x,y) \in C\cap (K^\times)^2\}},$$
where the bar denotes the Euclidean closure in $\mathbb{R}^2$.

If $K$ is not algebraically closed, we apply minus valuation coordinatewise to all points of $C$ in the algebraic closure of $K$ to obtain the tropicalization.
\end{definition}
Since we require our field to be Henselian, the valuation extends uniquely to the algebraic closure. Equivalently, $\Trop(C)$ is the image of the Berkovich analytification of $C \cap (\mathbb{G}_m^2)$ under coordinatewise valuation \cite{analytification}.

\begin{remark}
One can also define a compactification of $\mathbb{R}^2$ that is the tropicalization of $\mathbb{P}^2_K$, including its toric boundary, as in Section 3 of \cite{analytification}. The definition above then extends naturally to take boundary points into account. A smooth quartic is never contained in the toric boundary, and hence meets the torus non-trivially. Consequently, this compact extended tropicalization equals the closure of the tropicalization of its points in the torus as defined above (see Lemma 3.1.1 \cite{OP13}). We can therefore restrict attention to the tropicalization in the torus as described above.
\end{remark}

\begin{example}[Tropical line]\label{ex-tropline}
Let $C=V(x+y+t)$. We can parametrize $C$ as $$\{(x,-x-t)\;|\; x\neq -t, x\in K^\times\}.$$ By considering the three cases $\val(x)>1$, $\val(x)<1$, and $\val(x)=1$, where in the last case the situation $\ini(x)=-1$ plays an important role, we can see that the tropicalization consists of three rays starting at the point $(-1,-1)$, one diagonal, one horizontal and one vertical.
\end{example}

\begin{remark}[Tropical dual plane]\label{rem-dualtropicalplane}
By varying the coefficients in Example \ref{ex-tropline}, one can see that the tropicalization of any line whose defining equation has three nonzero coefficients (what we are calling a \emph{tropical line}) consists of three rays centered at a vertex whose coordinates are determined by the valuations of the coefficients. Thus, we can and do identify the \emph{tropical dual plane} parametrizing tropical lines in $\mathbb{R}^2$ with the tropical plane $\mathbb{R}^2$ itself, where a tropical line is identified with its vertex. (Note, however, that this does not extend to an identification of the compact tropicalization of $\mathbb{P}^2$ with that of $(\mathbb{P}^2)^\vee$.)
\end{remark}

\begin{definition}[Newton subdivision]
\label{def-newtonsub}
Let $Q\in K[x,y,z]$ be a homogeneous polynomial of degree $d$, and denote the coefficient of the monomial $x^iy^jz^{d-i-j}$ by $A_{ij}$. Let $D  \subset \mathbb{R}^2\times \mathbb{R}$ be the convex hull of the set $$\{((i,j), -\val(A_{ij})) \ | \ A_{ij}\neq 0 \}.$$ Project the faces of $D$ which can be seen from above down to $\mathbb{R}^2$. The images of these faces form a convex subdivision $\Delta_Q$ of the Newton polygon, called the \emph{Newton subdivision}. 
\end{definition}

\begin{example}\label{example-newtonsubd}
Let $Q=t^{-1}x^2+t^{-2}xy+t^{-1}y^2+t^{-2}x+t^{-2}y+t^{-1}$.
Figure \ref{fig-Newtonsubd} shows the set $D$ on the left and the projection of its upper faces, i.e.\ its   Newton subvision, on the right.
\begin{figure}
%\begin{center}

  \begin{minipage}[c]{0.5\linewidth}
  \centering
  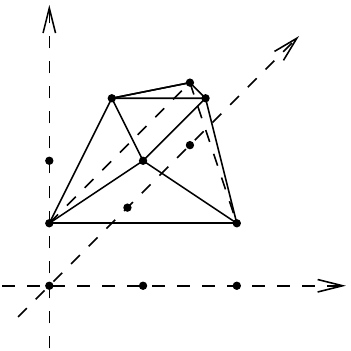            
  \end{minipage}\hfill
  \begin{minipage}[c]{0.5\linewidth}
\centering  
    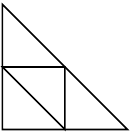
   
  \end{minipage}

\caption{On the left, the set $D$ for the polynomial $Q$ in Example~\ref{example-newtonsubd}, on the right is the Newton subdivision $\Delta_Q$.}\label{fig-Newtonsubd}
\end{figure}

\end{example}
\begin{theorem}[Duality theorem]\label{theorem-dualnewtonsub}
For $C=V(Q)$, the tropical curve $\Trop(C)$ is the $1$-skeleton of a subdivision of $\mathbb{R}^2$ that is dual to the Newton subdivision $\Delta_Q$, i.e., there is a natural inclusion reversing bijection of faces between these two subdivisions.

\end{theorem}
 \noindent A proof can be found, for example, in \cite{Mi03}, Proposition 3.11.
This duality is illustrated in Figure \ref{fig-exdual}.
\begin{figure}

\begin{center}
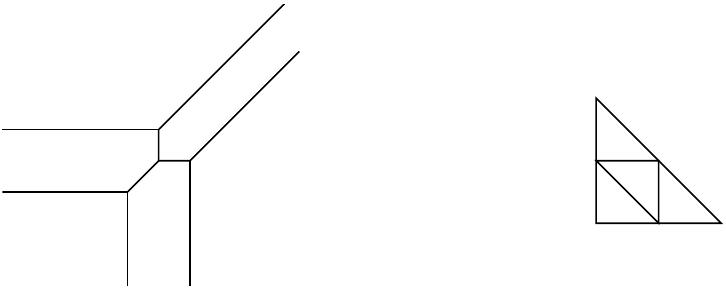
\end{center}
\caption{The tropicalization of the curve from Example~\ref{example-newtonsubd} and the Newton subdivision $\Delta_Q$.}\label{fig-exdual}
\end{figure}

The tropicalization of a plane curve is called a \emph{tropical plane curve}. From duality to a Newton subdivision, we conclude that tropical curves satisfy the \emph{balancing condition}: if we weight each edge by the lattice length of its dual edge, then the weighted sum of the primitive edge generators around a vertex is $0$. 
 
A tropical plane curve is \emph{smooth} of degree $d$ if it is dual to a unimodular triangulation of the triangle with vertices $(0,0), (0,d), (d,0)$. Here, unimodular means that it is subdivided into $d^2$ triangles each having area $\frac{1}{2}$.

\subsection{Bitangents to tropicalized quartics} \label{sec:prelimsbitangents}

Suppose that two tropical plane curves $\Gamma_1,\Gamma_2$ intersect transversally, i.e.\ the intersection is a finite set of points each of which is contained in the interior of an edge of both $\Gamma_1$ and $\Gamma_2$. Let $P\in \Gamma_1\cap \Gamma_2$. Choose weighted direction vectors $u_1,u_2$ for edges of $\Gamma_1,\Gamma_2$ emanating from $P$. Then the \emph{intersection multiplicity} of the two curves at $P$ is $\Gamma_1\cdot\Gamma_2\mid_P = |\det(u_1,u_2)|$, and $\Gamma_1\cdot \Gamma_2=\sum_{P\in \Gamma_1\cap \Gamma_2}\Gamma_1\cdot\Gamma_2\mid_P$. The balancing condition ensures that this definition does not depend on the choice of direction vectors. If the curves do not intersect transversally, we  use the \emph{stable} intersection: choose a direction $v$ so that $\Gamma_1$ and $\Gamma_2 + \epsilon\cdot v$ intersect transversally whenever $\epsilon$ is small enough. Then the stable intersection is
$$
\lim_{\epsilon\to 0} \Gamma_1\cdot(\Gamma_2+\epsilon\cdot v).
$$

When the intersection is not transverse, the number of intersection points, counted with multiplicity, in the preimage of a connected component of the  intersection of the tropicalizations is the sum of the multiplicities of  the points of the stable tropical intersection that lie in this component, see \cite{OR13}, Theorem~6.4. 

\begin{definition}[Bitangent]\label{def-tropbitangent}
We say that a tropical line $\Lambda$ is \emph{bitangent} to a tropicalized quartic $\Trop(C)$ in $\mathbb{R}^2$ if $\Lambda$ and $\Trop(C)$ intersect in either 
\begin{itemize}
\item one connected component with total intersection multiplicity $4$, or
\item two connected components each with total intersection multiplicity $2$.
\end{itemize}
\end{definition}
In the first case, the intersection can be a point, or a segment of an edge of $\Gamma$ containing a vertex, or three edges adjacent to the same vertex (in which case --- if liftable, see Definition \ref{def-lift} --- the two tangency points tropicalize to the midpoints of the two segments we obtain by subtracting the shortest edge length from the others; see Figure \ref{fig-starshaped}).

In the second case of Definition \ref{def-tropbitangent}, the intersection can be a point, or a segment of an edge of $\Gamma$ (in which case --- if liftable, see Definition \ref{def-lift} ---  the tangency point tropicalizes to the midpoint of the segment).

For a bitangent $\Lambda$ to $\Trop(C)$, we call a connected component of the intersection $\Lambda\cap \Trop(C)$ a \emph{tropical tangency component}. 

\begin{figure}
\begin{center}

\tikzset{every picture/.style={line width=0.75pt}} %set default line width to 0.75pt        

\begin{tikzpicture}[x=0.75pt,y=0.75pt,yscale=-1,xscale=1]
%uncomment if require: \path (0,784); %set diagram left start at 0, and has height of 784

%Straight Lines [id:da7613707862423045] 
\draw    (140,80) -- (190,80) ;
%Straight Lines [id:da6546026692841378] 
\draw    (190,80) -- (240,30) ;
%Straight Lines [id:da4092533909737609] 
\draw    (190,90) -- (190,80) ;
%Straight Lines [id:da8840151820901058] 
\draw    (140,80) -- (130,70) ;
%Straight Lines [id:da8131333315995171] 
\draw    (250,30) -- (240,30) ;
%Straight Lines [id:da2749313547263582] 
\draw    (140,90) -- (140,80) ;
%Straight Lines [id:da38699658879689824] 
\draw    (240,30) -- (240,20) ;
%Straight Lines [id:da6078432363654829] 
\draw    (190,90) -- (180,90) ;
%Straight Lines [id:da4565037988334588] 
\draw    (200,100) -- (190,90) ;
%Shape: Circle [id:dp6182703997731491] 
\draw  [color={rgb, 255:red, 74; green, 144; blue, 226 }  ,draw opacity=1 ][fill={rgb, 255:red, 74; green, 144; blue, 226 }  ,fill opacity=1 ] (168.26,80.12) .. controls (168.26,79.31) and (168.92,78.65) .. (169.73,78.65) .. controls (170.55,78.65) and (171.2,79.31) .. (171.2,80.13) .. controls (171.2,80.94) and (170.54,81.6) .. (169.73,81.6) .. controls (168.91,81.6) and (168.25,80.93) .. (168.26,80.12) -- cycle ;
%Shape: Circle [id:dp33512966267331834] 
\draw  [color={rgb, 255:red, 74; green, 144; blue, 226 }  ,draw opacity=1 ][fill={rgb, 255:red, 74; green, 144; blue, 226 }  ,fill opacity=1 ] (207.75,60.27) .. controls (207.75,59.46) and (208.41,58.8) .. (209.23,58.8) .. controls (210.04,58.81) and (210.7,59.47) .. (210.7,60.28) .. controls (210.7,61.09) and (210.04,61.75) .. (209.22,61.75) .. controls (208.41,61.75) and (207.75,61.09) .. (207.75,60.27) -- cycle ;

\end{tikzpicture}

\end{center}

\caption{A local picture of a tropicalized quartic $\Trop(C)$ together with 
the tropicalization of the two tangency points of a bitangent $\ell$ to $C$. Here the line $\ell$ is a lift of a tropical line $\Lambda$ whose vertex coincides with the center vertex of $\Trop(C)$. The intersection of $\Lambda $ and $\Trop(C)$ is a single connected component with tropical intersection multiplicity $4$.  The tropical line does not appear in the figure. }
\label{fig-starshaped}
\end{figure}
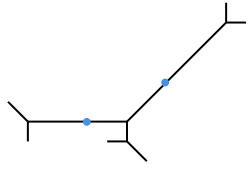

It is often the case that a tropicalized quartic $\Trop(C)$ in $\mathbb{R}^2$ admits infinitely many tropical bitangents. 
A natural question to ask is: which of the infinitely many tropical bitangent lines are tropicalizations of bitangent lines of $C$?

\begin{definition}[Lift of a bitangent]\label{def-lift}
Let $C$ be a quartic defined over $K$. Let $\Lambda$ be a tropical line which is tropically bitangent to $\Trop(C)$. A bitangent line $L$ of $C$ is called a \emph{lift} of $\Lambda$ if $\Trop(L)=\Lambda$. If such a lift exists, we say that $\Lambda$ is \emph{liftable}. If in addition the lift is defined over $K$, we say that $\Lambda$ is \emph{liftable over $K$}.
\end{definition}

We declare that two tropical bitangent lines are equivalent if we can move one tropical line to the other (connecting their corresponding points in the tropical dual plane) while maintaining bitangency, i.e., the equivalence classes are the connected components of the locus of tropical bitangents in the tropical dual plane. For tropicalized quartics, this is equivalent to saying that the tropical bitangent lines correspond to the same theta characteristic in the tropical Jacobian \cite[Definition 3.8]{BLMPR}. 

When studying tropical bitangents and their lifts, we restrict to \emph{generic} tropicalized quartics (see Remark \ref{rem-gen} and \cite{LM17}, 3.3). This ensures that the tropicalized quartic is smooth, that the tropical tangency components are not contained on the same ray, and that when the intersection $\Trop(C) \cap \Lambda$ consists of a vertex and three ray segments for a liftable tropical bitangent $\Lambda$, then the tropicalization of tangency points are not on the vertex. 

\begin{definition}[Bitangent class]
A connected component of the set of tropical bitangents lines to a tropicalized quartic $\Trop(C)$ in the tropical dual plane (see Remark \ref{rem-dualtropicalplane}) is called a \emph{bitangent class} of $\Trop(C)$. 
\end{definition}
By \cite{BLMPR}, a smooth tropicalized quartic has precisely $7$ bitangent classes. The statement can be generalized for non-smooth tropicalized quartics, in which case bitangent classes must be counted with a suitable multiplicity \cite{LL}.

By \cite{CM20}, a bitangent class is a polyhedral complex. First, we equip it with the coarsest polyhedral complex structure that is needed (for all the shapes in the classification, such a coarsest polyhedral complex structure exists). By duality (see Theorem \ref{theorem-dualnewtonsub}), $\mathbb{R}^2$ obtains a polyhedral complex structure with cells the vertices and edges of $\Trop(C)$, and the connected components of $\mathbb{R}^2\setminus \Trop(C)$. We use this to refine the polyhedral complex structure of a bitangent class:
\begin{definition}[Shape of a bitangent class]\label{def-shape}
The \emph{shape} of a bitangent class is the polyhedral complex structure obtained by refining the bitangent class with the polyhedral complex structure of $\mathbb{R}^2$ given by $\Trop(C)$ (where the containment of a polyhedral cell in $\Trop(C)$ is encoded using colors in Figure \ref{fig:2dCells}). We identify the tropical dual plane with the original $\mathbb{R}^2$ to make this refinement.
\end{definition}
Shapes of bitangent classes are classified up to $\mathbb{S}_3$-symmetry in \cite{CM20}; see especially Figure~6 in \emph{loc.\ cit.} There are 41 such shapes, denoted with letters (A), (B), $\ldots$ or double letters (EE) etc. For convenience of the reader, we reproduce this figure here; see Figure \ref{fig:2dCells}.
\begin{figure}[htb]
\includegraphics[scale=0.3]{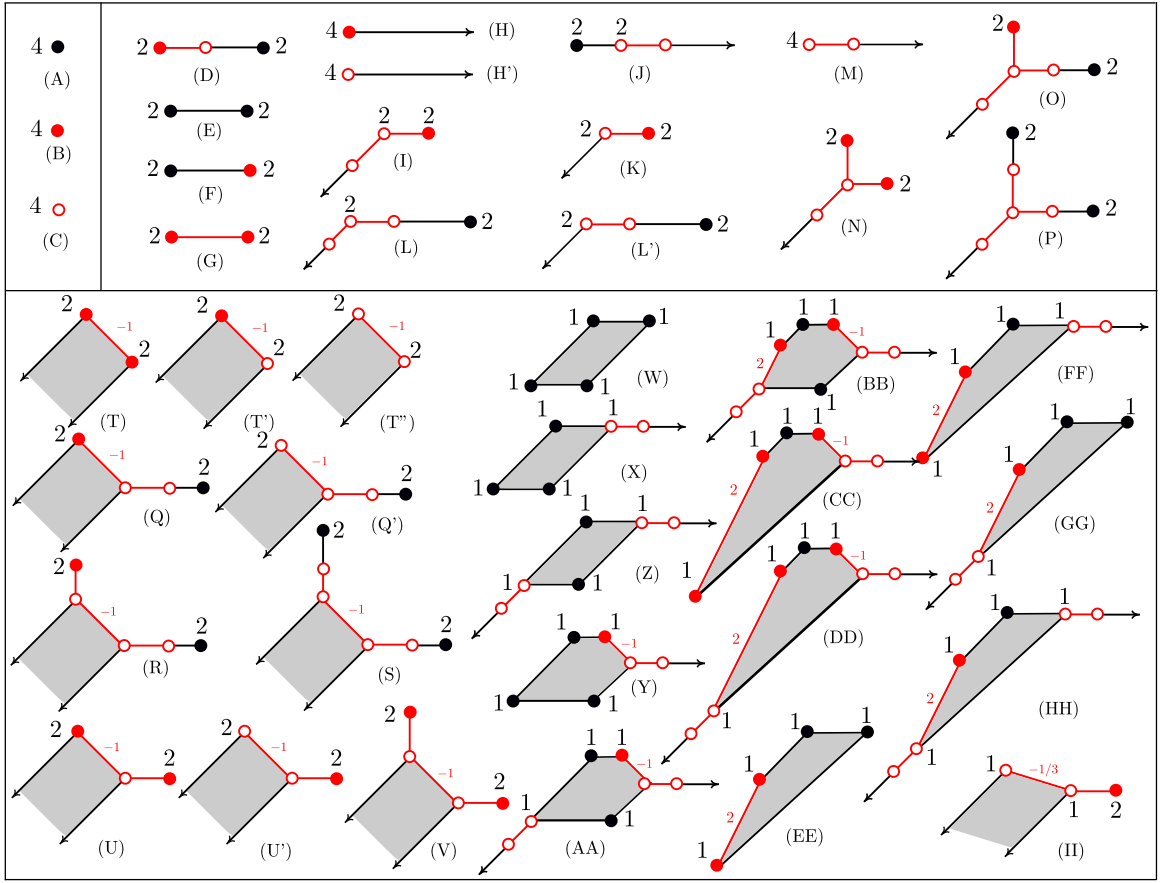}
\caption{Taken from \cite{CM20}: Orbit representatives of all 41 shapes of bitangent classes to $\Gamma$, grouped by the dimension of its maximal cells. The numbers above each vertex indicate lifting multiplicities over the complex numbers, whereas the red ones above edges indicate slopes. The black cells of each bitangent class miss $\Gamma$, whereas the red ones lie on it. The unfilled dots are vertices of $\Gamma$.\label{fig:2dCells}}
\end{figure}

\begin{definition}[Dual motif, \cite{GP21}]\label{def-dualmotif}
Let $\Trop(C)$ be a smooth tropicalized quartic and $B$ one of its seven bitangent classes. Let $\Lambda$ be a tropical bitangent in $B$. The \emph{dual motif} of the bitangent class $B$ is the set of all triangles and edges in the Newton subdivision whose dual vertices resp.\ edges intersect a bitangent in $B$.

 \end{definition}
For an example, see Figure \ref{fig-tropbitangentclass}. All tropical bitangents in the bitangent class considered there share the same tropical tangency component on the lower right, the intersection of the vertical segment. Dual to this edge, with its two end vertices, are the two triangles on the right in the dual Newton subdivision. The other tropical tangency component varies for the tropical bitangents in $B$. It is either on an edge, or an end vertex of the same edge. Accordingly, we obtain another pair of triangles on the left in the dual Newton subdivision for the dual motif. If, as in Figure \ref{fig-tropbitangentclass}, all tropical bitangents in a bitangent class have two disjoint tropical tangency components, we can divide the dual motif into two parts, one for each tropical tangency component.

\begin{example}\label{ex-bitangentclass}
Figure \ref{fig-tropbitangent} shows a tropicalized quartic and a tropical bitangent line. The tropical line intersects the tropicalized quartic in two tropical tangency components. For the right tropical tangency component, we have to use stable intersection to check that it is a tangency. The tropicalization of a tangency point for any lift is the midpoint of the segment of intersection. We can see that we can move the vertex of the tropical bitangent upwards or downwards, until we hit vertices of the tropicalized quartic, maintaining the bitangency. The bitangent class is therefore a line segment as depicted in Figure \ref{fig-tropbitangentclass}. As it is disjoint from $\Trop(C)$, its shape is also just a segment. This is a shape of type (E) in the classification of \cite{CM20}. In \cite{LM17}, it is shown that, if $C$ is any algebraic quartic with this tropicalization, exactly $2$ of the $28$ bitangent lines to $C$ tropicalize to the tropical line with vertex the upper red point, exactly $2$ to the one with vertex the lower red point, and none to a point in the interior of the red segment (see Theorem \ref{thm-4lifts}). Figure \ref{fig-tropbitangentclass} also depicts the dual motifs for the two tropical tangency components. Circled in red is the vertex dual to the connected component of $\mathbb{R}^2$ in which the bitangent class is contained. 

\begin{figure}
\begin{center}
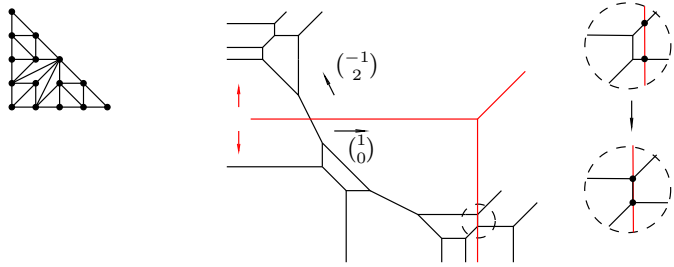
\end{center}
\caption{A tropicalized quartic, with its dual Newton subdivision, and a tropical bitangent line. 
We can move the vertex of the tropical bitangent upwards or downwards maintaining the bitangency.}\label{fig-tropbitangent}
\end{figure}

\begin{figure}
\begin{center}
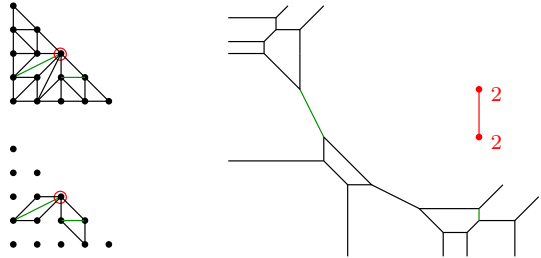
\end{center}
\caption{The bitangent class and dual motif of the tropical bitangent to the tropicalized quartic depicted in Figure \ref{fig-tropbitangent}.}\label{fig-tropbitangentclass}
\end{figure}

\end{example}

\begin{remark}[Genericity of tropicalized quartics]\label{rem-gen}
Compared to \cite{LM17}, we require an additional genericity assumption on our tropicalized quartics: we subdivide the cone in the secondary fan corresponding to the unimodular triangulation according to the types of tropical bitangent classes that can occur and require our tropicalized quartic to correspond to a point in the interior of a cone in this subdivision. This subdivision of the secondary fan is computed in \cite{GP21a}. Put differently, we require the edge lengths of $\Trop(C)$ to be generic in the sense that no unexpected alignment of vertices happens, see Figure \ref{fig-align}. This picture shows a local piece of a tropicalized quartic on the left. For a generic element in the corresponding cone of the secondary fan, we expect the two lengths $l_1$ and $l_2$ to be different. If these lengths are equal, as depicted on the right, the two lower vertices align.

We call a bitangent shape that occurs for such a generic tropicalized quartic a \emph{generic bitangent shape}.
\end{remark}

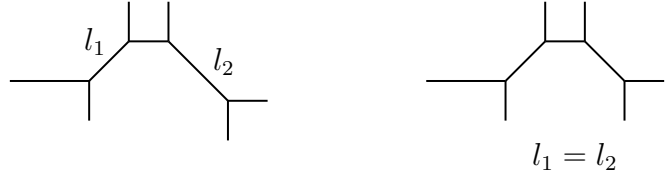
\begin{figure}
\begin{center}

\tikzset{every picture/.style={line width=0.75pt}} %set default line width to 0.75pt        

\begin{tikzpicture}[x=0.75pt,y=0.75pt,yscale=-1,xscale=1]
%uncomment if require: \path (0,784); %set diagram left start at 0, and has height of 784

%Straight Lines [id:da506090417968714] 
\draw    (370,150) -- (410,150) ;
%Straight Lines [id:da4202246168264483] 
\draw    (410,150) -- (410,170) ;
%Straight Lines [id:da9530623863285599] 
\draw    (430,130) -- (410,150) ;
%Straight Lines [id:da01725781675789051] 
\draw    (450,130) -- (470,150) ;
%Straight Lines [id:da5780247747559762] 
\draw    (430,130) -- (450,130) ;
%Straight Lines [id:da36033464743732113] 
\draw    (430,110) -- (430,130) ;
%Straight Lines [id:da46674703714600474] 
\draw    (450,110) -- (450,130) ;
%Straight Lines [id:da30935181008926216] 
\draw    (470,150) -- (470,170) ;
%Straight Lines [id:da07280142489538932] 
\draw    (490,150) -- (470,150) ;
%Straight Lines [id:da7067088050202952] 
\draw    (160,150) -- (200,150) ;
%Straight Lines [id:da47505257706048987] 
\draw    (200,150) -- (200,170) ;
%Straight Lines [id:da3201415587580497] 
\draw    (220,130) -- (200,150) ;
%Straight Lines [id:da6356673425942266] 
\draw    (240,130) -- (270,160) ;
%Straight Lines [id:da27104854437954407] 
\draw    (220,130) -- (240,130) ;
%Straight Lines [id:da04586812623978043] 
\draw    (220,110) -- (220,130) ;
%Straight Lines [id:da7625765372067596] 
\draw    (240,110) -- (240,130) ;
%Straight Lines [id:da44579550345804064] 
\draw    (270,160) -- (270,180) ;
%Straight Lines [id:da08152704561915847] 
\draw    (290,160) -- (270,160) ;

% Text Node
\draw (196,122) node [anchor=north west][inner sep=0.75pt]   [align=left] {$\displaystyle l_{1}$};
% Text Node
\draw (261,132) node [anchor=north west][inner sep=0.75pt]   [align=left] {$\displaystyle l_{2}$};
% Text Node
\draw (422,180) node [anchor=north west][inner sep=0.75pt]   [align=left] {$\displaystyle l_{1} =l_{2}$};

\end{tikzpicture}

\end{center}
\caption{Local pictures of tropicalized quartics. An unexpected alignment of two vertices is depicted on the right. For a generic quartic, as shown on the left, the two lengths $l_1$ and $l_2$ are different.}\label{fig-align}
\end{figure}

The following statements concern lifts of tropical bitangent classes over $\mathbb{C}\{\!\{t\}\!\}$ and $\mathbb{R}\{\!\{t\}\!\}$. 

\begin{theorem}[\cite{Chan2, LM17}]\label{thm-4lifts}
Each bitangent class of a generic tropicalized quartic has $4$ lifts over $\mathbb{C}\{\!\{t\}\!\}$. More precisely, there are either $4$ tropical lines in the bitangent class which each lifts once, or $2$ which each lift twice, or $3$ of which one lifts twice, or $1$ which lifts four times.
\end{theorem}

\begin{theorem}[\cite{CM20}]\label{thm-reallifts}
Each bitangent class of a generic tropicalized quartic has either $0$ or $4$ lifts over $\mathbb{R}\{\!\{t\}\!\}$, i.e.\ the possible obstruction for real lifting is the same for all representatives of a given tropical bitangent class that have a complex lift.
\end{theorem}

These results are based on the classification of shapes of bitangent classes in \cite{CM20} up to $\mathbb{S}_3$-symmetry. When considering $\GW$-multiplicities as in Section \ref{sec-GW}, we fix the line $\{z=0\}$ and do not have $\mathbb{S}_3$-symmetry for that reason, only $\mathbb{S}_2$-symmetry for exchanging the variables $x$ and $y$. 

In Appendix \ref{app:classification}, we classify tuples of shapes of generic bitangent classes together with their dual motifs up to $\mathbb{S}_2$-symmetry exchanging $x$ and $y$. 
The classification is built on the classification in \cite{CM20}. It restricts the cases studied there to generic shapes, 
considers $\mathbb{S}_3/\mathbb{S}_2$-orbits, and pairs up with dual motifs. In Appendix \ref{app:classification}, we depict all dual motifs together with local pictures of the tropicalized tangency points of the liftable tropical bitangents in the bitangent class are given.

\begin{example} \label{ex:possiblerunningexample}
Consider the plane quartic $C=V(Q)$ for
\begin{align*}
Q(x,y)  = & t^{36} x^4 + t^{18}x^3 y + t^2x^2y^2 + t^{18}xy^3 + t^{36} y^4 + t^{23} x^3 + t^6 x^2y +t^6xy^2  \\  &  + t^{23} y^3 + t^{12}x^2 + xy + t^{12}y^2 + t^2x+t^2y + 1
\end{align*}
over the field of Puiseux series $\mathbb{C}\{\!\{t\}\!\}$.

Its tropicalization $\Trop(C)$ together with the seven bitangent classes (in red) is depicted in Figure \ref{fig-7bitangents}. The liftable tropical bitangents are depicted as red dots. If a bitangent class has four red dots, each lifts once, if it has two red dots, each lifts twice and a single dot lifts to four bitangents.

\begin{figure}
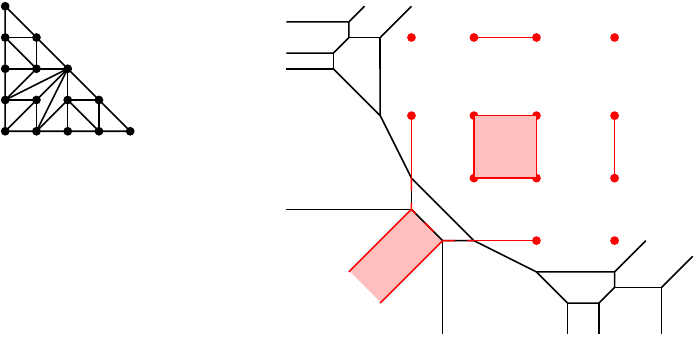
\caption{A tropicalized quartic with its seven bitangent classes.}\label{fig-7bitangents}
\end{figure}

\end{example}

\section{Lifting tropical bitangents and tropical points of tangencies}\label{sec-Lifting}

\subsection{Local lifting equations and their solutions}\label{subsec-locallift}
We review the local lifting techniques for tropical bitangent lines to tropicalized plane quartics from \cite{CM20, LM17}.

As before, we let $C=V(Q)$ be a quartic and denote the coefficients of the defining polynomial $Q$ by $A_{ij}$.
For a point $p\in \Trop(C)$, let $\ini_p(Q)$ denote the sum of the (dehomogenized) terms $\ini(A_{ij})x^iy^j\in k[x,y]$ for which the maximum $\max\{-\val(A_{ij})+ip_x+jp_y\}$ is attained. We call it the \emph{initial form} of $Q$ at $p$. The initial form can be be read off from the dual Newton subdivision: if $p$ is contained in the interior of a region of $\mathbb{R}^2\setminus C$, then the initial form is given precisely by the term of $Q$ corresponding to this region. If $p$ is in the interior of an edge, we use all terms that correspond to points of the dual edge. If $p$ is a vertex, we use all terms that correspond to points of the dual polygon.

\begin{definition}[Local equations for tropicalized  tangency points]\label{def:localpoly}
Let $T$ be a connected component of the intersection of $\Trop(C)$ with a tropical bitangent $\Lambda$, and assume a lift of $\Lambda$ is given by an equation of the form $\ell=y+M+Nx$, where $M$ and $N$ are unknowns. Let $A(T)=\bigcup_{p\in T}\supp \{\ini_p Q\}$ be the union of the supports of the initial forms of $Q$ for all $p\in T$. We let $Q_T$ be the restriction of $Q$ to the terms appearing in $A(T)$, and call it the \emph{local equation for $C$ at $T$}. 

Analogously, we let $B(T)$ be the union of the supports of the initial forms of $\ell$ for all $p\in T$, and $\ell_T$, the restriction of $\ell$ to $B(T)$, the local equation of $\ell$.

We let $W_T=\frac{\partial Q_T}{\partial x}\cdot \frac{\partial \ell_T}{\partial y} - \frac{\partial Q_T}{\partial y}\cdot \frac{\partial \ell_T}{\partial x}$ be the local version of the Wronskian.
\end{definition}
Notice that vanishing of the Wronskian $W= \frac{\partial Q}{\partial x}\cdot \frac{\partial \ell}{\partial y} - \frac{\partial Q}{\partial y}\cdot \frac{\partial \ell}{\partial x}$ implies that the gradients of $V(Q)$ and $V(\ell)$ coincide, i.e.\ that the line $V(\ell)$ is tangent to $V(Q)$ at $(x,y)$.

If $p\in T$ is the tropicalization of a tangency point, we solve the \emph{local lifting equations}
\begin{align*}
& Q_T(x,y)=0,\;\; \ell_T(x,y)=0, \;\;W_T(x,y)=0.
\end{align*}
If the tropical tangency component is just a point, $T=\{p\}$, then this amounts to solving the equations of the initial forms of $Q$, $\ell$ and $W$ at $p$. If $T$ is a segment of intersection involving two vertices of an edge $E$ of $\Trop(C)$, then $Q_T$ contains not only the terms of the edge $E^\vee$ dual to $E$, but also the terms corresponding to the vertices which form triangles with $E^\vee$ in the dual subdivision. If $T$ is a segment containing one vertex $V$ of $\Trop(C)$ and the vertex of $\Lambda$, then $\ell_T=\ell$ and $Q_T$ contains the terms corresponding to $E^\vee$ and the vertex of the triangle dual to $V$. If $T$ consists of three segments joined at a vertex $V$, $Q_T$ contains the terms of the triangle dual to $V$ plus the terms of its neighbouring triangles, and $\ell_T=\ell$.

If $T$ is more than just a point, the equations above are hard to solve. One uses a technical trick which is called a tropical modification or tropical refinement to produce solutions up to the order which appears as biggest valuation of the terms involved. For a discussion on how to treat the case of a segment $T$ containing two vertices of $\Trop(C)$, see \cite{LM17} Proposition 3.7 and 3.9. For the case of a segment containing the vertex of $\Lambda$, see Lemma 5.2, Case (3a) \cite{CM20} and Proposition 3.12 in \cite{LM17}. For the case of three segments joined at a vertex, see Proposition 3.12 in \cite{LM17}. For general background on the technique of tropical modification, see e.g.\ \cite{BLdM11, IMS09, Mi06}.
Once we solved in such a way for $x,y$ up to terms of higher valuation, it follows from the Henselian property (see also \cite{Eisenbud}, Exercises 7.25, 7.26  and \cite{IMS09}, Chapter 2) that these solutions can uniquely be completed to an element in $K$. For more details, see Section 2.3 in \cite{LM17}.
If the initials are in the residue field $k$, then the complete solutions are in $K$.

\begin{example}\label{ex-locallifting}
We compute the initials of the four lifts of the bitangent class from Example \ref{ex-bitangentclass}.
Consider first the upper vertex. There are two lifts over the algebraic closure which tropicalize to it. We denote the equations of those two lifts by $M_1+N_1x+y=0$ resp.\ $M_2+N_2x+y=0$. 
Let $p_1$ be the left tropical tangency component. In $\Trop(C)$, it is at the vertex dual to $A_{01}y+A_{12}xy^2+A_{22}x^2y^2$. In the tropical bitangent line, it is in the horizontal ray, i.e.\ dual to $M_i+y$.
We denote the initials of the $A_{ij}$ by $a_{ij}$ and analogously for the coefficients of the line equations. Then the initials $(x_i,y_i)$ of the lift of the left tropical tangency component satisfy
\begin{align*}
&a_{01}y_{i1}+a_{12}x_{i1}y_{i1}^2+a_{22}x_{i1}^2y_{i1}^2=0,\\&
m_i+y_{i1}=0,\\&
2\cdot a_{22}\cdot x_{i1}\cdot y_{i1}^2+a_{12}\cdot y_{i1}^2=0,
\end{align*}
for $i=1,2$.
Solving for $m_i$, $x_{i1}$ and $y_{i1}$ (for example by computing a Gr\"obner basis of the ideal defined by the three equations using a computer algebra system such as \textsc{Singular} \cite{DGPS} or \textsc{Oscar} \cite{Oscar}) we obtain, for $i=1,2$,
\begin{align*}
m_i=-\frac{4\cdot a_{01}\cdot a_{22}}{a_{12}^2},\;\;  x_{i1}=-\frac{a_{12}}{2\cdot a_{22}},\;\;  y_{i1}= \frac{4\cdot a_{01}\cdot a_{22}}{a_{12}^2}.
\end{align*}

The second tropical tangency component is the same for all liftable tropical bitangents in this bitangent class (i.e.\ for the tropical line with vertex the upper vertex of the red edge, and for the tropical line with vertex the lower vertex of the red edge in Figure \ref{fig-tropbitangentclass}). Therefore, we now first study the two lifts tropicalizing to the lower vertex and their tropical tangency component on the left. It is at the vertex dual to $A_{01}y+A_{11}xy+A_{22}x^2y^2$. The three equations to solve are thus:
\begin{align*}
&a_{01}y_{i1}+a_{11}x_{i1}y_{i1}+a_{22}x_{i1}^2y_{i1}^2=0,\\&
m_i+y_{i1}=0,\\&
2\cdot a_{22}\cdot x_{i1}\cdot y_{i1}^2+a_{11}\cdot y_{i1}=0,
\end{align*}
for $i=3,4$, and the solutions are
\begin{align*}
&m_i=-\frac{a_{11}^2}{4\cdot a_{01}\cdot a_{22}},\;\; x_{i1}=-\frac{2\cdot a_{01}}{ a_{11}},\;\;y_{i1}=\frac{a_{11}^2}{4\cdot a_{01}\cdot a_{22}} .
\end{align*}

Let $p_2$ be the tangency point in the right tropical tangency component. This tropical tangency component appears not only for the two lifts of the upper vertex of the bitangent class, but also for the lower ones. That is, we can use the corresponding local lifting equations to solve for one tropical tangency component for each of the four lifts.
The point $p_2$ is
contained in a connected component of the intersection $\Trop(C)$ and a tropical bitangent which consists of a bounded edge of $\Trop(C)$. 
Dual to the vertices of this bounded edge are two triangles of the dual subdivision which meet along the edge  joining $(2,1)$ and $(3,1)$. 
Therefore, 
we must take as local equation $$A_{21}x^2y+A_{31}x^3y + A_{30}x^3 + A_{22}x^2y^2.$$ 
In the tropical line, the tropical tangency component  $p_2$ is contained in the vertical ray, thus the local equation to consider is $m_i+n_ix$ for $i=1,\ldots,4$. As we saw earlier, $m_i$ is determined by the left (upper or lower) tropical tangency component. In the local equations, we now solve for the ratio $\frac{m_i}{n_i}$, which is then sufficient to compute all coefficients of the lifts. It turns out that there are two solutions whose initials coincide. In order to be able to differentiate the two solutions whose initials coincide, we also list some contributions of higher valuation.

The solutions we obtain are 
\begin{align*}
&\frac{M_i}{N_i}=\frac{A_{21}}{A_{31}}\pm \frac{2A_{22}}{A_{31}} \sqrt{-\frac{A_{30}A_{21}}{A_{22}A_{31}}}+ \ldots\\ & X_{i2}=- \frac{A_{21}}{A_{31}} \mp \frac{2A_{22}}{A_{31}}\sqrt{-\frac{A_{30}A_{21}}{A_{22}A_{31}}}+ \ldots\\ &Y_{i2}=\pm \sqrt{-\frac{A_{30}A_{21}}{A_{22}A_{31}}}+ \ldots
\end{align*}
for $i=1,\ldots,4$. Here, we write the solutions not in terms of the initial $a_{ij}$ but in terms of the whole coefficients $A_{ij}$, to keep track of the contributions of higher valuation.

Combining the results of the local lifting equations, we can now list the initials of the coefficient of the bitangent equation and the tangency point for all four lifts:

\hspace{0.2cm}

\begin{tabular}{|l||c|c|c|c|c|c|}
\hline
$i$ &$m_i$&$n_i$&$x_{i1}$&$y_{i1}$&$x_{i2}$&$y_{i2}$\\ \hline \hline
 1&$-\frac{4\cdot a_{01}\cdot a_{22}}{a_{12}^2}$ &$-\frac{4\cdot a_{01}\cdot a_{22}\cdot a_{31}}{a_{12}^2\cdot a_{21}}$ &$-\frac{a_{12}}{2\cdot a_{22}}$&$\frac{4\cdot a_{01}\cdot a_{22}}{a_{12}^2}$ &$- \frac{a_{21}}{a_{31}}$ &  $\sqrt{-\frac{a_{30}a_{21}}{a_{22}a_{31}}} $\\ \hline
 2&$ -\frac{4\cdot a_{01}\cdot a_{22}}{a_{12}^2}$& $-\frac{4\cdot a_{01}\cdot a_{22}\cdot a_{31}}{a_{12}^2\cdot a_{21}}$&$-\frac{a_{12}}{2\cdot a_{22}}$&$\frac{4\cdot a_{01}\cdot a_{22}}{a_{12}^2}$ & $- \frac{a_{21}}{a_{31}}$& $- \sqrt{-\frac{a_{30}a_{21}}{a_{22}a_{31}}}$ \\ \hline
 3&$-\frac{a_{11}^2}{4\cdot a_{01}\cdot a_{22}}$ &$-\frac{a_{11}^2\cdot a_{31}}{4\cdot a_{01}\cdot a_{22}\cdot a_{21}}$ &$-\frac{2\cdot a_{01}}{ a_{11}}$ & $\frac{a_{11}^2}{4\cdot a_{01}\cdot a_{22}}$& $- \frac{a_{21}}{a_{31}}$& $\sqrt{-\frac{a_{30}a_{21}}{a_{22}a_{31}}} $ \\ \hline
 4& $-\frac{a_{11}^2}{4\cdot a_{01}\cdot a_{22}}$& $-\frac{a_{11}^2\cdot a_{31}}{4\cdot a_{01}\cdot a_{22}\cdot a_{21}}$&$-\frac{2\cdot a_{01}}{ a_{11}}$ &$\frac{a_{11}^2}{4\cdot a_{01}\cdot a_{22}}$ &$- \frac{a_{21}}{a_{31}}$ & $- \sqrt{-\frac{a_{30}a_{21}}{a_{22}a_{31}}}$ \\ 
\hline
\end{tabular}

\end{example}

As in the example above, these local lifting equations can be solved for any case appearing in the combinatorial classification in Appendix \ref{app:classification}, and the solutions are always given by Laurent terms in the initials $a_{ij}$ of the coefficients of $Q$, or by square roots thereof.

The consequence for lifting can be summed up as follows:

\begin{theorem}\label{thm:twistlift}
 Let $C = V(Q)$ be a quartic curve defined over $K$ with $\Trop(C)$ a generic tropicalized quartic and suppose $\Lambda$ is a liftable tropical bitangent to $C$. Then whether or not $\Lambda$ lifts over $K$ is determined by $\Trop(C)$  and the equivalence classes of  initials of the coefficients of $Q$ in $k^\times/(k^\times)^2$. 

\end{theorem}

The proof follows along the same lines as in the complex and real case. we go through the classification in Appendix \ref{app:classification} and solve the local lifting equations, no over the ground field $K$.

\subsection{Lifting conditions and twisted edges}\label{subsec-twisting}

It follows from the classification of bitangent shapes in \cite{CM20} and their lifting multiplicities that some bitangent shapes lift over any field (this holds for tropical bitangent lines with lifting multiplicity one, the rough argument being that a uniquely solvable system of linear equations over a field $k$ also has its solution over $k$). Obstructions to lifting arise due to higher lifting multiplicities. 
In the following, we study tropical tangency components which produce a factor leading to such higher lifting multiplicities. 
Roughly, such factors arise whenever a tropical tangency component is locally fixed, i.e.\ we cannot vary the local part of the tropical bitangent without destroying the tangency. This is the case for segments of intersections, or for intersections involving the vertex of the tropical bitangent line.

In the case of a tropical tangency component on a segment of intersection, the potential obstruction for lifting can be phrased in terms of twisted edges or  relatively twisted edges. 
By going through the cases in Appendix \ref{app:classification}, it turns out that the cases for lifting we study here cover all bitangent shapes except (C), for which lifting is more involved. For our purpose of understanding arithmetic multiplicities in Section \ref{sec-GW}, precise lifting conditions for (C) are not needed and we therefore do not consider these details here.

We start by stating the lifting solutions, which can be determined with the methods described in Section~\ref{subsec-locallift}.
In the following, we denote the vertex of the dual subdivision corresponding to the term $A_{ij}x^iy^j$ in (the dehomogenized version of) $Q$ by $V_{ij}$.

\begin{lemma}\label{lem-localliftsegment} Assume the tropical line $\Lambda$ intersects the tropicalized quartic $\Trop(C)$ in two connected components of which one is a segment of a horizontal edge dual to an edge connecting $V_{1l}$ and $V_{1l+1}$. This edge forms triangles with two points $V_{0i}$ and $V_{j2}$.
Then there are two lifts for the coefficient $M$ of a line $V(M+Nx+y)$ and for the tangency points $(x,y)$ tropicalizing to this segment. 

If the vertex of $\Lambda$ is not contained in the segment, we have

$$M=\frac{A_{1l}}{A_{1l+1}}\pm 2\sqrt{(-1)^{i+j}A_{0i}A_{2j}\left(\frac{A_{1l}}{A_{1l+1}}\right)^{i+j}}\frac{(-1)^l A_{1l+1}^{l-1}}{A_{1l}^{l}} +\ldots,$$

$$y=- \frac{A_{1l}}{A_{1l+1}}\mp 2\sqrt{(-1)^{i+j}A_{0i}A_{2j}\left(\frac{A_{1l}}{A_{1l+1}}\right)^{i+j}}\frac{(-1)^l A_{1l+1}^{l-1}}{A_{1l}^{l}} +\ldots,$$

$$x=\pm \sqrt{(-1)^{i-j}\frac{A_{0i}}{A_{2j}}\left(\frac{A_{1l}}{A_{1l+1}}\right)^{i-j}}+\ldots,$$
where we assume that the lift of $\Lambda$ has the equation $y+M+Nx$. Here, the dots stand for terms of higher valuation.

If the vertex of $\Lambda$ is contained in the segment, we have

$$M=\frac{A_{1l}}{A_{1l+1}}\pm 2\sqrt{(-1)^{i+l+1}A_{0i}A_{1l+1} N \left(\frac{A_{1l}}{A_{1l+1}}\right)^{i+l}}\frac{(-1)^{l} A_{1l+1}^{l-1}}{A_{1l}^{l}} +\ldots,$$

$$y=- \frac{A_{1l}}{A_{1l+1}}\mp 2\sqrt{(-1)^{i+l+1}A_{0i}A_{1l+1} N\left(\frac{A_{1l}}{A_{1l+1}}\right)^{i+l}}\frac{(-1)^l A_{1l+1}^{l-1}}{A_{1l}^{l}} +\ldots,$$

$$x=\pm \sqrt{(-1)^{i-l+1}\frac{A_{0i}}{A_{1l+1} N}\left(\frac{A_{1l}}{A_{1l+1}}\right)^{i-l}}+ \ldots.$$
\end{lemma}

This follows from a computation using the local lifting equations, see also Proposition 5.2 in \cite{CM20}.

 In the following, for a vertex $r$ in the dual Newton subdivision of $Q$, we let $a_r$ be the initial of the corresponding coefficient $A_r$ of $Q$.
\begin{definition}[Twisted edge]\label{def:twist} 
 Let $C$  be a curve defined over $K$ with smooth tropicalization $\Trop(C)$. Let $e $  be a bounded edge in $\Trop(C)$ so that the dual subdivision corresponding to $e$ is the segment joined by lattice points $q, q' \in \mathbb{Z}^2$ and the two triangles of the subdivision intersecting along this segment have vertices $r, r' \in \mathbb{Z}^2$.
We say that the edge  $e$ is  \emph{twisted over $k$} if 
$$\sqrt{(-1)^{\delta(e)} a_ra_{r'}  (a_q a_{q'})^{\delta(e)} }\notin k,$$
where $\delta(e) = 0$ if $r =  r' \mod 2$ and $\delta(e) =1$ if $r \neq r' \mod 2$. 
\end{definition}

\begin{example}\label{ex:delta01}
To see the two possible cases for $\delta(e)$ consider the lattice polytopes in Figure \ref{fig:dualSubdivisionTwist}.
Notice that on the left hand side the  coordinates of the vertices  of the two triangles opposite the overlapping edge $e$ have vector difference $(0,0)  $ modulo $2$. Therefore, the points $r$ and $r'$ are equal modulo $2$ and $\delta(e)=0$. On the right hand side the vector difference of the coordinates of the two vertices has difference $(1, 1)$ modulo $2$. Therefore we have $\delta(e) = 1$ in this case. 
\end{example}

\begin{figure}
\includegraphics{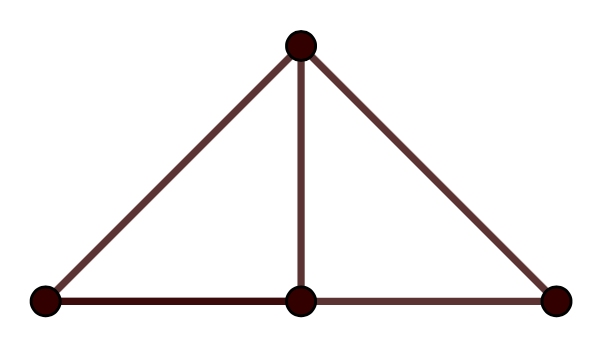}
\hspace{1cm}
\includegraphics{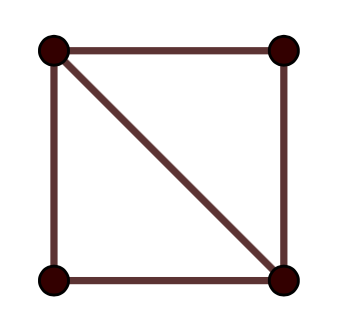}

\caption{The dual subdivision to a bounded edge $e$ and its endpoints  in a tropicalized curve, first in the case when $\delta(e) = 0$ on the left, and then when $\delta(e) = 1$ on the right. } \label{fig:dualSubdivisionTwist}
\end{figure}

\begin{example}\label{ex:realTwist}
The motivation for the terminology of twisting comes from considering amoebas of curves defined over the real numbers. 
Throughout this example, let $k = \mathbb{R}$ and $K = \mathbb{R} \{\!\{t\}\!\}$. Let $C$  be a curve defined over $K$ with smooth tropicalization $\Trop(C) = \Gamma$. Suppose that the defining polynomial $Q$ of   $C$ has convergent coefficients  for $t$ sufficiently small. 
Then we can consider the family of real curves $C_t$ defined by the family of polynomials $Q_t$ for $t$ sufficiently small. The tropical 
curve $\Gamma$ is also obtained as the limit as $s$ tends to infinity of amoebas of the family of curves $\widetilde{C}_s$ defined by the 
family of polynomials $Q_s$, where $s = t^{-1}$.  The amoeba of a curve  $\widetilde{C}_s$ is defined to be $\mathcal{A}_s := \text{Log}_s(\widetilde{C}_s)$, where $\text{Log}_s$ denotes the coordinatewise base $s$ logarithm of  absolute values,  see \cite[Section 2]{BIMS}. 
Moreover, for $s$ sufficiently large, the real amoeba $\mathbb{R}\mathcal{A}_s := \text{Log}_s(\mathbb{R} \widetilde{C}_s)  $ 
either crosses a bounded edge $e$ of $\Gamma$ or it does not, see Figure \ref{fig:amoebatwist}. Following  \cite[Section 3]{BIMS},  call the edge $e$ \emph{twisted} about the edge $e$ if the amoeba crosses it.

The twisting of the edges of $\Gamma$ are determined solely by the signs of the leading terms of the coefficients of the defining polynomial of $C$ as described in \cite[Remark 3.9]{BIMS}. Again let $q, q'$ and $r, r'$ denote the lattice points dual to the edge $e$ as in Definition \ref{def:twist}. Letting $\gamma_q$,$\gamma_{q'}$,$\gamma_r$ and  $\gamma_{r'}$ denote the signs of the leading terms  of the corresponding coefficients, then 
\begin{enumerate}
\item if  $r$ and $r'$ are distinct modulo $2$ then $e$ is twisted if and only if $\gamma_q\gamma_{q'}\gamma_r\gamma_{r'} >0; $
\item if $r$ and $r'$ are equal  modulo $2$ then $e$ is twisted if and only if $\gamma_r\gamma_{r'} <0$. 

\end{enumerate}
This is equivalent to our notion of twist in Definition \ref{def:twist} when $k = \mathbb{R}$. 
See Figure \ref{fig:dualSubdivisionTwist} for two examples of dual subdivisions exhibiting both cases in terms of $r$ and $r'$ above.  
\end{example}

\begin{figure}
\includegraphics[scale = 0.7]{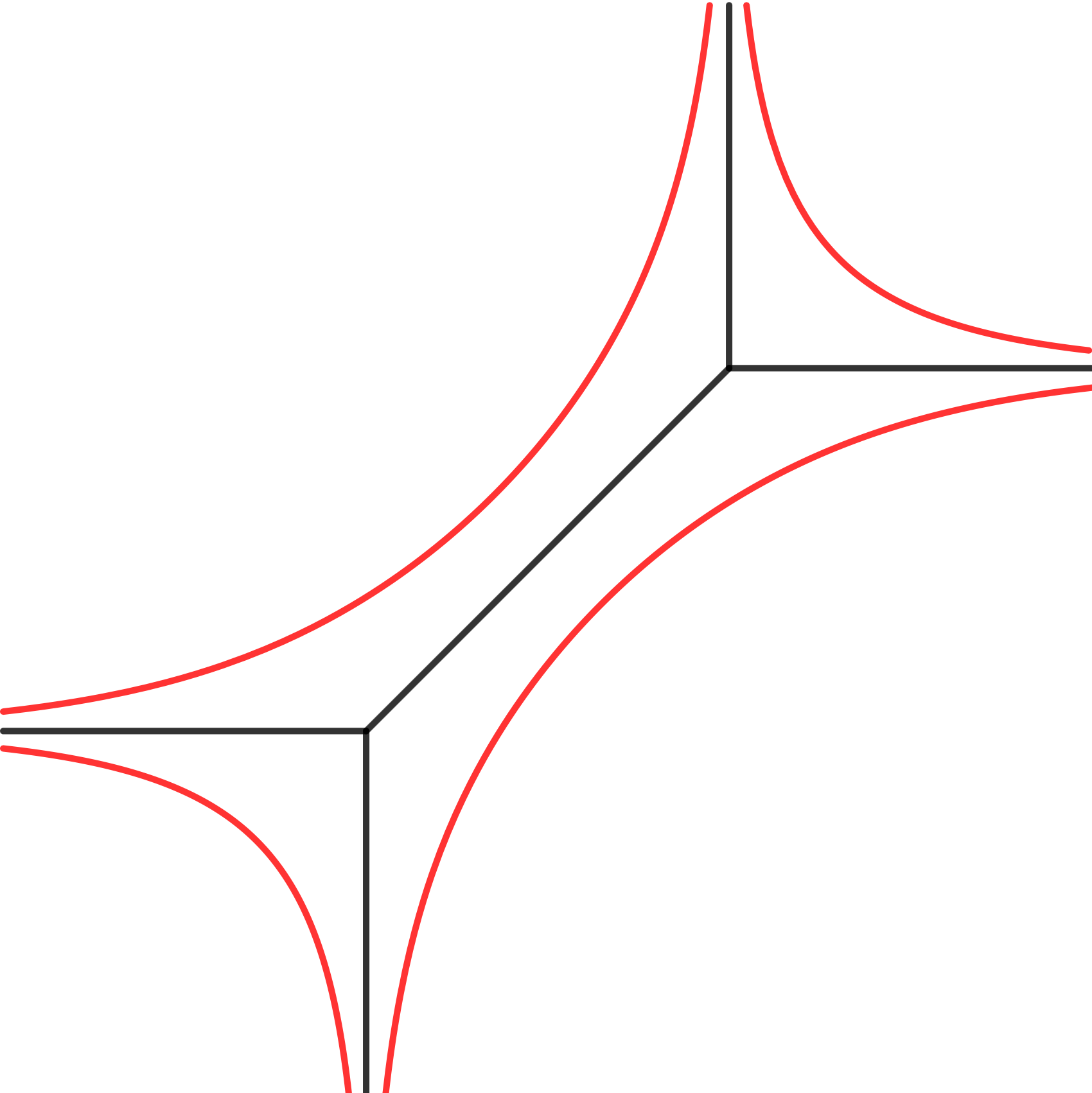} 
\includegraphics[scale = 0.7]{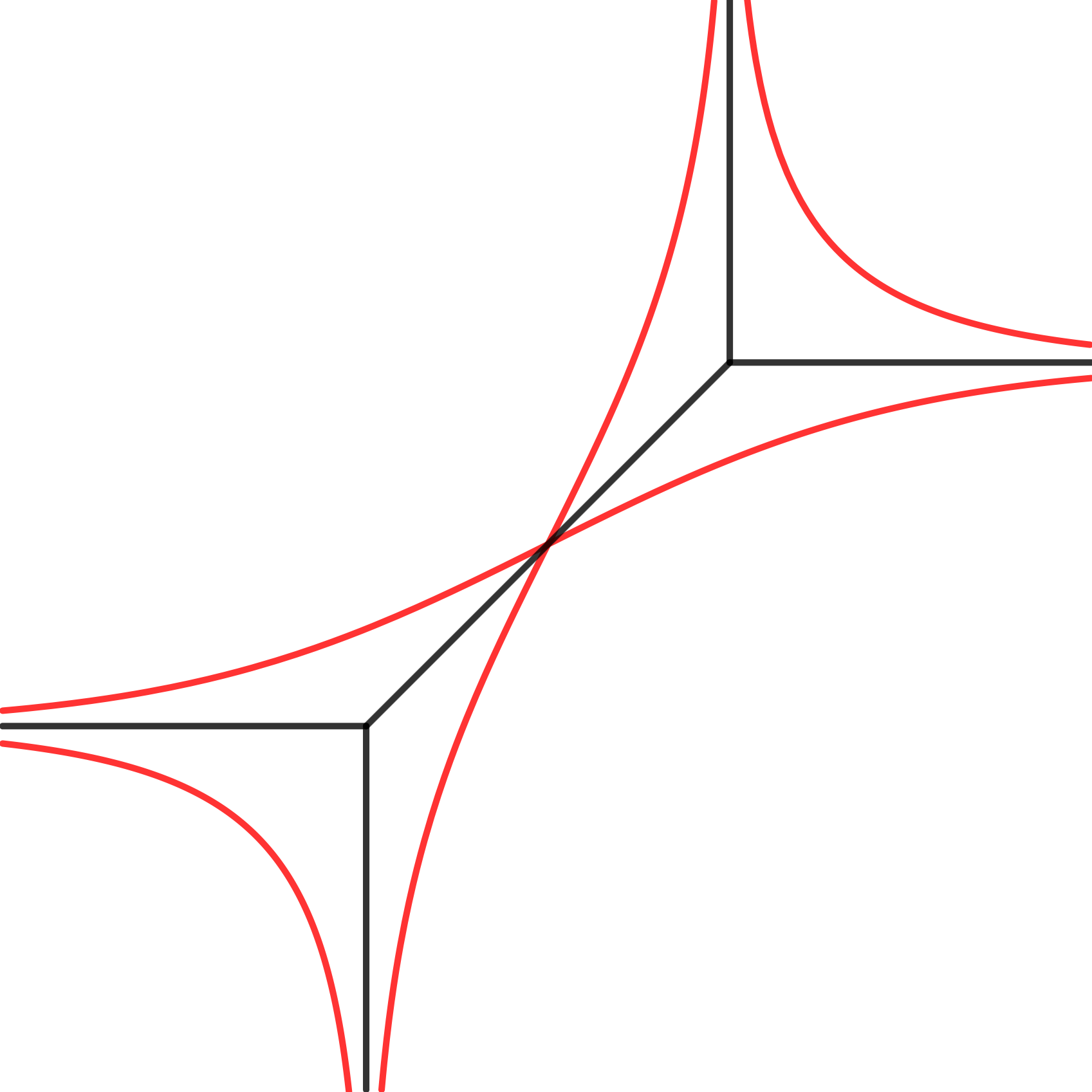}
\caption{An illustration of the non-twisting and twisting of the real amoeba (in red) from Example \ref{ex:realTwist}. } \label{fig:amoebatwist}
\end{figure}

Using our local lifting solutions, obstructions to lifting can be phrased in terms of the twisting of an edge:

\begin{proposition}\label{prop-twisted}
 Let $C$ be a curve defined over $K$ such that $\Trop(C)$ is a smooth tropicalized curve. Let $\Lambda $ be a tropical line and suppose a bounded edge $e$ of $\Trop(C)$ is a tropical tangency component of  $\Lambda \cap \Trop(C)$ strictly contained in a ray of $\Lambda$. 
If $L=V(\ell)$ is a lift of $\Lambda$ which is tangent to $C$ at a point tropicalizing to $e$ then the local equation  $\ell_e$ is defined over $K$ if and only if  $e$ is not  twisted over $k$. 
\end{proposition}

\begin{proof}
The statement follows by applying Lemma \ref{lem-localliftsegment}.
\end{proof}
We now want to consider the situation when $\Trop(C)$ and $\Lambda$ have a  tropical tangency component $T$ which is a segment strictly contained in an edge $e$  of $\Trop(C)$,  i.e.\ the vertex of $\Lambda$ is contained in the segment of intersection, see Figure \ref{fig:relativeTwist},
This happens e.g.\ for the shape (D) we discuss in the proof of Theorem \ref{thm-2H}.
Let $s, s'$ be the lattice points dual to the edge of $\Lambda$ containing $T$ and $q, q'$ be the lattice points dual to $e$ in $C$. Moreover we choose the labels such that $s$ and $q$ correspond to the region of $\mathbb{R}^2\setminus C$ resp.\ $\mathbb{R}^2\setminus \Lambda$ on the same side of $e$, and $s'$ and $q'$ correspond to the region on the other side of $e$.
%assume the points $s, s'$ and $q, q'$ are labeled so as to be on corresponding sides with respect to $e$. 
 As before, for a vertex $r$ in the dual Newton subdivision of $Q$, we let $a_r$ be the initial of the corresponding coefficient $A_r$ of $Q$. Similarly, for a vertex $s$ of the triangle dual to $\Lambda$, we let $b_s$ denote the coefficient of the defining equation of a lift $L$ of $\Lambda$.
By Lemma \ref{lem-localliftsegment} (where the coefficient $M$ of the line can be expressed as  $\frac{b_s}{b_{s'}}$ in the present notation, as the $y$-coefficient there was $1$), in order for a lift $L$ of $\Lambda $ to have a tangency point $P$ with $C$ which tropicalizes to $T$ we must have
 \begin{equation}\label{eq:relativelytwistedRelation1} 
 \frac{b_s}{b_{s'}}=\frac{a_q}{a_{q'}}.
 \end{equation}

\begin{definition}[Relatively twisted edge]\label{def:relTwist}

Let $\Trop(C)$ be smooth, and let $\Lambda$ be a tropical line which has a tropical tangency component $T$ which is a segment strictly contained in an edge $e$  of $\Trop(C)$. 
Assume $e$ is dual to the edge $q q'$ and $\Lambda$ meets the vertex of $e$ which is dual to the triangle spanned by $q{q'}$ and $r$. Let $a_q, a_{q^{\prime}}, a_r$ denote the corresponding initials.  
We use the analogous indices for the initials of the coefficients of the polynomial defining $L$, i.e.\ $b_s$ and $b_{s'}$ are the initials of the monomials of the endpoints of the edge in the dual subdivision  to the ray of $\Lambda$ which contains $e$, and $b_{r'}$ is the 
initial of the remaining coefficient of the defining polynomial of $L$.

We say that the edge $e$ is relatively twisted over $k$ with respect to $C$ and $L$ if 
$$\frac{b_s}{b_{s'}}=\frac{a_q}{a_{q'}} \qquad \text{and} \qquad \sqrt{ (-1)^{\delta(e,L)+1} a_r (a_qa_{q'})^{\delta(e,L)} b_{r'} a_qb_s } \in k,
$$
where $\delta(e,L) = 0$ if the edge in the dual subdivision with endpoints  $a_r, a_q$ is parallel modulo $2$ to the edge in the dual subdivision with endpoints $b_{r'}, b_s$, and  otherwise $\delta (e,L)= 1$.  

\end{definition}

Notice that the first condition $\frac{b_s}{b_{s'}}=\frac{a_q}{a_{q'}}$ implies that the second condition is equivalent to 
$$\sqrt{ (-1)^{\delta(e,L)+1} a_r (a_qa_{q'})^{\delta(e,L)} b_{r'} a_{q'}b_{s'}} \in k.
$$

Also notice that by the balancing condition for tropical curves, the term $\delta(e,L)$ in the above definition can be equivalently be computed by comparing the directions mod $2$ of the edges with endpoints $r, {q^{\prime}}$ and ${r'}, {s'}$.

\begin{figure}
\includegraphics[scale=0.8]{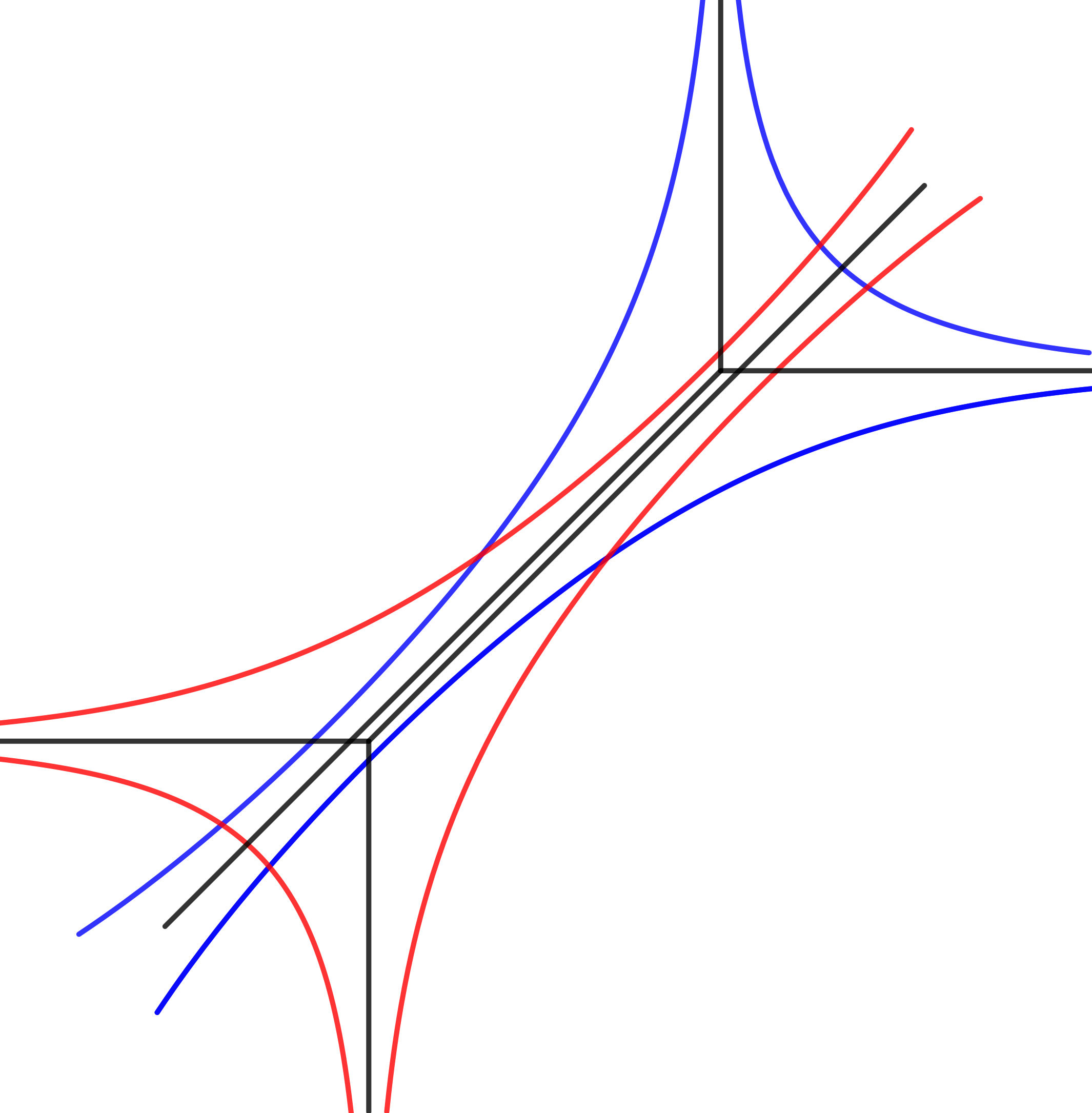}
\put(-20, 70){$\varepsilon_{1}^C$}
\put(-67, 120){$\varepsilon_{2}^C$}
\put(-120,60){$ \varepsilon_{1}^L$}
\put(-75, 20){$\varepsilon_{2}^L$}
\caption{%The left hand side depicts the parts of the dual subdivision to $\Gamma$ and $\Lambda$ which are involved in the definition of relative twists. On the right hand side is 
The picture shows the real amoebas of the two curves, the one of the line in red and (locally) of the quartic in blue.   The segment in the overlap of $\Gamma$ and $\Lambda$ is relatively twisted if the two branches of the two amoebas which are on the same side of the tropicalized curve $\Gamma$ are the image under the coordinatewise logarithm map of pieces of the real algebraic curves which live in the same orthants of $\mathbb{R}^2$.} \label{fig:relativeTwist}
\end{figure}
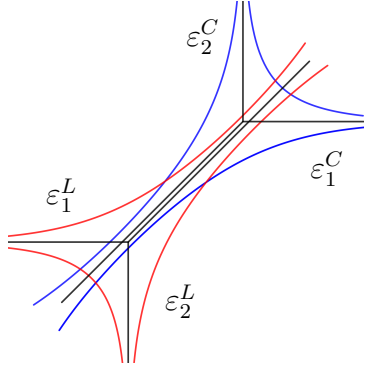

\begin{example}\label{ex:relativeTwistsReal}
Once again the above definition of relative twists is motivated by the geometric picture over the real numbers, as was explained in the case of twisted edges in Example \ref{ex:realTwist}. 
When considering real tropical intersections of curves, Le Texier defines the notion of relative twists \cite{LeTexier}.
As in Example \ref{ex:realTwist}, consider the real amoebas of both the family of lines and the  family of curves, as drawn in Figure \ref{fig:relativeTwist}. The coordinatewise  logarithm map is the composition of the absolute value map and the logarithm map $\log_t \: \mathbb{R}_{> 0} \to \mathbb{R}$. Therefore, each branch of the amoeba of a real curve can be labelled by the orthant of $\mathbb{(R}^\times)^2$ from which it came. In Figure \ref{fig:relativeTwist}, the branches of the amoebas of the family $C$ along  the edge $e$ are labelled $\varepsilon_{1}^C, \varepsilon_{2}^C \in \{+, -\}^2$, and the branches of the amoebas of the family $L$ along  the edge $e$ are labelled $\varepsilon_{1}^L, \varepsilon_{2}^L \in \{+, -\}^2$. The condition that $$\frac{b_s}{b_{s'}}=\frac{a_q}{a_{q'}}$$ amounts to the equality  of the sets
$$\{ \varepsilon_{1}^C, \varepsilon_{2}^C \} = \{ \varepsilon_{1}^L, \varepsilon_{2}^L \}.$$
In the real case, the edge $e$ being relatively twisted with respect to $C$ and $L$ amounts to $ \varepsilon_{1}^C = \varepsilon_{1}^L$ and  $ \varepsilon_{2}^C= \varepsilon_{2}^L$, where $\varepsilon_{i}^C$ and $\varepsilon_{i}^L$
are signs of the branches the amoebas of $C$ and $L$ respectively, which are on opposite sides of the bounded edge $e$. See  Figure \ref{fig:relativeTwist} for the labelling and \cite[Proposition 4.13]{LeTexier} for a proof of this statement. 

In order for $C$ and $L$ to have a real  tangency point tropicalizing to the segment, the edge $e$  must be relatively twisted, see \cite[Theorem 1.4]{LeTexier}. Otherwise there would be two real points in the intersection.  Notice that this condition is opposite to the case when the tropical tangency component is an entire edge. In that case, in order to have a real tangency we required the edge to be non-twisted. 
\end{example}

\begin{proposition}\label{prop-reltwist}
Let $C$ be a curve defined over $K$ such that $\Trop(C)$ is a  smooth tropicalized curve. Let $\Lambda $ be a tropical line and suppose the segment  $T$ is a tropical tangency component of  $\Lambda \cap \Trop(C)$ strictly contained in an edge of $\Lambda$ and of $\Trop(C)$.
Then a lift $L=V(\ell)$ of $\Lambda$ which is tangent to $C$ at a point tropicalizing to $T$ has local equation $\ell_T$ defined over $K$ if and only if  the edge $T$ is relatively twisted over $k$. 
\end{proposition}

\begin{proof}
The statement once again follows  from Lemma \ref{lem-localliftsegment}.
\end{proof}
A condition similar to twisting can be given also for another type of tangency, as follows.
\begin{proposition}\label{prop:vertexEll} 
Let $C$ be a curve defined over $K$ such that $\Trop(C)$ is a  smooth tropicalized curve. Let $\Lambda $ be a tropical line and suppose $p$ is an isolated point of $\Lambda \cap \Trop(C)$ which is a vertex $p$ of $\Lambda$ with intersection multiplicity $2$. 
Then a lift $L$ of $\Lambda$  tangent to $C$ at a point tropicalizing to $p$ is defined over $K$  if and only if $b_{r'},b_{s'} \in k$ and 
$$\sqrt{-a_ra_sb_{r'}b_{s'}} \in k,$$
where $r, s \in \mathbb{Z}^2$ are the lattice endpoints of the edge dual to the edge $E$ of $\Trop(C)$ and $r', s'\in \mathbb{Z}^2$  are lattice endpoints dual to the unique edge of $\Lambda$ which has intersection multiplicity $2$ with $E$. 
\end{proposition}

\begin{proof}
Let $Q_p$ and $\ell_p$ denote the local equations for the curve $C$ and a lift of the tropical line $\Lambda$ at $p$, respectively. Then $Q_p$ is the  binomial $A_r x^{r_1}y^{r_2} + A_sx^{s_1 }y^{s_2}$ and $\ell_p = \ell$ is a trinomial. 
We denote the coefficients of the  linear form  $\ell$  by $B_{q'}, B_{r'}, B_{s'}$. 

If  $\Lambda$ and $C$ have intersection multiplicity $2$ at the vertex $p$ and $r', s'\in \mathbb{Z}^2$  are lattice endpoints dual to the unique edge of $\Lambda$ which has intersection multiplicity $2$ with $E$, then the direction of $E$ is determined. For example, if $r' = (0,0)$ and $s' = (1, 0)$ then  the direction of $E$ must be $(2, 1)$ and hence $Q_p = A_r x^{r_1}y^{r_2} + A_sx^{r_1 + 1}y^{r_2 - 2}$ and set $\ell =  B_{q'}y + B_{s'}x + B_{r'}$. Solving this system of two equations in the torus reduces to the degree two equation in  $y$ given by $$A_ry^2 - \frac{A_sB_t}{B_{s'}}y - \frac{A_sB_{r'}}{B_{s'}} = 0.$$
In order to have a tangency tropicalizing to $p$ we require the  discriminant of this equation to be equal to zero, so that 
$A_r^2 B_{q'}^2 + 4 A_r A_sB_{r'}B_{s'} = 0$. 
 Passing to the initials this implies  that $b_{q'} = \pm \sqrt{  \frac{ - 4 a_r a_sb_{r'}b_{s'}} { a_r^2}} $. 
Thus we have  $b_{q'}  \in k$ if and only if $\sqrt{-a_ra_sb_{r'}b_{s'}} \in k$ and the statement is proven. 
The other two cases for $r'$ and $s'$ are solved analogously. 
\end{proof}

By going through the cases in Appendix \ref{app:classification}, one can see that Propositions \ref{prop-twisted}, \ref{prop-reltwist} and \ref{prop:vertexEll} cover all necessary lifting conditions of generic bitangent shapes except for shape (C). Shape (C) is more involved (see also Proposition 6.4 in \cite{CM20} for the case of real lifting), and we leave it out here.

\subsection{Comparing lifting for different fields}\label{subsec-comparelift}

In this subsection, we sum up the results about lifting of tropical bitangents, and use our study of lifting to compare lifting over different fields.

\begin{proposition}\label{prop-initialsolutionsareLaurentterms}
Let $C=V(Q)$ be a quartic defined over $K$ with generic tropicalization $\Trop(C)$ and $\Lambda$ a tropical bitangent, and assume that the residue characteristic of $K$ is not $2$ or $3$. Then the initials of the coefficients of the defining equation $\ell$ of a lift $L=V(\ell)$ of $\Lambda$, viewed in the algebraic closure of $k$, are Laurent terms in the initials $a_{ij}$ of the coefficients of $Q$ or square roots thereof.

If $\Lambda$ is not contained in a bitangent class of shape (II) in the classification in \cite{CM20}, the coefficients of these Laurent terms are of the form $\pm 2^m$ for some $m\in \mathbb{Z}$. 

For some of the liftable tropical bitangents of a class of shape (II), the coefficients involve $\sqrt{2}$ and $\sqrt{3}$.
\end{proposition}

\begin{proof}
This is a generalization of Theorem 1.2 in \cite{CM20}. As the proof of this theorem, the result follows from a case-by-case analysis involving the classification of shapes of bitangent classes from \cite{CM20}, see also Appendix \ref{app:classification}. Unlike in Theorem 1.2 \cite{CM20}, which focused on the case of lifting over $\mathbb{R}$, we have to pay attention to the coefficients of the solutions of our local lifting equations in order to decide liftability over $K$.
\end{proof}

If we check all $4$ local lifting computations for a tropical bitangent class, we observe that the obstruction for lifting to $K$ is the same for all $4$, i.e.\ the radicands which appear are all equal up to square (see e.g.\ Example \ref{ex-locallifting}). Thus we can conclude:

\begin{theorem}\label{thm-allornone}
Assume the residue characteristic of $K$ is not $2$ or $3$.  Given a generic tropicalization $\Trop(C)$ of a quartic $C=V(Q)$ defined over $K$, and  a bitangent class $S$ of $\Trop(C)$, 
then either  all four lifts of $S$ are defined    over $K$ or none of them are. 

\end{theorem}

As noted above, the obstructions for lifting are equal not only for the coefficients of the equations of the bitangent lines, but also for the tangency points. Taking the special behaviour of shape (II) into account (see Proposition \ref{prop-initialsolutionsareLaurentterms}), we can deduce:

\begin{corollary}\label{cor-totallyreal}
Assume $\sqrt{2}, \sqrt{3}$ exist in $k^\times$.
Given a generic tropicalization $\Trop(C)$ of a quartic $C=V(Q)$ defined over $K$, and a  bitangent class $S$ of $\Trop(C)$ such that the equations of  all four lifts are defined over $K$. Then 
the tangency points  are also all defined over $K$.

\end{corollary}

To compare lifting over different fields $K_1$ and $K_2$ with residue fields $k_1$ and $k_2$ respectively, we first suppose 
there exists an isomorphism $\phi: k_1^{\times} / (k_1^{\times})^2 \to k_2^{\times} / (k_2^{\times})^2$.
For an element $a \in k^{\times} $, let $\overline{a} $ denote its class in $k^{\times} / (k^{\times} )^2$. 
As a main example, consider the case when $k$ is the real numbers or any finite field. Then  $$k^{\times} / (k^{\times})^2 \cong \mathbb{Z}/2 \mathbb{Z},$$
and an isomorphism $k^{\times} / (k^{\times})^2 \to  \mathbb{Z}/2 \mathbb{Z}$ is given by the Legendre symbol. 
For a finite field $k$ we let $\legendre{a}{k}$ denote the Legendre symbol of $a \in k^\times$. Concretely, we have 
$$
\overline{a} = \legendre{a}{k} =  
\begin{cases}
+1  \text{ if } x^2 = a \text{ has a solution in } k \\
-1  \text{ if } x^2 = a \text{ has no solution in } k 
\end{cases}
$$
  If we take $k = \mathbb{R}$ then the Legendre symbol is simply remembering whether $a \in k^{\times}$ is positive or negative. 
For two fields $k_1, k_2$ with isomorphisms $k_i^{\times} / (k_i^{\times})^2 \to  \mathbb{Z}/2 \mathbb{Z}$, there is a unique isomorphism 
$\phi: k_1^{\times} / (k_1^{\times})^2 \to k_2^{\times} / (k_2^{\times})^2$. 

Using Theorem \ref{thm:twistlift} we can relate the lifts over different fields when $k_1^{\times} / (k_1^{\times})^2$ and $k_2^{\times} / (k_2^{\times})^2$ are isomorphic.

\begin{theorem}\label{thm-comparinglifting}
Let $K_1$ and $K_2$ be  fields with residue fields $k_1$ and $k_2$, respectively. Suppose there exists an isomorphism of groups $\phi : k_{1}^{\times}/ (k_{1}^{\times})^2  \to k_{2}^{\times}/ (k_{2}^{\times})^2$ such that  $\phi(\overline{-1}) = \overline{-1}$.

For $i=1,2$, let $C_i = V(Q_i) $ be a quartic curve defined over $K_i$ such that $\Trop(C_i)$ is generic.
Let $Q_1 = \sum A_{ij} x^i y^j z^{4-i-j}$ and $Q_2 = \sum B_{ij} x^i y^j z^{4-i-j}$.
We assume $val(B_{ij}) = val(A_{ij})$, in particular $\Trop(C_1)=\Trop(C_2)$, and
 $\phi(\overline{a_{ij}} ) = \overline{b_{ij}}$ 
 for all $i, j$, where $a_{ij}, b_{ij}$ are the initials of $A_{ij}$ and $B_{ij}$, respectively. 

Then a tropical bitangent  $\Lambda$ to the tropicalized quartic $\Trop(C_1) $ lifts to a bitangent  
of $C_1$ over $K_1$ if and only if $\Lambda$ lifts to a bitangent  of $C_2$ defined over $K_2$.
\end{theorem}
This follows from Theorem \ref{thm:twistlift}.

The following corollary deals with the special case of comparing lifting over a finite field $k$ with the real numbers. Here we can make use of the Legendre symbols. This is particularly useful, because lifting over the real numbers can be checked computationally using the \textsc{polymake}-extension of Geiger and Panizzut \cite{polymake, GP21a, GP21}. 

\begin{corollary}\label{cor:compareliftingRk}
Let the residue field $k$ be finite and of characteristic $p$ with $p \equiv 3  \mod 4$ and order $p^{2l + 1}.$ 

Let $C = V(Q) $ be a quartic curve defined over $K$ with $Q = \sum A_{ij} x^i y^j$. 
Then a tropical bitangent  $\Lambda$ to the tropicalized curve $\Trop(C) $ lifts to a bitangent 
over $K$ if and only if $L$ lifts to a bitangent  of $C' = V(Q')$ defined over $\mathbb{R}\{\!\{t \}\!\}$,
where $Q' = \sum B_{ij} x^i y^j$ and  $B_{ij} \in \mathbb{R}\{\!\{t \}\!\}$  are such that  $val(B_{ij}) = val(A_{ij})$ and $$ \legendre{a_{ij}}{k} = \legendre{b_{ij}}{\mathbb{R}}.$$

\end{corollary}

This follows from Theorem \ref{thm-comparinglifting} by inserting $K_2=\mathbb{R}\{\!\{t \}\!\}$.

 \section{The Grothendieck-Witt ring and $\mathbb{A}^1$-enumerative invariants}\label{sec-GW}

Inspired by and building upon a broader program using $\mathbb{A}^1$-homotopy theory to introduce arithmetic refinements to enumerative geometry via quadratic forms \cite{Hoyois14, KW19, KW20,  Levine20}, Larson and Vogt considered an arithmetic count of bitangents for smooth plane quartics with values in the Grothendieck-Witt ring of the ground field \cite{LV21}. One of our main goals is to demonstrate, using the arithmetic count of bitangents as an illustrative test case, how tropical methods can be useful for computing such arithmetic counts.

\subsection{The Grothendieck-Witt ring}\label{subsecGWintro}
We now recall the definition and basic properties of the Grothendieck-Witt rings in which these arithmetic counts take their values.  See \cite{Lam05} for a beautiful and comprehensive expository treatment, which includes the proofs that we omit.  In this section, $K$ denotes a field of characteristic not equal to two.

A \emph{quadratic space} is a finite-dimensional $K$-vector space $V$ equipped with a symmetric bilinear form $q \colon V \times V \to K$. Two quadratic spaces $(V,q)$ and $(V',q')$ are isomorphic if there is an isomorphism of $k$-vector spaces $\phi\colon V \to V'$ such that $q(v,w) = q'(\phi(v), \phi(w))$ for all $v$, $w$ in $V$.

\begin{definition}
For $a \in K^\times$, we write $\langle a \rangle$ for the $1$-dimensional quadratic space $(k,q)$ with $q(x,y) = axy$.   The {\em hyperbolic plane} is $\mathbb{H} = \langle 1 \rangle \oplus \langle -1 \rangle.$
\end{definition}

\begin{lemma}\label{lem:hyperbolic}
For any $a \in K^\times$, the quadratic space $\langle a \rangle \oplus \langle -a \rangle$ is isomorphic to $\mathbb{H}$, as is the quadratic space $K^2$ with quadratic form $\big((x_1, x_2),(y_1, y_2)\big) \mapsto ax_1y_2 + ax_2y_1$.
\end{lemma}
For a proof, see e.g.\ Corollary 16 in \cite{Lam05}.

Isomorphism classes of quadratic spaces naturally form a semiring with addition and multiplication given by orthogonal direct sum and tensor product, respectively.  A theorem of Witt shows that this addition is cancellative, i.e., if $V$, $W$, and $W'$ are quadratic spaces such that $V \oplus W$ is isomorphic to $V \oplus W'$ then $W$ is isomorphic to $W'$.  As a consequence, this semiring injects into its associated ring of formal differences.

\begin{definition}
The  \emph{Grothendieck-Witt ring} $\GW(K)$ is the ring of formal differences of isomorphism classes of quadratic spaces over $K$.
\end{definition}

\noindent In other words, elements of $\GW(K)$ are formal differences $V - W$, where $V$ and $W$ are isomorphism classes of quadratic spaces.

\begin{definition}
The degree map $\deg \colon \GW(k) \to \mathbb{Z}$ takes a quadratic space $V$ to its dimension as a $K$-vector space. 
\end{definition}

\noindent For instance, $\deg \langle a \rangle = 1$ and $\deg \mathbb{H} = 2$. Note that $\deg$ is a map of rings. All of the $\GW(K)$-valued arithmetic counts that we consider specialize to the classical integer valued enumerative invariants when composed with the degree map.

\begin{lemma}
\label{lem-GWrelations}
As an additive group, $\GW(K)$ is generated by $\{\langle a \rangle : a\in K^\times\}$, with relations generated by
\begin{enumerate}
\item $\langle a \rangle= \langle a b^2 \rangle$ for all $a$, $b$ in $K^\times$, and
\item $\langle a\rangle + \langle b \rangle= \langle a+b \rangle+ \langle ab(a+b) \rangle$,
for all $a$, $b$ in $K^\times$ such that $a + b \neq 0$.
\end{enumerate}
\end{lemma}

\noindent When no confusion seems possible, we write $\langle a \rangle$, $\mathbb{H}$, and so on, not only for a given quadratic space, but also for its class in $\GW(K)$. 

We will frequently consider how Grothendieck-Witt rings behave with respect to finite field extensions.  Suppose $K' / K$ is a finite extension. Then any finite dimensional $K'$-vector space $V$ is also a finite-dimensional as a $K$-vector space, and we write $V_K$ to denote $V$, viewed as a $K$-vector space. If $(V,q)$ is a quadratic space over $K'$, then $(V_K, \Tr_{K'|K} \circ q)$ is a quadratic space over $K$. One writes 
\[
\Tr_{K'|K} \colon \GW(K') \to \GW(K)
\]
for the induced map of Grothendieck-Witt rings.  Note that $\Tr_{K'|K}$ multiplies degrees by a factor of $[K' : K]$, since $\dim_K V_K = [K':K] \dim_{K'} V$.

\begin{example}\label{ex:degree2trace} 
Let  $K'/K$ be a field extension of degree 2, and let $a \in (K')^\times$. We now explain how to compute $\Tr_{K'|K} \langle a \rangle \in \GW(K)$.
Suppose $K' = K(\beta)$, and $b = \beta^2$.  Use $\{ 1, \beta \}$ as a $K$-basis for $K'$, and write
\[
a = r + s \beta.
\]
Then the bilinear form on $\Tr_{K'|K} \langle a \rangle$ may be expressed by the symmetric matrix
  $$\left(\begin{array}{cc} 2r & 2 bs  \\2 bs & 2br \end{array}\right).$$
If $r = 0$, then $\Tr_{K'|K} \langle a \rangle \cong \mathbb{H}$, by Lemma~\ref{lem:hyperbolic}.  If $r \neq 0$, then we can diagonalize the symmetric matrix to get $\Tr_{K'|K} \langle a \rangle \cong \langle 2r \rangle \oplus \langle 2br (r^2 - bs^2) \rangle$.

A case of particular interest is when $K = \mathbb{R}$ and $b = -1$.  Then $r^2 - bs^2$ is positive and hence is a square, so the above quadratic space is isomorphic to $\langle 2r \rangle \oplus \langle -2r \rangle \cong \mathbb{H}$. Thus, for any $a \in \mathbb{C}^\times$, we have $\Tr_{\mathbb{C}|\mathbb{R}} \langle a \rangle \cong \mathbb{H}$.
\end{example}

\subsection{Grothendieck-Witt rings over valued fields} \label{subsec:GWisom}

We now return to the case where $K$ is a Henselian valued field with residue field $k$. We assume the characteristic of $k$ is not 2 and fix a section of the valuation $\sigma \colon \val(K^\times) \to K^\times$.  The initial of an element $a \in K^\times$, denoted $\ini(a)$, is the image of $\sigma(-\val(a)) \cdot a$ in $k^\times$.  We also make the simplifying assumption that the value group of $K$ is $2$-divisible, cf. Remark~\ref{rem:2-divis}.

\begin{theorem} \label{thm-GWKGWk}
Let $K$ be a Henselian valued field of residue characteristic not equal to 2, with 2-divisible value group, and with a section of the valuation.

There is an isomorphism of Grothendieck-Witt rings $$g: \GW(K) \xrightarrow{\sim} \GW(k), \  \langle A \rangle \mapsto \langle \ini(A) \rangle.$$ Moreover, this isomorphism is independent of the section $\sigma$.
\end{theorem}

\begin{proof}
We first show that $g$ is well-defined, i.e., that $\langle A \rangle \mapsto \langle \ini(A) \rangle$ respects the relations (1) and (2) from Lemma~\ref{lem-GWrelations}. Let $a=\ini(A)$ and $b=\ini(B)$.  For (1), we have $\langle \ini(AB^2) \rangle = \langle ab^2\rangle$.  For (2), we consider a few subcases. First, suppose $\val(A) \neq \val(B)$.  Without loss of generality we may assume $\val(A) < \val(B)$.  Then $\ini (\langle A\rangle +\langle B\rangle)= \langle a\rangle + \langle b\rangle$, while $\ini \langle A+B \rangle = \langle a \rangle$ and  $\ini \langle AB(A+B) \rangle = \langle a^2 b \rangle$.  Using \ref{lem-GWrelations}(1), we see that $ g(\langle A+B \rangle +\langle AB(A+B)\rangle) = \langle a \rangle+\langle a^2 b\rangle = \langle a \rangle +\langle b \rangle$, as required. If $\val(A)=\val(B)$ but $a+b\neq 0$, then (2) is clear.  It remains to consider the case where $\val(A)=\val(B)$ and $a+b=0$.  Let $c=\ini(A + B)$.  Then $g(\langle A + B\rangle)=\langle c\rangle$ and $g(\langle AB(A+B)\rangle) =\langle abc\rangle=\langle -a^2c\rangle$, which is equal to $\langle -c\rangle$ in $\GW(k)$, again by \ref{lem-GWrelations}(1).  Finally, $g(\langle A\rangle+\langle B\rangle) = \langle a\rangle + \langle -a\rangle  = \langle c\rangle  + \langle -c\rangle$, by Lemma \ref{lem:hyperbolic}. This proves that $g$ is well-defined.  

It is evident that $g$ respects addition and multiplication, so it is a ring map. We now construct its inverse.  For each $a \in k^\times$, choose some $A \in R$ with residue $a$.  We claim that $\langle a \rangle \mapsto \langle A \rangle$ gives a well-defined map $\GW(k) \to \GW(K)$.  Indeed, if $A$ and $A'$ are two lifts of $a$, then $A/A'$ has valuation 0 and initial $1$.  Since the characteristic of $k$ is not 2 and $K$ is Henselian, it follows that $A/A'$ is a square and hence $\langle A \rangle \cong  \langle A' \rangle$.  Moreover, since $\sigma$ is a section and $\val(K^\times)$ is 2-divisible, $t^v$ has a square root for every $v \in \val(K^\times)$. Thus $\GW(K)$ is generated by quadratic spaces $\langle A \rangle$ such that $\val(A) = 0$, and hence the map so defined surjects onto $\GW(K)$ and gives an inverse to $g$, as required.

Finally, again using the fact that $\val(K^\times)$ is 2-divisible, the image of the section $\sigma$ is contained in the multiplicative subgroup of squares.  Thus, if we choose a different section, the resulting initials of any $a \in K^\times$ will differ by a square and hence the isomorphism $\GW(K) \to \GW(k)$ is independent of this choice.
\end{proof}

\subsection{The Qtype of a bitangent to a quartic}\label{subsec-GWbitangent}

A line $L$ in $\mathbb{P}^2_k$ is a closed point of $(\mathbb{P}^2_k)^\vee$. Let $K_L$ denote the field of definition of $L$.  So $K_L/K$ is a finite extension and $L$ corresponds to a Galois orbit of geometric lines defined over the finite extension $K_{L}/K$.  Suppose $L$ is a bitangent to our plane quartic $C$.  We assume furthermore that the intersection $L \cap C$ is disjoint from the line at infinity $L_\infty$ given in homogeneous coordinates by $z = 0$. The {\em $\GW$-multiplicity of the bitangent $L$} with respect to the fixed reference line $L_\infty$ is an element of degree $[K_{L}:K]$ in $\GW(K)$ that is defined as follows. 

Let $\mathbb{A}^2 = \mathbb{P}^2 \setminus L_\infty$. Let $Q(x,y)$ be a (non-homogeneous) quartic polynomial that vanishes on the affine plane quartic $C \cap \mathbb{A}^2$, and let $\partial_L$ denote the derivation with respect to a linear form (defined over $K_{L}$) that vanishes on $L$.  Note that both $Q(x,y)$ and $\partial_L$ are defined only up to nonzero scalars (in $K^\times$ and $K_{L}^\times$, respectively). 

\begin{definition}
Let $P_1$, $P_2 \in C(\overline K)$ be the points where $L$ is tangent to $C$. Then the $\mathrm{Qtype}$ of $L$ with respect to $L_\infty$ is
\begin{equation}\label{eq:Qtype}
 \Qtype(L) := \langle\partial_L Q(P_1)\cdot \partial_L Q(P_2)  \rangle \in \GW(K_{L}), 
\end{equation}
and the {\em $\GW$-multiplicity of $L$} is $$ \mult_{\GW} (L) := \Tr_{K_L|K} (\Qtype(L)) \in \GW(K).$$
\end{definition}

\medskip

\noindent To see that $\Qtype(L)$ is well-defined, first suppose $P_1$ and $P_2$ are rational over $K_{L}$.  Note that rescaling $Q$ or $\partial_L$ by a nonzero factor $a$ multiplies  $\partial_L Q(P_1)\cdot \partial_L Q(P_2)$ by $a^2$.  Then since $\langle a^2 b \rangle = \langle b \rangle$, the $\Qtype$ of $L$ is well-defined.

Otherwise, if $P_1$ and $P_2$ are not rational over $K_L$ then they are rational over some quadratic extension $K_1/K_{L}$.  In this case $\partial_L Q(P_1)$ is conjugate to $\partial_L Q(P_2)$, and so $\partial_L Q(P_1)\cdot \partial_L Q(P_2)$ is in $K_{L}$. Once again, rescaling $Q$ or $L$ changes $\partial_L Q(P_1)\cdot \partial_L Q(P_2)$ by a square, and so $\Qtype(L)$ is well-defined in $\GW(K_L)$.

\begin{example}\label{ex:realQtype}
When $K = \mathbb{R}$, the $\Qtype$ has a natural topological interpretation, as explained by Larson and Vogt, see page 3 in \cite{LV21}. Suppose a bitangent is defined over $\mathbb{R}$. If each of the bitangency points of $L$ is defined over $\mathbb{R}$, then near these two points in the affine plane $\mathbb{A}^2(\mathbb{R})$, the real locus of $C$ is either on the same side of the bitangent line, in which case the $\Qtype$ is $\langle 1 \rangle$, or on different sides, in which case the $\Qtype$ is $ \langle -1 \rangle$. See Figure \ref{fig-geomQtypereals}. When these points are not defined over $\mathbb{R}$, the $\Qtype$ is $\langle 1 \rangle$. If the bitangent line is not defined over $\mathbb{R}$ then the $\Qtype$ of $L$ is $\langle 1 \rangle \in \GW(\mathbb{C})$
 and $\mult_{\GW} (L)  = \langle 1 \rangle + \langle -1 \rangle = \mathbb{H}$ by Example \ref{ex:degree2trace}. 
 \end{example}

\begin{figure}
\begin{center}

\tikzset{every picture/.style={line width=0.75pt}} %set default line width to 0.75pt        

\begin{tikzpicture}[x=0.75pt,y=0.75pt,yscale=-1,xscale=1]
%uncomment if require: \path (0,300); %set diagram left start at 0, and has height of 300

%Straight Lines [id:da706416055074312] 
\draw    (57.73,180.39) -- (173.07,180.72) ;
%Straight Lines [id:da2374649614394415] 
\draw    (289.17,146.22) -- (376.94,146.22) ;
%Curve Lines [id:da9960547514109344] 
\draw [color={rgb, 255:red, 74; green, 144; blue, 226 }  ,draw opacity=1 ]   (124.33,170.17) .. controls (151.33,205.5) and (162.2,139.2) .. (162.6,110.8) ;
%Curve Lines [id:da9988959346027133] 
\draw [color={rgb, 255:red, 74; green, 144; blue, 226 }  ,draw opacity=1 ]   (97.4,170.47) .. controls (104.67,160.17) and (119,160.83) .. (124.33,170.17) ;
%Curve Lines [id:da4163213605581233] 
\draw [color={rgb, 255:red, 74; green, 144; blue, 226 }  ,draw opacity=1 ]   (65.4,110.87) .. controls (65.4,138.07) and (66.33,204.17) .. (97.4,170.47) ;
%Shape: Circle [id:dp7322917879697153] 
\draw  [fill={rgb, 255:red, 0; green, 0; blue, 0 }  ,fill opacity=1 ] (138.05,180.43) .. controls (138.05,179.8) and (138.56,179.29) .. (139.19,179.29) .. controls (139.82,179.29) and (140.33,179.81) .. (140.33,180.43) .. controls (140.33,181.06) and (139.81,181.57) .. (139.18,181.57) .. controls (138.55,181.57) and (138.04,181.06) .. (138.05,180.43) -- cycle ;
%Shape: Circle [id:dp6742091908899995] 
\draw  [fill={rgb, 255:red, 0; green, 0; blue, 0 }  ,fill opacity=1 ] (80.9,179.98) .. controls (80.91,179.35) and (81.42,178.84) .. (82.05,178.85) .. controls (82.68,178.85) and (83.19,179.36) .. (83.19,179.99) .. controls (83.18,180.62) and (82.67,181.13) .. (82.04,181.13) .. controls (81.41,181.12) and (80.9,180.61) .. (80.9,179.98) -- cycle ;
%Shape: Circle [id:dp4007298215774645] 
\draw  [fill={rgb, 255:red, 0; green, 0; blue, 0 }  ,fill opacity=1 ] (308.18,145.98) .. controls (308.18,145.35) and (308.7,144.84) .. (309.33,144.85) .. controls (309.96,144.85) and (310.47,145.36) .. (310.46,145.99) .. controls (310.46,146.62) and (309.95,147.13) .. (309.32,147.13) .. controls (308.69,147.12) and (308.18,146.61) .. (308.18,145.98) -- cycle ;
%Curve Lines [id:da8899224962722957] 
\draw [color={rgb, 255:red, 74; green, 144; blue, 226 }  ,draw opacity=1 ]   (296.06,197.67) .. controls (304.72,128.67) and (314.06,129.11) .. (320.94,197.78) ;
%Curve Lines [id:da5299572570915227] 
\draw [color={rgb, 255:red, 74; green, 144; blue, 226 }  ,draw opacity=1 ]   (344.06,95.78) .. controls (350.06,161.78) and (368.06,163.11) .. (373.39,95.78) ;
%Shape: Circle [id:dp37369222730823304] 
\draw  [fill={rgb, 255:red, 0; green, 0; blue, 0 }  ,fill opacity=1 ] (358.18,145.98) .. controls (358.18,145.35) and (358.7,144.84) .. (359.33,144.85) .. controls (359.96,144.85) and (360.47,145.36) .. (360.46,145.99) .. controls (360.46,146.62) and (359.95,147.13) .. (359.32,147.13) .. controls (358.69,147.12) and (358.18,146.61) .. (358.18,145.98) -- cycle ;

% Text Node
\draw (103.9,213) node [anchor=north west][inner sep=0.75pt]   [align=left] {$\displaystyle \langle 1\rangle $};
% Text Node
\draw (330.9,213) node [anchor=north west][inner sep=0.75pt]   [align=left] {$\displaystyle \langle -1\rangle $};

\end{tikzpicture}

\end{center}
\caption{Geometric characterization of $\Qtype$ for the reals.}\label{fig-geomQtypereals}
\end{figure}
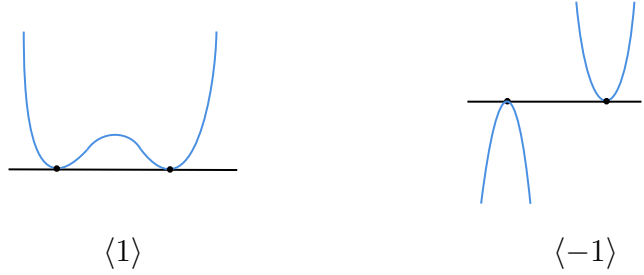

\begin{example}\label{ex:degree2extension} 
Suppose $K = \mathbb{R}$ (or $\mathbb{R}\{\!\{t\}\!\}$).  If $L$ is a bitangent line that is not rational over $K$, (i.e., if $L$ is the point of $(\mathbb{P}^2_\mathbb{R})^\vee$ corresponding to a pair of complex conjugate lines) then, by Example~\ref{ex:degree2trace}, we have $\mult_{\GW}(L) = \mathbb{H}$.
\end{example}

\subsection{The $\GW$-multiplicity of a tropical bitangent class}\label{subsec-tropQtype}

In the following, we study $\GW$-multiplicities for the four liftable members of a tropical bitangent class.

We start, in the following Lemma, with picking a vector for the computation of $\partial_L Q(P)$ for a bitangent line $L$. It turns out that our choice is particularly useful for the tropical degeneration.

\begin{lemma}\label{lem-gradientvector}
Let $L=V(y+M+Nx)$ be a bitangent to a quartic curve defined by $Q$, and let $P$ be a point of tangency.
Then $\partial_L Q(P)=(\frac{\partial}{\partial y}+ \frac{1}{N}\cdot \frac{\partial}{\partial x})Q(P)$.
\end{lemma}
\begin{proof}
Up to rescaling, the direction vector of the line $L$ is $(-1,N)$. To compute $\partial_L$, we can take any vector which is not parallel to the direction vector and multiply it with the gradient. If $(\frac{1}{N},1)$ was parallel to $(-1,N)$, then $-\frac{1}{N}\cdot (-1,N)=(\frac{1}{N},-1)=(\frac{1}{N},1)$ and thus $-1=1$, which is not the case as we are not in characteristic $2$.
\end{proof}

\begin{remark}\label{rem-Laurentmonomial}
In the following theorem, we make use of our choice of gradient and express $\GW$-multiplicity in terms of derivatives of initial forms, which can be determined using the tropicalization of the quartic.

Using the classification in Appendix \ref{app:classification}, one can show that any such derivative of an initial form that can appear equals, after inserting initials of the tangency point, a Laurent monomial in the initials of the coefficients of the quartic times possibly a square root thereof.
\end{remark}

\begin{theorem}\label{thm-Qtypetropical}

Assume the bitangent $L$ is given by the equation $M+NX+Y$ with $M,N\in K$, and let $m,n$ denote their initials.
If the tropicalization $p$ of a tangency point $P$ is contained in the interior of the horizontal ray of $\Trop(L)$, then $$\ini(\partial_L Q(P))= 2\frac{\partial}{\partial y} \ini_{p}(Q)(\ini(P)).$$ If it is contained in the interior of the vertical ray of $\Trop(L)$, then $$\ini(\partial_L Q(P))= 2\frac{1}{n}\frac{\partial}{\partial x} \ini_{p}(Q)(\ini(P)).$$ If it is contained in the interior of the diagonal ray or at the vertex, then $$\ini(\partial_L Q(P))= (\frac{\partial}{\partial y}+ \frac{1}{n}\cdot \frac{\partial}{\partial x})\ini_{p}(Q)(\ini(P)) .$$

In particular, the $\GW$-multiplicity  $\langle\partial_L Q(P_1)\cdot \partial_L Q(P_2)  \rangle$ can be expressed in terms of derivatives of initial forms.

\end{theorem}

\begin{proof}
Without restriction, we can assume $p=(0,0)$ and $Q$ only contains terms whose coefficients have valuation $0$ or higher. The terms of valuation $0$ are then precisely the terms contributing to $\ini_p(Q)$.

By Lemma \ref{lem-gradientvector}, $\partial_L Q(P)= (\frac{\partial}{\partial y}+ \frac{1}{N}\cdot \frac{\partial}{\partial x})Q(P)$. The initial equals \begin{equation}
\ini (\partial_L Q(P)) = \ini \Big(\Big(\frac{\partial}{\partial y}+ \frac{1}{N}\cdot \frac{\partial}{\partial x}\Big)Q(P)\Big)=\ini \Big(\frac{\partial}{\partial y}Q(P)+ \frac{1}{N}\cdot \frac{\partial}{\partial x}Q(P)\Big).\label{eq-iniQ}\end{equation}

If the two summands above are of the same valuation, then cancellation could happen, which would be an obstruction to expressing $\ini(\partial_L Q(P))$ in terms of derivatives of initial forms. 

Assume cancellation happens, then $$\ini\Big(\frac{\partial}{\partial y}Q(P)\Big) = -\frac{1}{n} \ini\Big(\frac{\partial}{\partial x}Q(P)\Big).$$
But the Wronskian $W=N\cdot \frac{\partial}{\partial y}Q(P)- \frac{\partial}{\partial x}Q(P)=0$ vanishes, and thus also its initial vanishes, so $n\cdot \ini(\frac{\partial}{\partial y}Q(P)) = \ini (\frac{\partial}{\partial x}Q(P))$. We substitute $n\cdot \ini(\frac{\partial}{\partial y}Q(P))$ for $\ini (\frac{\partial}{\partial x}Q(P))$ in the above equation, obtaining
$\ini(\frac{\partial}{\partial y}Q(P))= -\ini(\frac{\partial}{\partial y}Q(P))$ from which we deduce $1=-1$ which is a contradiction as we are not in characteristic $2$. Thus no cancellation happens.

There are three cases to distinguish:
\begin{itemize}
\item Assume $p$ is on the horizontal ray. Then the vertex of $\Trop(L)$ is at $(a,0)$ for some $a>0$. The valuation of $N$ is $a$, the valuation of $M$ is zero. 

Consider the Wronskian $W=N\cdot \frac{\partial}{\partial y}Q(P)- \frac{\partial}{\partial x}Q(P)=0$. Consequently, $\ini(N\cdot \frac{\partial}{\partial y}Q(P)- \frac{\partial}{\partial x}Q(P))=0$. But the summand $N\cdot \frac{\partial}{\partial y}Q(P)$ has terms of valuation $a$ or higher, and the summand $\frac{\partial}{\partial x}Q(P)$ has terms of valuation $0$ or higher. Cancellation can thus only appear if $\frac{\partial}{\partial x}Q(P)$ vanishes up to valuation $a$.

Thus both summands in the above equation (\ref{eq-iniQ}) for $\ini(\partial_L Q(P)) $ have expected valuation $0$. 

From the computations for $\GW$-multiplicities we perform in Appendix \ref{app:2H} and \ref{app-tableQtypes} (see Remark \ref{rem-Laurentmonomial}), we conclude that $\ini(\frac{\partial}{\partial y}Q(P))= \frac{\partial}{\partial y}(\ini_p(Q)(\ini(P))\neq 0$ and thus this summand is of the expected valuation. 
The equality $\ini(\frac{\partial}{\partial y}Q(P))= \frac{\partial}{\partial y}(\ini_p(Q)(\ini(P))$ holds since for a polynomial $F(x,y)$ and a point $P=(P_x,P_y)$ that we insert, we have $\ini(F(P_x,P_y))=\ini_{(p_x,p_y)}(F)(\ini(P))$, where $(p_x,p_y)$ contains the valuations, if no cancellation happens in the initial form when inserting $\ini(P)$. Furthermore we have $\ini_{p}(\frac{\partial}{\partial y} Q)= \frac{\partial}{\partial y}(\ini_p(Q))$. This is true since the $p$-weight of a term of $Q$ equals the $p+p_y$-weight of a term of $\frac{\partial}{\partial y} Q$.

Thus
$\val(N\cdot \frac{\partial}{\partial y}Q(P))=a$ and since there is cancellation in the Wronskian, also $\val(\frac{\partial}{\partial x}Q(P))=a$. Furthermore, since there is cancellation in $W$, $n\ini( \frac{\partial}{\partial y}Q(P))= \ini(\frac{\partial}{\partial x}Q(P))$. 

Inserting this into the above equation (\ref{eq-iniQ}), we obtain
\begin{align*}\ini( \partial_L Q(P))& =  \ini \Big(\frac{\partial}{\partial y}Q(P)\Big)+ \ini \Big(\frac{1}{N}\cdot \frac{\partial}{\partial x}Q(P)\Big)\\&
= \ini \Big(\frac{\partial}{\partial y}Q(P)\Big)+ \frac{1}{n}\ini\Big( \frac{\partial}{\partial x}Q(P)\Big)
\\&= \ini \Big(\frac{\partial}{\partial y}Q(P)\Big)+ \frac{1}{n}\cdot n\ini \Big(\frac{\partial}{\partial y}Q(P)\Big)
\\&= \ini \Big(\frac{\partial}{\partial y}Q(P)\Big)+ \ini\Big( \frac{\partial}{\partial y}Q(P)\Big)
\\&= 2 \ini\Big(\frac{\partial}{\partial y}Q(P)\Big) = 2 \frac{\partial}{\partial y}(\ini_p(Q))(\ini(P)).\end{align*}

\item  Assume $p$ is on the vertical ray. Then the vertex of $\Trop(L)$ is at $(0,a)$ for some $a>0$. The valuation of $N$ and $M$ is $-a$. 

Consider the Wronskian $W=N\cdot \frac{\partial}{\partial y}Q(P)- \frac{\partial}{\partial x}Q(P)=0$. Consequently, $\ini(N\cdot \frac{\partial}{\partial y}Q(P)- \frac{\partial}{\partial x}Q(P))=0$. But the summand $N\cdot \frac{\partial}{\partial y}Q(P)$ has terms of valuation $-a$ or higher, and the summand $\frac{\partial}{\partial x}Q(P)$ has terms of valuation $0$ or higher. Cancellation can thus only appear if $\frac{\partial}{\partial y}Q(P)$ vanishes up to valuation $a$.

Thus both summands in the above equation for $\ini(\partial_L Q(P)) $ (\ref{eq-iniQ}) have expected valuation $0$. 

From the computations for $\GW$-multiplicities we perform in Appendix \ref{app:2H} and \ref{app-tableQtypes} (see Remark \ref{rem-Laurentmonomial}), we conclude that $\ini(\frac{\partial}{\partial x}Q(P))= \frac{\partial}{\partial x}(\ini_p(Q)(\ini(P))\neq 0$ and thus this summand is of the expected valuation. Thus $\val(\frac{\partial}{\partial x}Q(P))=0$ and since there is cancellation in the Wronskian, also
$\val(N\cdot \frac{\partial}{\partial y}Q(P))=0$. Furthermore, since there is cancellation in $W$, $n\ini( \frac{\partial}{\partial y}Q(P))= \ini(\frac{\partial}{\partial x}Q(P))$. 

Inserting this into the above equation for $\ini(\partial_L Q(P)) $, we obtain
\begin{align*}\ini(\partial_L Q(P))& =  \ini \Big(\frac{\partial}{\partial y}Q(P)\Big)+ \ini\Big(\frac{1}{N}\cdot \frac{\partial}{\partial x}Q(P)\Big)\\&
= \ini \Big(\frac{\partial}{\partial y}Q(P)\Big)+ \frac{1}{n}\ini\Big( \frac{\partial}{\partial x}Q(P)\Big)
\\&=\frac{1}{n}\ini\Big( \frac{\partial}{\partial x}Q(P)\Big)+ \frac{1}{n}\cdot n\ini\Big(\frac{\partial}{\partial y}Q(P)\Big)
\\&= 2 \frac{1}{n}\ini\Big( \frac{\partial}{\partial x}Q(P)\Big) = 2\frac{1}{n} \frac{\partial}{\partial x}(\ini_p(Q))(\ini(P)).\end{align*}

\item  Assume $p$ is on the vertical ray or the vertex. Then the vertex of $\Trop(L)$ is at $(-a,-a)$ for some $a\geq 0$. The valuation of $M$ is $a$, the valuation of $N$ is zero. 
Then both summands in equation (\ref{eq-iniQ}) are of valuation zero and we can directly express

\begin{align*}\ini(\partial_L Q(P))& =  \ini \Big(\frac{\partial}{\partial y}Q(P)\Big)+ \ini\Big(\frac{1}{N}\cdot \frac{\partial}{\partial x}Q(P)\Big)\\&
= \ini \Big(\frac{\partial}{\partial y}Q(P)\Big)+ \frac{1}{n}\ini\Big( \frac{\partial}{\partial x}Q(P)\Big).\\&
= \frac{\partial}{\partial y}(\ini_p(Q))(\ini(P)) + \frac{1}{n} \frac{\partial}{\partial x}(\ini_p(Q))(\ini(P)).
\end{align*}

\end{itemize}

For the statement in particular, note that by Theorem \ref{thm-GWKGWk}, we have $\langle\partial_L Q(P_1)\cdot \partial_L Q(P_2)  \rangle = \langle\ini(\partial_L Q(P_1)\cdot \partial_L Q(P_2) ) \rangle$. The initial of the product $\ini(\partial_L Q(P_1)\cdot \partial_L Q(P_2) ) $ equals the product of initials $\ini(\partial_L Q(P_1))\cdot \ini( \partial_L Q(P_2) ) $. The factors can be expressed in terms of derivatives of initial forms as described above.
\end{proof}

\begin{remark}\label{rem:degree2extension}
By Example \ref{ex:degree2extension}, a tropical bitangent with lifting multiplicity $2$ which does not lift to $K$ contributes $\mathbb{H}$ to the arithmetic count of bitangents: Such a lift is defined over a field extension of degree $2$. The element whose class in $\GW(K)$ we take equals a Laurent term in the initials of $Q$ times a square root of such a term by Remark \ref{rem-Laurentmonomial} (see also Appendix \ref{app:2H} and \ref{app-tableQtypes}). Thus, in the notation of Example \ref{ex:degree2extension}, $r=0$ and we get a contribution of $\mathbb{H}$. 
\end{remark}

In the following example, we consider tropical bitangents of lifting multiplicity four whose lifts live in a field extension.
\begin{example}\label{ex:degree4extension} 
There are tropical bitangent classes for which we need to add two roots to the ground field to lift to $4$ bitangent lines. These are precisely the zero-dimensional classes (A), (B) and (C) in the classification in \cite{CM20}.
Assume these two roots are $\alpha$ and $\beta$ and $ K(\alpha,\beta)/K$ is of degree $4$.

Let $a \in K(\alpha,\beta)^\times$. We now compute $\Tr_{K(\alpha,\beta)|K} \langle a \rangle \in \GW(K)$.
A basis for $K(\alpha,\beta)$ is given by $1,\alpha,\beta,\alpha\cdot \beta$. We write $a$ in this basis:
\begin{align*}
a=c_1\cdot 1 +c_2\cdot \alpha+c_3\cdot \beta+c_4\cdot \alpha\cdot \beta.
\end{align*}
We study the bilinear map
\begin{align*}
\tilde{q}: K(\alpha,\beta)\times K(\alpha,\beta)\rightarrow K(\alpha,\beta) \rightarrow K,
\end{align*}
where the first arrow is the map $(x,y)\mapsto axy$ and the second is the trace.
Let us first compute the trace for an arbitrary element $b=b_1\cdot 1 +b_2\cdot \alpha+b_3\cdot \beta+b_4\cdot \alpha\cdot \beta$.
The matrix we obtain is
  $$\left(\begin{array}{cccc} b_1 & b_2 \alpha^2 & b_3 \beta^2 & b_4\alpha^2\beta^2  \\b_2 & b_1&b_4 \beta^2 & b_3\beta^2\\ b_3 & b_4\alpha^2 & b_1 & b_2\alpha^2\\ b_4& b_3& b_2& b_1\end{array}\right).$$
The trace is thus $4b_1$.
Now we can insert basis vectors in the above bilinear map to obtain the symmetric matrix
  $$\left(\begin{array}{cccc} 4c_1  &4c_2\alpha^2 &4c_3\beta^2 &4c_4\alpha^2\beta^2  \\4c_2\alpha^2 &4c_1\alpha^2 &4c_4\alpha^2\beta^2 &4c_3\alpha^2\beta^2\\4c_3\beta^2 &4c_4 \alpha^2\beta^2 &4c_1\beta^2 &4c_2\alpha^2\beta^2 \\4c_4\alpha^2\beta^2 &4c_3\alpha^2\beta^2 &4c_2\alpha^2\beta^2 &4c_1\alpha^2\beta^2 \end{array}\right).$$

As we can see in our computations for lifts of tropical bitangents and their $\Qtype$ (see Remark \ref{rem-Laurentmonomial} and Appendix \ref{app:classification}), the bilinear map we obtain for a lift of a tropical bitangent class of type (A), (B) or (C) is of the form $c_2\alpha$ or $c_4\alpha\beta$. Thus, we can insert $c_1=c_3=c_4=0$ or $c_1=c_2=c_3=0$ in the above matrix. In both cases, we obtain a decomposition into two hyperbolic planes.

This computation shows that the $\GW$-multiplicity of a tropical bitangent class whose $4$ lifts live in a field extension of degree $4$ which we obtain by adjoining two roots is $2\mathbb{H}$.
\end{example}

\begin{remark}[Computation of $\GW$-multiplicities]\label{rem-computation}
Theorem \ref{thm-Qtypetropical} together with the theory of tropical bitangents allows to compute $\GW$-multiplicities of quartics as follows:
\begin{itemize}
\item Compute the tropicalization $\Trop(C)$, e.g.\ using the software system \textsc{Polymake} \cite{polymake} or the libraby tropical.lib in the computer algebra system \textsc{Singular} \cite{JMM07a, DGPS}.
\item Using the \textsc{Polymake}-extension on tropical bitangents of quartics by Geiger-Panizzut \cite{GP21a}, compute all its tropical bitangent classes.
\item For each liftable tropical bitangent, use Theorem \ref{thm-Qtypetropical} to compute its $\Qtype$.
\item For each tropical bitangent class which does not lift, deduce from Remark \ref{rem:degree2extension} and \ref{ex:degree4extension} that the $\GW$-multiplicity of the four bitangents equals $2\mathbb{H}$.
\end{itemize}
\end{remark}

In the following, we prepare statements which will be used for the further study of the $\Qtype$ of tropical bitangent lines which lift over $K$.
Our goal is to provide tools for the proof of Theorem \ref{thm-2H}, stating that many bitangent shapes yield a contribution of $2\mathbb{H}$.
Lemma \ref{lem-monomialout} is an observation that simplifies computations in general.
Lemmas \ref{lem-edgeonedgetangency}, \ref{lem-edgeonedgediagonal}, \ref{lem-movehorizontal} and \ref{lem-Hb} will be needed to pair up contributions to arithmetic multiplicities for tropical bitangents in the same bitangent class for various bitangent shapes in Theorem \ref{thm-2H}, spelled out in more detail in the Appendix \ref{app:2H}, see Theorem \ref{thm-2Hdetails}. Going through Appendix \ref{app:2H}, one can see that these four Lemmas are not sufficient to cover all cases. They cover a large subset of the cases however, and each of them appears several times (i.e.\ for many bitangent shapes), which is why we include these statements in the main part of the paper. The remaining cases appear more individually, so we discuss them only in the Appendix after going through the case-by-case analysis of bitangent shapes.

\begin{lemma}\label{lem-monomialout}
Assume the initial form of $Q$ at tropical tangency component $p_1$ equals a monomial $m$ times a form $q$, i.e.\ $\ini_{p_1}(Q)=m\cdot q$. Let  $\partial_L$ be as above. Let $P_1$ be a lift of $p_1$. Then $\partial_L(\ini_{p_1}(Q)) (\ini(P_1))= m\cdot \partial_L(q)(\ini(P_1))$.
\end{lemma}
\begin{proof}
This holds true since the second summand we obtain from the derivative of the product disappears since $q(\ini(P_1))=0$. The monomial cannot vanish by inserting nonzero values.
\end{proof}

\begin{lemma}\label{lem-edgeonedgetangency}

Let $Q$ be a quartic polynomial over $K$ and assume $\Trop(V(Q))$ is smooth and generic. 
Assume there is a tropical tangency component in the interior of an untwisted horizonal bounded edge $E$ of $\Trop(V(Q))$. We denote the liftable tangency point by $p$.
Then the lifts of the tropical bitangents at $p$ come in pairs $L_1, L_2$ 
such that the tangency points $P \in L_1\cap V(Q)$ and $\widetilde{P} \in L_2\cap V(Q)$ which tropicalize to $p$ satisfy  $$\ini(\partial_{L_1}(Q(P))   = - \ini( \partial_{L_2}(Q (\widetilde{P}))  .$$ 
\end{lemma}

\begin{proof}
By Theorem \ref{thm-Qtypetropical},  $\ini(\partial_{L_1}(Q(P))= 2 \frac{\partial}{\partial y} \ini_{p}(Q)(\ini(P))$, and anlogously for $L_2$, as the tropical tangency component $p$ is on the horizontal ray.
The initial form is $\ini_{p}(Q)=a_{1k}xy^k+a_{1k+1}xy^{k+1}$.
By Lemma \ref{lem-monomialout}, after inserting the initials of the solutions $(x_0,y_0)$ of the tangency point we have
$$ \ini(\partial_{L_i}(Q(P_i)) = 2xy^k \frac{\partial}{\partial y}(a_{1k}+a_{1k+1}y) (x_0,y_0)=2 a_{1k+1}x_0y_0^k.$$

Plugging in the solutions for the tangency points $P$ and $\widetilde{P}$ from Lemma \ref{lem-localliftsegment} we obtain the result.
\end{proof}

The following is the analogue of Lemma \ref{lem-edgeonedgetangency} where the untwisted horizontal bounded edge 
$E$ is replaced by a diagonal bounded edge. As we broke symmetry by declaring $\{z=0\}$ to be our infinite line, we cannot expect a similar statement. It is interesting to note that while in Lemma \ref{lem-edgeonedgetangency}, the contributions for the two different lifts were negatives of each other (equal up to sign), here we obtain the same contribution.

\begin{lemma}\label{lem-edgeonedgediagonal}
Let $Q$ be a quartic polynomial over $K$ and assume $\Trop(V(Q))$ is smooth and generic. 

Assume there is a tropical tangency component in the interior of an untwisted diagonal bounded edge $E$ of $\Trop(V(Q))$. We denote the liftable tangency point in it by $p$.
 As in Lemma \ref{lem-edgeonedgetangency}, the lifts of the tropical bitangents at $p$ come in pairs $L_1, L_2$ 
but now the tangency points $P \in L_1\cap V(Q)$ and $\widetilde{P} \in L_2\cap V(Q)$ which tropicalize to $p$ satisfy 
 $$\ini(\partial_{L_1}(Q(P))   =  \ini( \partial_{L_2}(Q (\widetilde{P}))  .$$

\end{lemma}
\begin{proof}
Using Theorem \ref{thm-Qtypetropical} we have to derive with $\frac{\partial}{\partial y}+\frac{1}{n}\cdot \frac{\partial}{\partial x}$, and by Lemma \ref{lem-monomialout}, after inserting the initials of the solutions $(x_0,y_0)$ of the tangency point we have

$$ \partial_L(\ini_{p}(Q)) (x_0,y_0) = x^my^{2-m} \frac{\partial}{\partial y}+\frac{1}{n}\cdot \frac{\partial}{\partial x}(a_{m,3-m}y+a_{m+1,2-m}x) =2 a_{m,3-m}x_0^my_0^{2-m},$$
as $n=\frac{a_{m+1,2-m}}{a_{m,3-m}}$ by Lemma \ref{lem-localliftsegment}. Also, $\frac{y_0}{x_0}=-n$, so we obtain
$$2a_{m,3-m}\Big(-\frac{a_{m+1,2-m}}{a_{m,3-m}}\Big)^{2-m}\cdot x_0^{2m-2}.$$
Up to squares, this equals
$$(-1)^m\cdot 2\cdot a_{m,3-m}^{m-1}a_{m+1,2-m}^m,$$
no matter whether we insert the positive or the negative root which is the solution for $x_0$.
\end{proof}

\begin{lemma}\label{lem-movehorizontal}
Assume the horizontal or vertical ray of a tropical bitangent line can move along a bounded edge maintaining tangency, and the tropicalization of the  tangency points $p_1$ and $p_2$ of the liftable members $L_1, L_2$ of the corresponding bitangent class are the two end vertices.
Then 
$$\ini(\partial_{L_1}(Q(P_1))   = - \ini( \partial_{L_2}(Q (P_2)) \mbox{ up to squares} .$$
\end{lemma}
\begin{proof}
By symmetry, we can without restriction assume that the ray is horizontal.
The only direction vectors of dual edges that fit into a the Newton polygon of a quartic, i.e.\ the triangle with vertices $(0,0)$, $(0,4)$ and $(4,0)$, and intersect with multiplicity $2$ a horizontal edge are $\binom{2}{1}$, $\binom{2}{-1}$ and $\binom{2}{3}$.
Because of smoothness, the adjacent vertices correspond to triangles as depicted in Figure \ref{fig-movehorizontal}.
\begin{figure}
\begin{center}

\tikzset{every picture/.style={line width=0.75pt}} %set default line width to 0.75pt        

\begin{tikzpicture}[x=0.75pt,y=0.75pt,yscale=-1,xscale=1]
%uncomment if require: \path (0,784); %set diagram left start at 0, and has height of 784

%Straight Lines [id:da7805461821421156] 
\draw    (80.13,349.75) -- (100.13,330.25) ;
%Shape: Circle [id:dp9135348579399889] 
\draw  [fill={rgb, 255:red, 0; green, 0; blue, 0 }  ,fill opacity=1 ] (78.25,349.75) .. controls (78.25,348.71) and (79.09,347.88) .. (80.13,347.88) .. controls (81.16,347.88) and (82,348.71) .. (82,349.75) .. controls (82,350.79) and (81.16,351.63) .. (80.13,351.63) .. controls (79.09,351.63) and (78.25,350.79) .. (78.25,349.75) -- cycle ;
%Shape: Circle [id:dp35522444718877955] 
\draw  [fill={rgb, 255:red, 0; green, 0; blue, 0 }  ,fill opacity=1 ] (98,350) .. controls (98,348.96) and (98.84,348.13) .. (99.88,348.13) .. controls (100.91,348.13) and (101.75,348.96) .. (101.75,350) .. controls (101.75,351.04) and (100.91,351.88) .. (99.88,351.88) .. controls (98.84,351.88) and (98,351.04) .. (98,350) -- cycle ;
%Shape: Circle [id:dp2794072951594728] 
\draw  [fill={rgb, 255:red, 0; green, 0; blue, 0 }  ,fill opacity=1 ] (117.75,350) .. controls (117.75,348.96) and (118.59,348.13) .. (119.63,348.13) .. controls (120.66,348.13) and (121.5,348.96) .. (121.5,350) .. controls (121.5,351.04) and (120.66,351.88) .. (119.63,351.88) .. controls (118.59,351.88) and (117.75,351.04) .. (117.75,350) -- cycle ;
%Shape: Circle [id:dp20860097525865218] 
\draw  [fill={rgb, 255:red, 0; green, 0; blue, 0 }  ,fill opacity=1 ] (78.25,330.5) .. controls (78.25,329.46) and (79.09,328.63) .. (80.13,328.63) .. controls (81.16,328.63) and (82,329.46) .. (82,330.5) .. controls (82,331.54) and (81.16,332.38) .. (80.13,332.38) .. controls (79.09,332.38) and (78.25,331.54) .. (78.25,330.5) -- cycle ;
%Shape: Circle [id:dp13013897444600264] 
\draw  [fill={rgb, 255:red, 0; green, 0; blue, 0 }  ,fill opacity=1 ] (98.25,330.25) .. controls (98.25,329.21) and (99.09,328.38) .. (100.13,328.38) .. controls (101.16,328.38) and (102,329.21) .. (102,330.25) .. controls (102,331.29) and (101.16,332.13) .. (100.13,332.13) .. controls (99.09,332.13) and (98.25,331.29) .. (98.25,330.25) -- cycle ;
%Shape: Circle [id:dp8699855776540648] 
\draw  [fill={rgb, 255:red, 0; green, 0; blue, 0 }  ,fill opacity=1 ] (118,330.5) .. controls (118,329.46) and (118.84,328.63) .. (119.88,328.63) .. controls (120.91,328.63) and (121.75,329.46) .. (121.75,330.5) .. controls (121.75,331.54) and (120.91,332.38) .. (119.88,332.38) .. controls (118.84,332.38) and (118,331.54) .. (118,330.5) -- cycle ;
%Shape: Circle [id:dp03600832679418087] 
\draw  [fill={rgb, 255:red, 0; green, 0; blue, 0 }  ,fill opacity=1 ] (160,350) .. controls (160,348.96) and (160.84,348.13) .. (161.88,348.13) .. controls (162.91,348.13) and (163.75,348.96) .. (163.75,350) .. controls (163.75,351.04) and (162.91,351.88) .. (161.88,351.88) .. controls (160.84,351.88) and (160,351.04) .. (160,350) -- cycle ;
%Shape: Circle [id:dp04440212926833775] 
\draw  [fill={rgb, 255:red, 0; green, 0; blue, 0 }  ,fill opacity=1 ] (179.75,350.25) .. controls (179.75,349.21) and (180.59,348.38) .. (181.63,348.38) .. controls (182.66,348.38) and (183.5,349.21) .. (183.5,350.25) .. controls (183.5,351.29) and (182.66,352.13) .. (181.63,352.13) .. controls (180.59,352.13) and (179.75,351.29) .. (179.75,350.25) -- cycle ;
%Shape: Circle [id:dp31565325952248335] 
\draw  [fill={rgb, 255:red, 0; green, 0; blue, 0 }  ,fill opacity=1 ] (199.5,350.25) .. controls (199.5,349.21) and (200.34,348.38) .. (201.38,348.38) .. controls (202.41,348.38) and (203.25,349.21) .. (203.25,350.25) .. controls (203.25,351.29) and (202.41,352.13) .. (201.38,352.13) .. controls (200.34,352.13) and (199.5,351.29) .. (199.5,350.25) -- cycle ;
%Shape: Circle [id:dp33538122497029077] 
\draw  [fill={rgb, 255:red, 0; green, 0; blue, 0 }  ,fill opacity=1 ] (160,330.75) .. controls (160,329.71) and (160.84,328.88) .. (161.88,328.88) .. controls (162.91,328.88) and (163.75,329.71) .. (163.75,330.75) .. controls (163.75,331.79) and (162.91,332.63) .. (161.88,332.63) .. controls (160.84,332.63) and (160,331.79) .. (160,330.75) -- cycle ;
%Shape: Circle [id:dp1606095631183191] 
\draw  [fill={rgb, 255:red, 0; green, 0; blue, 0 }  ,fill opacity=1 ] (180,330.5) .. controls (180,329.46) and (180.84,328.63) .. (181.88,328.63) .. controls (182.91,328.63) and (183.75,329.46) .. (183.75,330.5) .. controls (183.75,331.54) and (182.91,332.38) .. (181.88,332.38) .. controls (180.84,332.38) and (180,331.54) .. (180,330.5) -- cycle ;
%Shape: Circle [id:dp45987196321148105] 
\draw  [fill={rgb, 255:red, 0; green, 0; blue, 0 }  ,fill opacity=1 ] (199.75,330.75) .. controls (199.75,329.71) and (200.59,328.88) .. (201.63,328.88) .. controls (202.66,328.88) and (203.5,329.71) .. (203.5,330.75) .. controls (203.5,331.79) and (202.66,332.63) .. (201.63,332.63) .. controls (200.59,332.63) and (199.75,331.79) .. (199.75,330.75) -- cycle ;
%Shape: Circle [id:dp016029023971428247] 
\draw  [fill={rgb, 255:red, 0; green, 0; blue, 0 }  ,fill opacity=1 ] (229,340) .. controls (229,338.96) and (229.84,338.13) .. (230.88,338.13) .. controls (231.91,338.13) and (232.75,338.96) .. (232.75,340) .. controls (232.75,341.04) and (231.91,341.88) .. (230.88,341.88) .. controls (229.84,341.88) and (229,341.04) .. (229,340) -- cycle ;
%Shape: Circle [id:dp8043421102848354] 
\draw  [fill={rgb, 255:red, 0; green, 0; blue, 0 }  ,fill opacity=1 ] (248.75,340.25) .. controls (248.75,339.21) and (249.59,338.38) .. (250.63,338.38) .. controls (251.66,338.38) and (252.5,339.21) .. (252.5,340.25) .. controls (252.5,341.29) and (251.66,342.13) .. (250.63,342.13) .. controls (249.59,342.13) and (248.75,341.29) .. (248.75,340.25) -- cycle ;
%Shape: Circle [id:dp969215625049808] 
\draw  [fill={rgb, 255:red, 0; green, 0; blue, 0 }  ,fill opacity=1 ] (268.5,340.25) .. controls (268.5,339.21) and (269.34,338.38) .. (270.38,338.38) .. controls (271.41,338.38) and (272.25,339.21) .. (272.25,340.25) .. controls (272.25,341.29) and (271.41,342.13) .. (270.38,342.13) .. controls (269.34,342.13) and (268.5,341.29) .. (268.5,340.25) -- cycle ;
%Shape: Circle [id:dp5698430581000234] 
\draw  [fill={rgb, 255:red, 0; green, 0; blue, 0 }  ,fill opacity=1 ] (229,320.75) .. controls (229,319.71) and (229.84,318.88) .. (230.88,318.88) .. controls (231.91,318.88) and (232.75,319.71) .. (232.75,320.75) .. controls (232.75,321.79) and (231.91,322.63) .. (230.88,322.63) .. controls (229.84,322.63) and (229,321.79) .. (229,320.75) -- cycle ;
%Shape: Circle [id:dp534640407351671] 
\draw  [fill={rgb, 255:red, 0; green, 0; blue, 0 }  ,fill opacity=1 ] (249,320.5) .. controls (249,319.46) and (249.84,318.63) .. (250.88,318.63) .. controls (251.91,318.63) and (252.75,319.46) .. (252.75,320.5) .. controls (252.75,321.54) and (251.91,322.38) .. (250.88,322.38) .. controls (249.84,322.38) and (249,321.54) .. (249,320.5) -- cycle ;
%Shape: Circle [id:dp8709776713467443] 
\draw  [fill={rgb, 255:red, 0; green, 0; blue, 0 }  ,fill opacity=1 ] (268.75,320.75) .. controls (268.75,319.71) and (269.59,318.88) .. (270.63,318.88) .. controls (271.66,318.88) and (272.5,319.71) .. (272.5,320.75) .. controls (272.5,321.79) and (271.66,322.63) .. (270.63,322.63) .. controls (269.59,322.63) and (268.75,321.79) .. (268.75,320.75) -- cycle ;
%Shape: Circle [id:dp7049108540922384] 
\draw  [fill={rgb, 255:red, 0; green, 0; blue, 0 }  ,fill opacity=1 ] (228.88,379.25) .. controls (228.88,378.21) and (229.71,377.38) .. (230.75,377.38) .. controls (231.79,377.38) and (232.63,378.21) .. (232.63,379.25) .. controls (232.63,380.29) and (231.79,381.13) .. (230.75,381.13) .. controls (229.71,381.13) and (228.88,380.29) .. (228.88,379.25) -- cycle ;
%Shape: Circle [id:dp2822922479610508] 
\draw  [fill={rgb, 255:red, 0; green, 0; blue, 0 }  ,fill opacity=1 ] (248.63,379.5) .. controls (248.63,378.46) and (249.46,377.63) .. (250.5,377.63) .. controls (251.54,377.63) and (252.38,378.46) .. (252.38,379.5) .. controls (252.38,380.54) and (251.54,381.38) .. (250.5,381.38) .. controls (249.46,381.38) and (248.63,380.54) .. (248.63,379.5) -- cycle ;
%Shape: Circle [id:dp19241582282236047] 
\draw  [fill={rgb, 255:red, 0; green, 0; blue, 0 }  ,fill opacity=1 ] (268.38,379.5) .. controls (268.38,378.46) and (269.21,377.63) .. (270.25,377.63) .. controls (271.29,377.63) and (272.13,378.46) .. (272.13,379.5) .. controls (272.13,380.54) and (271.29,381.38) .. (270.25,381.38) .. controls (269.21,381.38) and (268.38,380.54) .. (268.38,379.5) -- cycle ;
%Shape: Circle [id:dp02478381190235368] 
\draw  [fill={rgb, 255:red, 0; green, 0; blue, 0 }  ,fill opacity=1 ] (228.88,360) .. controls (228.88,358.96) and (229.71,358.13) .. (230.75,358.13) .. controls (231.79,358.13) and (232.63,358.96) .. (232.63,360) .. controls (232.63,361.04) and (231.79,361.88) .. (230.75,361.88) .. controls (229.71,361.88) and (228.88,361.04) .. (228.88,360) -- cycle ;
%Shape: Circle [id:dp8750788091491746] 
\draw  [fill={rgb, 255:red, 0; green, 0; blue, 0 }  ,fill opacity=1 ] (248.88,359.75) .. controls (248.88,358.71) and (249.71,357.88) .. (250.75,357.88) .. controls (251.79,357.88) and (252.63,358.71) .. (252.63,359.75) .. controls (252.63,360.79) and (251.79,361.63) .. (250.75,361.63) .. controls (249.71,361.63) and (248.88,360.79) .. (248.88,359.75) -- cycle ;
%Shape: Circle [id:dp41951800182920007] 
\draw  [fill={rgb, 255:red, 0; green, 0; blue, 0 }  ,fill opacity=1 ] (268.63,360) .. controls (268.63,358.96) and (269.46,358.13) .. (270.5,358.13) .. controls (271.54,358.13) and (272.38,358.96) .. (272.38,360) .. controls (272.38,361.04) and (271.54,361.88) .. (270.5,361.88) .. controls (269.46,361.88) and (268.63,361.04) .. (268.63,360) -- cycle ;
%Straight Lines [id:da9839746037437977] 
\draw    (100.13,330.25) -- (119.88,330.5) ;
%Straight Lines [id:da9702573299610175] 
\draw    (80.13,349.75) -- (99.88,350) ;
%Straight Lines [id:da5485328966917523] 
\draw    (99.88,350) -- (119.88,330.5) ;
%Straight Lines [id:da17792705803800146] 
\draw    (80.13,349.75) -- (119.88,330.5) ;
%Straight Lines [id:da7713674482674776] 
\draw    (201.38,350.25) -- (161.88,330.75) ;
%Straight Lines [id:da23125230171086686] 
\draw    (162.13,330.25) -- (181.88,330.5) ;
%Straight Lines [id:da7164747394511426] 
\draw    (182.13,350.25) -- (201.88,350.5) ;
%Straight Lines [id:da8441291136037763] 
\draw    (182.13,350.25) -- (201.88,350.5) ;
%Straight Lines [id:da40050203962124376] 
\draw    (182.13,350.25) -- (201.88,350.5) ;
%Straight Lines [id:da9876230042697067] 
\draw    (181.63,350.25) -- (161.88,330.75) ;
%Straight Lines [id:da3843823355015312] 
\draw    (201.63,350) -- (181.88,330.5) ;
%Straight Lines [id:da7091297208144829] 
\draw    (270.25,379.5) -- (230.88,320.75) ;
%Straight Lines [id:da2026553444692374] 
\draw    (250.63,340.25) -- (230.88,320.75) ;
%Straight Lines [id:da7238320319592674] 
\draw    (270.25,379.5) -- (250.5,360) ;
%Straight Lines [id:da8416190253649044] 
\draw    (270.25,379.5) -- (250.63,340.25) ;
%Straight Lines [id:da16534126605823118] 
\draw    (250.5,360) -- (230.88,320.75) ;

% Text Node
\draw (66.75,347) node [anchor=north west][inner sep=0.75pt]   [align=left] {$\displaystyle a$};
% Text Node
\draw (124.75,317.5) node [anchor=north west][inner sep=0.75pt]   [align=left] {$\displaystyle c$};
% Text Node
\draw (93.25,312) node [anchor=north west][inner sep=0.75pt]   [align=left] {$\displaystyle b$};
% Text Node
\draw (100.25,347) node [anchor=north west][inner sep=0.75pt]   [align=left] {$\displaystyle d$};

\end{tikzpicture}

\end{center}
\caption{Possible duals of bounded edges intersecting a horizontal ray with multiplicity $2$.}\label{fig-movehorizontal}
\end{figure}
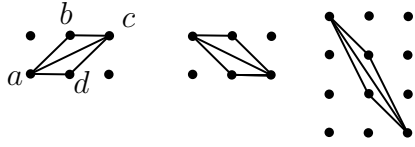

We perform the computation for the left picture in Figure \ref{fig-movehorizontal}, for the others, it is completely analogous.
Here, $\ini_{p_1}(Q)=x^iy^j\cdot (a+bxy+cx^2y)$ and $\ini_{p_2}(Q)=x^iy^j\cdot (a+dx+cx^2y)$. As the horizontal ray cannot meet any other point of the tropicalized quartic, we can conclude that $i=0$.
Using local lifting equations, we obtain $\ini(P_1)=(-\frac{b}{2c},\frac{4ac}{b^2})$ and $\ini(P_2)=(-\frac{2a}{d},\frac{d^2}{4ac})$.
Using Theorem \ref{thm-Qtypetropical} we have to derive with $\frac{\partial}{\partial y}$, and by Lemma \ref{lem-monomialout}, we obtain
 $$\ini(\partial_{L_1}(Q(P_1))=  \Big(\frac{4ac}{b^2}\Big)^j\cdot \Big(-b\frac{b}{2c}+c\frac{b^2}{4c^2}\Big)
 =\Big(\frac{4ac}{b^2}\Big)^j\cdot \Big(-\frac{b^2}{4c}\Big),$$ 
 $$ -\ini(\partial_{L_2}(Q(P_2))= \Big(\frac{d^2}{4ac}\Big)^j\cdot \Big(c\frac{4a^2}{d^2}\Big).
 $$
Up to squares, this equals $\pm a^jc^{j+1}$.
\end{proof}

\begin{remark}
An analogous statement for intersections with the diagonal ray does not hold. (See e.g.\ Appendix \ref{app-tableQtypes}, shape (Ea), where one tropical tangency component contributes $a_{21}$ and the other $-a_{20}a_{31}a_{30}$).
\end{remark}

\begin{lemma}\label{lem-Hb}
Assume a liftable tropical bitangent meets a tropicalized quartic at a point $p$ in the interior of an edge of direction $(1,-1)$ dual to $a_{00}+a_{11}xy$ with its vertex. Assume the second tropical tangency component is on the diagonal ray.
The lifts come in pairs $L_1$, $L_2$ with tangency points $P_1$ and $P_2$ tropicalizing to $p$, and we have
$$\ini(\partial_{L_1}(Q(P_1))   = - \ini( \partial_{L_2}(Q (P_2)).$$
\end{lemma}

\begin{proof}
By a local lifting computation, we obtain for the initials of the two lifts $P_1$ and $P_2$ the coordinates $(x_0,y_0)=\Big(\mp \frac{1}{2n}\sqrt{-\frac{4a_{00}n}{a_{11}}}, \mp \frac{1}{2}\sqrt{-\frac{4a_{00}n}{a_{11}}}\,\Big)$, where $n$ as usual denotes the coefficient of $x$ in the normalized equation defining the bitangent lift. Its initial is imposed by the second tropical tangency component which is on the diagonal ray.

Using Theorem \ref{thm-Qtypetropical} we have to derive with $\frac{\partial}{\partial y}+\frac{1}{n}\frac{\partial}{\partial x}$, and by Lemma \ref{lem-monomialout}, we obtain

$$ a_{11}\cdot x_0+\frac{1}{n} a_{11} y_0=  2a_{11} x_0.$$
Inserting the two solutions for $x_0$, which are negative of each other, the statement follows.
\end{proof}

\begin{theorem}\label{thm-2H}
A bitangent class $S$ of a generic tropicalized quartic $\Trop(C)$ contributes either $2\mathbb{H}$ to the $\mathbb{A}^1$-enumerative count of bitangents to $C$, or a sum of four monomials in the initials of the coefficients of the defining polynomial of $C$. 

 In particular, the total $\mathbb{A}^1$-enumerative count of bitangents to $C$ is determined by its tropicalization $\Trop(C)$ and the square classes of the initials of its coefficients.
\end{theorem}

The exceptional cases (which do not give $2\mathbb{H}$) and their $\GW$-multiplicities are listed in the Appendix \ref{app:2H}, details are spelled out in Theorem \ref{thm-2Hdetails}.

\begin{proof}
The strategy of the proof is to use Lemmas \ref{lem-edgeonedgetangency}, \ref{lem-movehorizontal} and \ref{lem-Hb} to argue that the four lifts come in pairs contributing $\langle \pm a\rangle$ for some $a$, which sums up to $2\mathbb{H}$.
This works if the lifts exist in the field, else we use Examples \ref{ex:degree2trace} and \ref{ex:degree4extension} to show that nevertheless, we obtain $2\mathbb{H}$ as total contribution.
The proof is a case-by-case analysis relying on the classification of bitangent shapes in Appendix \ref{app:classification}. Here, we show one case and refer to the detailed version, Theorem \ref{thm-2Hdetails} in Appendix \ref{app:2H}, for the remaining ones.

We focus on a bitangent shape of type (D) as in Figure 6 of \cite{CM20}. The bitangent shape is a bounded segment which partly overlaps with an edge of the tropicalized quartic. The liftable points are the two end vertices of the segment, each such point has lifting multiplicity $2$. Thus there are two tropical bitangent lines in the shape which each have two algebraic lifts.

The $\mathbb{S}_3/\mathbb{S}_2$-orbit of this shape consists of three cases (Da), (Db) and (Dc). In (Da), the bounded segment is horizontal, in (Db) it is diagonal and in (Dc) vertical.
Figure \ref{fig:D} shows the dual motifs and local pictures of the tropicalized quartic and the tropical tangency components. The pictures contain the two liftable tropical bitangent lines in blue. In each case, the bitangent class is the segment connecting the two vertices of the two blue tropical lines in each picture. Observe that the segment partially overlaps with an edge of the tropicalized quartic, as mentioned before.

 \begin{figure}
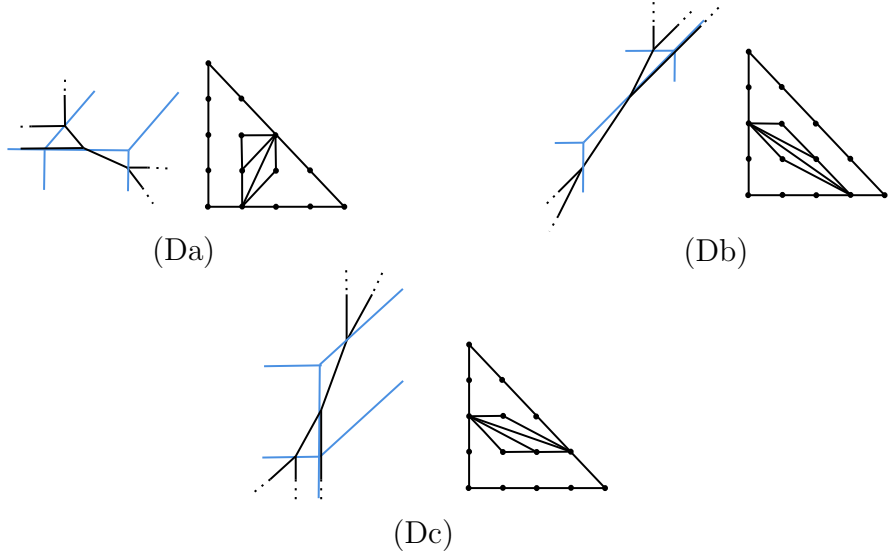

 \begin{center}

\tikzset{every picture/.style={line width=0.75pt}} %set default line width to 0.75pt        



 \end{center}
 \caption{Liftable tropical bitangents and motifs for shapes (Da), (Db) and (Dc).}\label{fig:D}
 \end{figure}

For (Da) and (Dc), each of the two liftable tropical bitangents has a tropical tangency component which is an overlap of a horizontal resp.\ vertical edge. 
Each liftable tropical bitangent has lifting multiplicity $2$, and the two lifts are given by the two choices we obtain for the tangency point on the overlap. The second tangency point is the same for both lifts, for each of the two liftable bitangents.
By Lemma \ref{lem-edgeonedgetangency}, we can pair up the two lifts for each such tropical bitangent such that their Qtypes are negative of each other. Altogether, we obtain $2\mathbb{H}$.

For (Db) this argument does not work, but a computation shows that the contribution of the other tropical tangency component , which is not in the interior of the diagonal edge, equals $a_{21}$ for one resp.\ $-a_{21}$ for the second, up to squares: for the upper tropical tangency component  $p_1$, we have
$\ini_{p_1}(Q)=a_{02}y^2+a_{12}xy^2+a_{21}x^2y$. Solving for the initials of the coordinates of the tangency point and the coefficient $M$ of the line equation $y+Mx+N$, we obtain $m=\ini(M)=-4\frac{a_{02}a_{21}}{a_{12}^2}$, $x=-2\frac{a_{02}}{a_{12}}$ and $y=4\frac{a_{02}a_{21}}{a_{12}^2}$. Using Lemma \ref{lem-monomialout} and Theorem \ref{thm-Qtypetropical}, we obtain $$2\frac{\partial}{\partial y} \ini_{p_1}(Q(\ini(P_1))=8\frac{a_{02}a_{21}}{a_{12}^2} \cdot \Big(a_{02}+a_{12}\cdot \Big(-2\frac{a_{02}}{a_{12}}\Big)\Big)= -8\frac{a_{02}^2}{a_{12}^2}\cdot a_{21}.$$

For the lower tropical tangency component $p_2$, solving for the initials of the coordinates of the tangency point and the ratio $\frac{M}{N}$ of the coefficients of the line equation, we obtain $\frac{m}{n}=\ini\frac{M}{N}=-\frac{a_{11}^2}{4a_{02}a_{30}}$, $x=\frac{a_{11}^2}{4a_{02}a_{30}}$ and $y=-\frac{a_{11}^3}{8a_{02}^2a_{30}}$. We have $\ini_{p_2}(Q)=a_{02}y^2+a_{11}xy+a_{30}x^3$ and
\begin{align*}&2\frac{1}{n}\frac{\partial}{\partial x} \ini_{p_2}(Q(\ini(P_2))=2\frac{1}{m}\cdot \frac{m}{n}\frac{\partial}{\partial x} \ini_{p_2}(Q) \\&= 2\frac{1}{m}\cdot \frac{m}{n} \cdot \Big(a_{11}\cdot \Big(-\frac{a_{11}^3}{8a_{02}^2a_{30}}\Big)+3a_{30}\cdot\Big(\frac{a_{11}^2}{4a_{02}a_{30}}\Big)^2 \Big) \\& = 2\frac{1}{m}\cdot \frac{m}{n} \cdot \Big( \frac{a_{11}^4}{16a_{02}^2a_{30}}\Big)
\\& = 2\Big(-\frac{a_{12}^2}{4a_{02}a_{21}}\Big)\Big( -\frac{a_{11}^2}{4a_{02}a_{30}}\Big)\Big( \frac{a_{11}^4}{16a_{02}^2a_{30}}\Big)
\\&= 2\frac{1}{16^2} \frac{a_{12}^2a_{11}^6}{a_{21}^2a_{02}^4a_{30}^2} a_{21}.
\end{align*}

Using Lemma \ref{lem-edgeonedgediagonal}, we can now pair up a lift of one with a lift of the other tropical bitangent and conclude that the total contribution is again $2\mathbb{H}$.
\end{proof}

\subsection{Comparing the $\Qtype$ for different fields} \label{subsec-comparingQtype}

Building on our comparison of lifting for different fields in Section~\ref{subsec-comparelift}, we can also discuss the comparison of the $\Qtype$ for different fields, as a consequence of Theorem \ref{thm-2H}.
Let $K_1$ and $K_2$ be  fields with residue fields $k_1$ and $k_2$, respectively. Suppose there exists an isomorphism of groups $\phi : k_{1}^{\times}/ (k_{1}^{\times})^2  \to k_{2}^{\times}/ (k_{2}^{\times})^2$ such that  $\phi(\overline{-1}) = \overline{-1}$ and 
$\phi(\overline{2}) = \overline{2}$. Then $\phi$ induces an isomorphism $\GW(k_1)\sim \GW(k_2)$ by sending 
 $\langle a \rangle$ to $\langle \phi(a) \rangle$, as the Grothendieck-Witt ring is generated by elements of the form $\langle a \rangle$ for $a \in k_{i}^{\times}/ (k_{i}^{\times})^2$. As  $\GW(K_i)\sim \GW(k_i)$ by Theorem \ref{thm-GWKGWk}, we also obtain  $\GW(K_1)\sim \GW(K_2)$.

\begin{theorem}\label{thm:qtypescomparison} 
Let $K_1$ and $K_2$ be  fields with residue fields $k_1$ and $k_2$, respectively. Suppose there exists an isomorphism of groups $\phi : k_{1}^{\times}/ (k_{1}^{\times})^2  \to k_{2}^{\times}/ (k_{2}^{\times})^2$ such that  $\phi(\overline{-1}) = \overline{-1}$ and 
$\phi(\overline{2}) = \overline{2}$.

For $i=1,2$, let $C_i = V(Q_i) $ be a quartic curve defined over $K_i$ such that $\Trop(C_i)$ is generic.
Let $Q_1 = \sum A_{ij} x^i y^j z^{4-i-j}$ and $Q_2 = \sum B_{ij} x^i y^j z^{4-i-j}$.
We assume $val(B_{ij}) = val(A_{ij})$, in particular $\Trop(C_1)=\Trop(C_2)$, and
$\phi(\overline{a_{ij}}) = \overline{b_{ij}}$ for all $i, j$, where $a_{ij}, b_{ij}$ are the initials of $A_{ij}$ and $B_{ij}$, respectively.

Suppose all lifts of $\Lambda $ are defined over $K_1$ and $K_2$ then 
$$\sum_{L \text{ lift of } \Lambda} \phi( QType(L) ) = \sum_{L' \text{ lift of } \Lambda}  QType(L') $$ 
where the lifts $L$ are over $K_1$ and the lifts $L'$ are over $K_2$. %and  if $QType(L) = \langle a \rangle$ then $\phi(QType(L) ) := 
%\langle \phi(a) \rangle.$
\end{theorem}

\begin{proof}
%If $k$ has order $p^{2l+1} $ and $p \equiv 7 \mod 8$, then we have that $-1$ is not a square in $k$, $2$ is a square in $k$,  and $k^{\times}/(k^{\times})^2 = k^{\times}/(k^{\times})^4$. This is  also the case when $k = \mathbb{R}$. 
The formulas for the $\Qtype$ of tropical bitangents lifting over $k$ are expressed in terms of a Laurent monomials in the $\{a_{ij}\}$ or the $\{b_{ij}\}$ with coefficients of $-1$, $2$, or squares in the residue fields. Moreover,  the Qtypes are defined modulo squares so the statement follows from the assumptions on the  $\{a_{ij}\}$ and $\{ b_{ij} \}$. 
\end{proof}

We present the special case comparing $\text{QTypes}$ over finite fields and the real numbers. 
Let $K$ be a field with finite residue field $k$ of size  $p^{2l + 1}$ for a prime $p$ satisfying  $p \equiv 7  \mod 8$. 
Let  $\phi : k^{\times}/ (k^{\times})^2  \to \mathbb{R}^{\times}/ (\mathbb{R}^{\times})^2$ denote the unique isomorphism between these two groups.  As before, $\phi$ induces an isomorphism of the Grothendieck-Witt rings.
\begin{corollary}\label{cor:LegendreRfinite}
%Let $K$ be a field with finite residue field $k$ of size  $p^{2l + 1}$ for a prime $p$ satisfying  $p \equiv 7  \mod 8$. 
%Let  $\phi : k^{\times}/ (k^{\times})^2  \to \mathbb{R}^{\times}/ (\mathbb{R}^{\times})^2$ denote the unique isomorphism between these two groups.  

Let $C = V(Q) $ be a quartic curve defined over $K$ with $Q = \sum A_{ij} x^i y^j$, 
and  $C' = V(Q')$ be defined over $\mathbb{R}\{\!\{t \}\!\}$,
where $Q' = \sum B_{ij} x^i y^j$ and  $B_{ij} \in \mathbb{R}\{\!\{t \}\!\}$. Assume that $val(B_{ij}) = val(A_{ij})$ and $$ \legendre{a_{ij}}{k} = \legendre{b_{ij}}{\mathbb{R}}.$$

Suppose all lifts of $\Lambda $ are defined over $K$ and $ \mathbb{R}\{\!\{t \}\!\}$, then
$$\sum_{L \text{ lift of } \Lambda} \phi( QType(L) ) = \sum_{L' \text{ lift of } \Lambda}  QType(L') $$ 
where the lifts $L$ are over $K_1$, the lifts $L'$ are over $\mathbb{R}\{\!\{t\}\!\}$.%, and  if $QType(L) = \langle a \rangle$ then $\phi(QType(L) ) := \langle \phi(a) \rangle.$
\end{corollary}

\begin{proof}
The statement  follows immediately from Theorem \ref{thm:qtypescomparison}, since in a field $k$ of characteristic $p$ with $p \equiv 7 \mod 8$ and with order $p^{2l+1}$ we have $\legendre{-1}{k} = \legendre{-1}{\mathbb{R}}$ and $\legendre{2}{k} = \legendre{2}{\mathbb{R}}$.
\end{proof}

\section{The real $\GW$-multiplicity in tropical geometry}\label{sec-real}
Let $K =\mathbb{R}\{\!\{t\}\!\}$  the field of Puiseux series with real coefficients. 
By the Tarski principle (see e.g.\ Chapter 1 in \cite{JL89}), i.e.\ elimination of quantifiers in the first order theory of real closed fields, $\mathbb{R}\{\!\{t\}\!\}$ is equivalent to the reals, so we can count real bitangents by lifting tropical bitangents to $\mathbb{R}\{\!\{t\}\!\}$. This principle has been applied to other problems in tropical geometry, see e.g.\ \cite{ABGJ}.

Recall from Example \ref{ex:realQtype}, that the $\Qtype$ of a real  bitangent is $\langle 1 \rangle$ or $\langle -1 \rangle$, and the two cases can be characterized geometrically (see \cite{LV21}): Assume first that the tangency points are real. In the affine chart of $\mathbb{P}^2$ we obtain by taking out the line at infinity, the quartic $V(Q)$ can pass on the same side of the bitangent line for the two tangency points, or on opposite sides (see Figure \ref{fig-geomQtypereals}). If the tangency points are not real, the $\Qtype$ is $\langle 1 \rangle$.

Larson and Vogt obtained interesting results for counts of real bitangents according to the sign of their $\Qtype$: If $V(Q)$ does not meet the line at infinity, then the number of real bitangents with $\Qtype$ equal to $\langle 1 \rangle$ minus the number of real bitangents  with $\Qtype$ equal to $\langle -1 \rangle$ equals $4$ \cite{LV21}. Based on a randomized search, they formulate the following conjecture:

\begin{conjecture}[{\cite[Conjecture~2]{LV21}}]\label{con-LV}
Let $V(Q)$ be a real quartic such that for each bitangent $L$, $L\cap V(Q)\cap L_\infty=\emptyset$, where $L_\infty$ denotes the line at infinity. Then the number of real bitangents with $\Qtype$ equal to $\langle 1 \rangle$ minus the number of real bitangents  with $\Qtype$ equal to $\langle -1 \rangle$ is in
$$\{0,2,4,6,8\}.$$
\end{conjecture}
Larson and Vogt prove that the numbers are nonnegative.

In the following, we give a partial proof for the conjecture, for the cases of quartics whose tropicalization is smooth and generic. In particular, the quartics we consider all intersect the infinite line in four points. 
It is interesting to observe that we do not obtain all the possible numbers with this restriction:

\begin{theorem}\label{thm-realQtypes}
Let $C=V(Q)$ be a real quartic whose tropicalization is smooth and generic. Then the number of real bitangents with $\Qtype$ equal to $\langle 1 \rangle$ minus the number of real bitangents  with $\Qtype$ equal to $\langle -1 \rangle$ is in
$$\{0,2,4\}.$$
\end{theorem}

Our approach to prove this theorem is via $\GW$-multiplicities of bitangent classes of the tropicalization. It builds on our classification of generic bitangent shapes and their dual motifs in Appendix \ref{app:classification}, on our study of $\GW$-multiplicities of tropical bitangents in Theorem \ref{thm-2H}, \ref{thm-2Hdetails} and Appendix \ref{app-tableQtypes}.
In the following lemma, we use the letters for the bitangent shapes which are introduced in the classification in Appendix \ref{app:classification}. 

\begin{lemma}\label{lem-realQtypes}
Let $V(Q)$ be a real quartic whose tropicalization is smooth and generic. Let $S$ be a liftable tropical bitangent class of $\Trop(V(Q))$. 
Let $s$ denote the number of lifts of $S$ with $\Qtype$ equal to $\langle 1 \rangle$ minus the number of lifts of $S$ with $\Qtype$ equal to $\langle -1 \rangle$.

If $S$ is of shape (Nb), (Ob), (Oc), (Pb), (Qb), (Rb), (Rc), (Sb),  (Ub), (Vb),  (YaI), (Yb), (CCb) or (IIc), then $s= 2$.
If $S$ is of shape  (BBb), then $s=4$.

Otherwise $s=0$.
\end{lemma}

\begin{proof}
This follows by going through the table in Appendix \ref{app-tableQtypes} and only paying attention to the signs.
\end{proof}

\begin{proof}[Proof of Theorem \ref{thm-realQtypes}]
By checking the dual motifs of a tropical bitangent class of one of the shapes with non-zero contribution in Lemma \ref{lem-realQtypes} (and their symmetric cases w.r.t.\ $x-y$-symmetry), we can see that the only ones which are not mutually exclusive are (Qb), (Rb), (Ub), (Vb), resp.\ (YbII), (BBb), (CCb). For each of those, within a symmetry class, it is always the same edges which contain the tropicalization of the tangency points, thus, it is only the lengths of the edges of the tropical curve which decide which of (Qb), (Rb), (Ub) or (Vb) (resp.\ which of  (YbII), (BBb) or (CCb)) shows up, in particular, they do not show up together either. The only possibility is to combine a shape from (Ub) or (Vb) with another such shape after applying $x-y$-symmetry.
Also, the three possibilities for (Nb) do not exclude each other on the level of Newton subdivisions, but the vertex dual to triangle can only align with one of the possible three edges, and so it is precisely one of the types of (Nb) that can occur together. We can see that we can combine at most two shapes that contribute $2$, all other shapes with nonzero contribution appear exclusively. As the maximal nonzero contribution from such a shape is $4$, we obtain the desired result.
\end{proof}

One does not need to rely on our results on $\GW$-multiplicities of tropical bitangent classes from Theorem \ref{thm-2H}, \ref{thm-2Hdetails} and Appendix \ref{app-tableQtypes} to show this result. An alternative approach is to use Viro's patchworking method, and read off the signs of the $\GW$-multiplicities using Figure \ref{fig-geomQtypereals}. We illustrate this approach in an example:

\begin{example}
In Viro's combinatorial patchworking \cite{Viro, IMS09}, we start from a regular subdivision dual to a tropical plane curve $C$ given by a tropical polynomial $F$. We associate signs to each term of the polynomial, resp.\ to each region of $\mathbb{R}^2\setminus C$. These signs represent the signs of the coefficients of a polynomial $f$ over the real Puiseux series tropicalizing to $F$. Consider the tropical plane curve $C$ inside an orthant (the tropicalization of affine space is $(\mathbb{R}\cup\{-\infty\})^2$) and glue four such orthants as usual, coming with the reflected versions of $C$. For the reflected copies of $C$, we add signs to the regions by inserting signs for the $x$ and $y$-coordinates into the corresponding term of $f$ as given by the respective orthant. Then, we take only those edges of the four copies of $C$ that disconnect a positive and a negative region. Viro's patchworking theorem states that the object we obtain in this way (viewed in the real projective plane which we obtain from the four glued orthants by compactifying and identifying boundaries accordingly) is homeomorphic to a real plane curve defined by $f$ after inserting a small value for the parameter $t$.

Figure \ref{fig-patchwork} shows the dual Newton subdivision of a tropicalized quartic, and the four copies of the tropicalized quartic. We pick the signs of all coefficients except $a_{03}$ to be positive, and $s_{03}<0$. The signs in the four copies are disctributed accordingly. With thick black lines, the patchworked quartic is depicted. By Viro's Theorem, real quartics close to the tropical limit are homeomorphic to this picture. We consider the bitangent class of shape (E) which is also depicted in Figure \ref{fig-tropbitangentclass}. It has two liftable members. Their patchworked versions are drawn with thick red lines and denoted $\Lambda_1$ and $\Lambda_2$ in Figure \ref{fig-patchwork}. We can see altogether $4$ tangency points, $p_1,\ldots,p_4$, shown in green in Figure \ref{fig-patchwork}. The two lifts of the tropical bitangent line for $\Lambda_1$ are represented by the same thick red line, but they are tangent at different points: one lift, call it $L_{11}$, is tangent at $p_1$ and $p_3$, the other, $L_{12}$, at $p_1$ and $p_4$. Analogously, one lift, $L_{21}$, of the tropical bitangent for $\Lambda_2$ is tangent at $p_2$ and $p_3$, the other, $L_{22}$, at $p_2$ and $p_4$. If we move along the thick red line for $\Lambda_1$, starting at the top right corner, we meet the first tangency point, $p_1$, on the right, and the second, $p_3$, also on the right. Thus, the $\Qtype$ of $L_{11}$ is $\langle 1 \rangle$, see Figure \ref{fig-geomQtypereals}. For $L_{12}$, we follow the same path, but now we meet $p_1$ on the right and $p_4$ on the left, so the $\Qtype$ of $L_{12}$ is $\langle -1 \rangle$. Analogously, the $\Qtype$ for $L_{21}$ is $\langle -1\rangle$ and for $L_{22}$ $\langle 1\rangle$. With this patchworked picture, we thus confirm the result from Lemma \ref{lem-realQtypes} that the overall contribution of a tropical bitangent class of shape (E) is zero.

\begin{figure}
\begin{center}
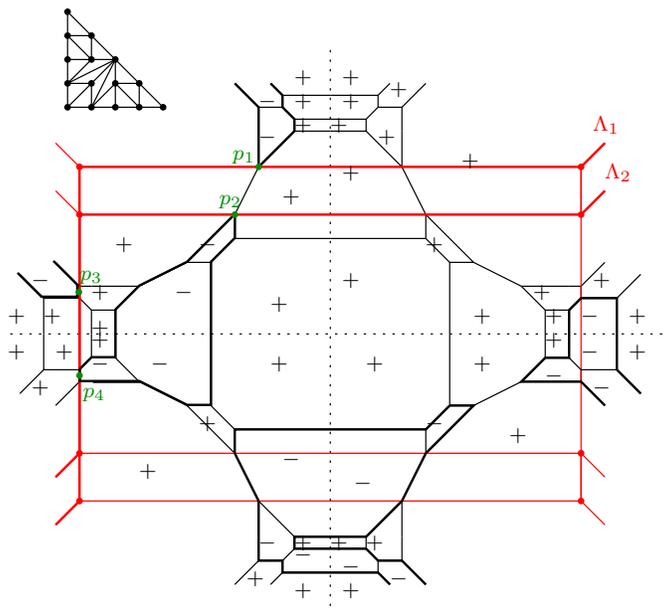
\end{center}
\caption{A patchworked quartic and the lifts of a tropical bitangent class.}\label{fig-patchwork}
\end{figure}

\end{example}

%In the meantime, Kummer and McKean proved Conjecture \ref{con-LV} \cite{KM23}.

%\bibliographystyle{plain}
%\bibliography{bitangents.bib}

\appendix 
\section{Classification, computations and table}

\subsection{Classification of generic bitangent shapes}\label{app:classification}

For the classification of bitangent shapes up to $\mathbb{S}_3$-symmetry, see Figure 6 in \cite{CM20}. There are 41 shapes up to $\mathbb{S}_3$-symmetry, labeled with capital letters such as (A), double capital letters such as (BB) or dashed capital letters such as (T$''$). For bitangent shapes, we follow the color coding of Figure 6 in \cite{CM20}: The black cells of each bitangent class miss the tropicalized quartic, whereas the red ones lie on it. The unfilled dots are vertices of the tropicalized quartic.

In our study, we fix the line $\{z=0\}$ and do not have $\mathbb{S}_3$-symmetry for that reason, only $\mathbb{S}_2$-symmetry for exchanging the variables $x$ and $y$. We label shapes in an  $\mathbb{S}_3/\mathbb{S}_2$-orbit by adding lower case letters such as (Da). Sometimes, the shapes allow different cases of dual motifs in the dual Newton subdivision. These different cases play a role for the computation of Qtypes. We label such different cases of dual motifs by adding Roman numbers such as (TaI).

We consider only generic bitangent shapes, i.e.\ bitangent shapes which appear in the open cones of the subdivided secondary fan, where the subdivision is chosen such that the bitangent shapes are constant in each cell (see Remark \ref{rem-gen}). This subdivision has been computed in \cite{GP21a}.

\begin{lemma}\label{lem-generic}
Let $S$ be a tropical bitangent class of a generic tropicalized quartic. Then, from the classification of possible shapes for $S$ in \cite{CM20}, Figure 6, $S$ can have the shapes (A), (B), (C), (D), (E), (F), (G), (H), (N), (O), (P), (T), (S), (R), (Q), (U), (V), (W), (Y), (BB), (CC), (EE),  or (II).
\end{lemma}
\begin{proof}
This follows since the existence of a tropical bitangent class of one of the remaining shapes ((H$'$), (I), (J), (K), (L$'$), (L), (M), (T$'$), (T$''$), (U$'$), (Q$'$), (X), (Z), (AA), (DD), (FF), (GG), (HH)) requires some equalities on the lengths of the edges of the tropicalized quartic to hold, hence they appear only in lower-dimenional cones of our subdivision of the secondary fan.
\end{proof}

\begin{figure}
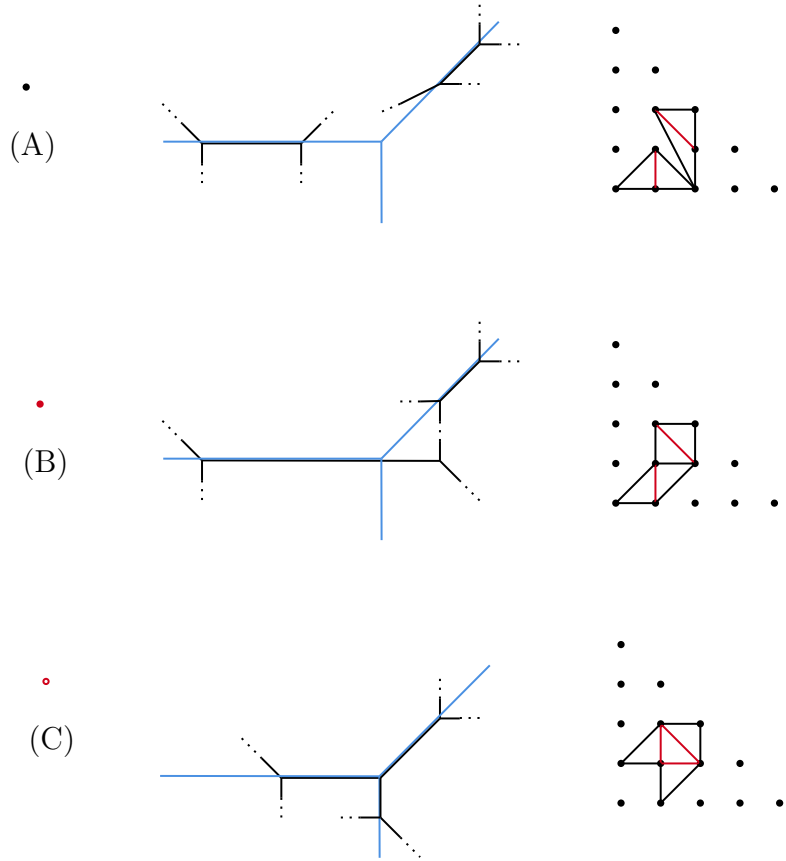


\begin{center}

\tikzset{every picture/.style={line width=0.75pt}} %set default line width to 0.75pt        



\end{center}
\caption{Bitangent shapes, local pictures and motifs of type (A), (B) and (C).} \label{fig:ABC}
\end{figure}

Figure \ref{fig:ABC} shows the three zero-dimensional bitangent shapes (A), (B) and (C). Next to the shape, we sketch the tropical bitangent in blue and parts of the tropicalized quartic which carry the tropicalization of the tangency points. Next, we draw an exemplary dual motif. The red edges are dual to the edges that overlap with the tropicalized quartic. The red edges could also be shifted for case (A) and (B). In case (A) and (B), the vertices on the vertical and diagonal boundary edge of the Newton triangle could also be shifted. In case (C), all vertices on the boundary could be shifted. We refrain from drawing the $\mathbb{S}_3/\mathbb{S}_2$-orbits of these bitangent shapes, as the arguments for the computation for the Qtypes do not change. Note that each such shape has a $\mathbb{S}_3/\mathbb{S}_2$-orbit of size two, as the tropicalization of the tangency points can be on the horizontal and vertical edge, or on the diagonal and horizontal edge.

Figure \ref{fig:D} shows bitangent shape (Da). Shape (Db) and (Dc) are diagonal resp.\ vertical. Figure \ref{fig:D} depicts the liftable tropical bitangents together with the local parts of the tropicalized quartic carrying the tropicalization of the  tangency points, and the dual motifs.

\begin{figure}
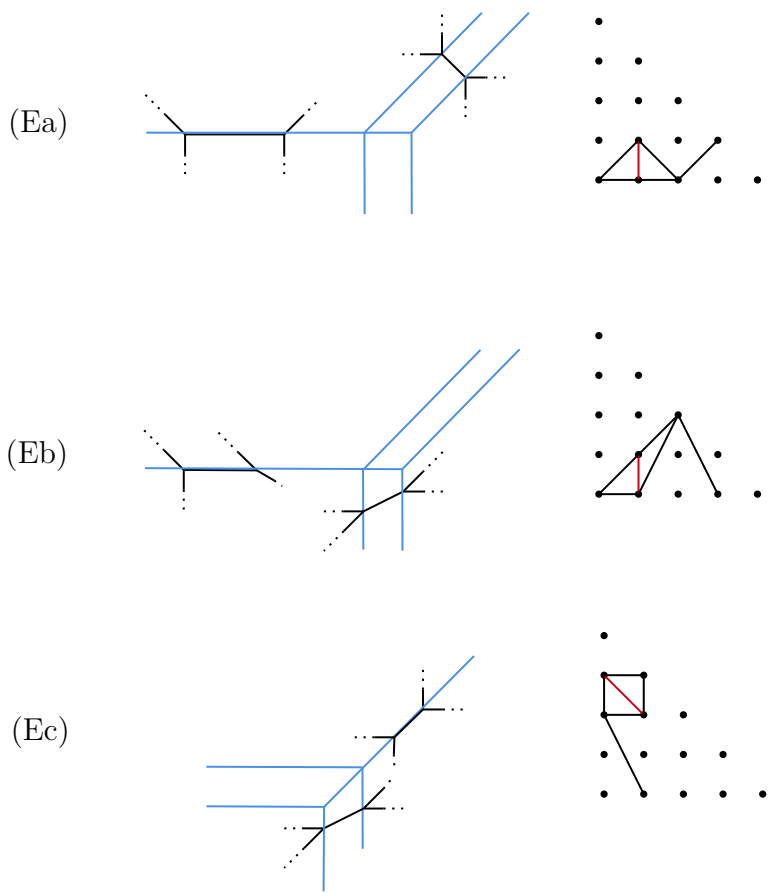


\begin{center}

\tikzset{every picture/.style={line width=0.75pt}} %set default line width to 0.75pt        



\end{center}
\caption{Local pictures of tropicalizations of  tangency points and partial motifs for type (Ea), (Eb) and (Ec).}\label{fig:E}
\end{figure}

Bitangent shape (E) is a segment which does not meet the tropicalized quartic, i.e.\ in Figure 6 \cite{CM20} it is drawn completely in black. The two end points of the segment are the liftable tropical bitangents. Figure \ref{fig:E} shows the behaviour of the tropicalization of the tangency points (i.e.\ local parts of the tropicalized quartic) together with the two liftable tropical bitangents in blue, and the partial dual motifs for the shapes (Ea), (Eb) and (Ec) which we obtain as the $\mathbb{S}_3/\mathbb{S}_2$-orbit of (E). In the dual pictures, the red edge could also be shifted (upwards for (Ea) and (Eb), antidiagonally for (Ec)), as could the vertex which forms the left resp.\ upper triangle with the red edge.

\begin{figure}
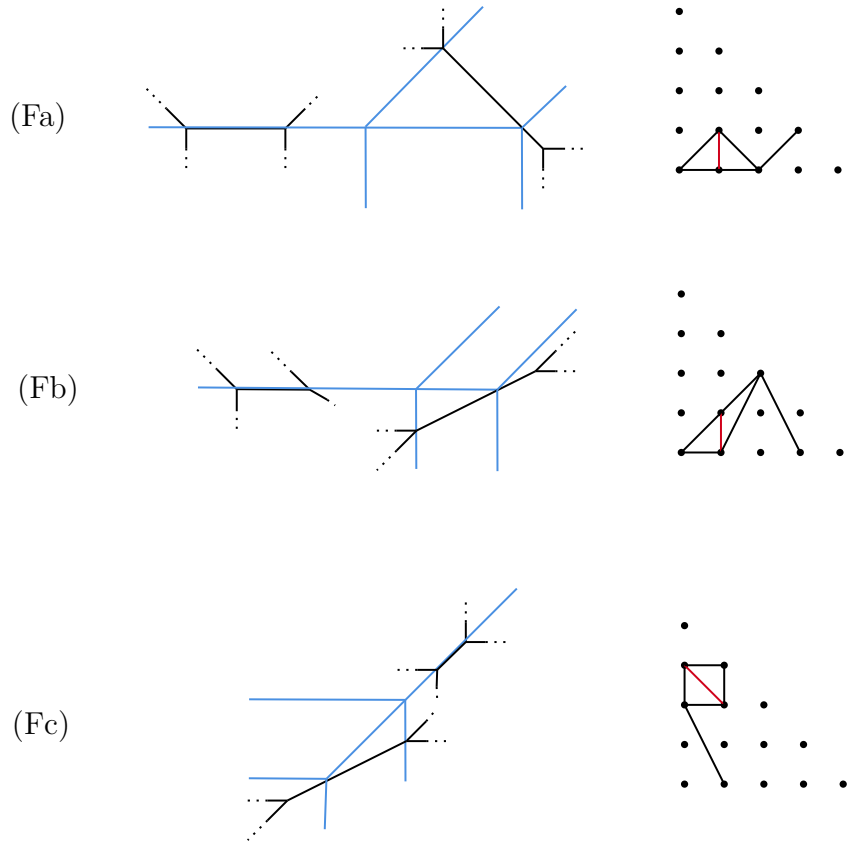

\begin{center}

\tikzset{every picture/.style={line width=0.75pt}} %set default line width to 0.75pt        



\caption{Local pictures of tropicalization of the  tangency points and partial dual motifs for type (Fa), (Fb) and (Fc).}\label{fig:F}
\end{center}
\end{figure}

Bitangent shape (F) is a segment which touches the tropicalized quartic, i.e.\ in the color coding we draw one vertex red. Again, the two end points of the segment are the liftable tropical bitangents. Figure \ref{fig:F} shows the behaviour of the tropicalization of the  tangency points together with the two liftable tropical bitangents in blue, and the partial dual motifs for the shapes (Fa), (Fb) and (Fc) which are the $\mathbb{S}_3/\mathbb{S}_2$-orbit of (F). In the dual pictures, the red edge could also be shifted (upwards for (Fa) and (Fb), antidiagonally for (Fc)), as could the vertex which forms the left resp.\ upper triangle with the red edge.

\begin{figure}
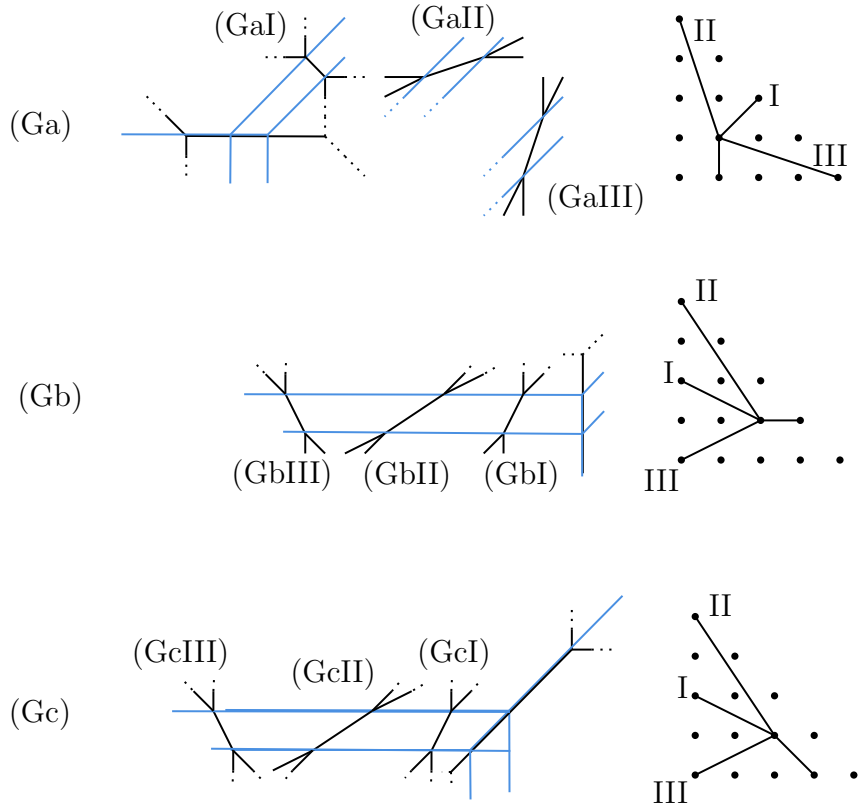

\begin{center}

\tikzset{every picture/.style={line width=0.75pt}} %set default line width to 0.75pt        



\end{center}
\caption{Local pictures of tropical tangency components and partial dual motifs for types (GaI), (GaII), (GaIII), (GbI), (GbII), (GbIII), (GcI), (GcII) and (GcIII).}\label{fig:G}
\end{figure}

Bitangent shape (G) is a segment which is completely contained in the tropicalized quartic. One tropical tangency is an overlap of an edge with the ray spanned by the segment. The two liftable members are the two end points of the segment. As we deform one liftable tropical bitangent to the other, the second tropical tangency component wanders over an edge which can have three possible slopes. We differentiate those three cases by using Roman letters (I), (II) and (III). Figure \ref{fig:G} combines the three possibilities (I), (II), (III) for the three $\mathbb{S}_3/\mathbb{S}_2$-orbits (Ga), (Gb) and (Gc) in one picture each, both for the local pictures of the tropicalization of the tangency points and for the partial dual motifs.

\begin{figure}
\begin{center}

\tikzset{every picture/.style={line width=0.75pt}} %set default line width to 0.75pt        

\begin{tikzpicture}[x=0.75pt,y=0.75pt,yscale=-1,xscale=1]
%uncomment if require: \path (0,784); %set diagram left start at 0, and has height of 784

%Shape: Circle [id:dp21241661257900524] 
\draw  [fill={rgb, 255:red, 0; green, 0; blue, 0 }  ,fill opacity=1 ] (347.36,130.71) .. controls (347.36,129.98) and (347.95,129.39) .. (348.68,129.39) .. controls (349.41,129.39) and (350,129.98) .. (350,130.71) .. controls (350,131.44) and (349.41,132.03) .. (348.68,132.03) .. controls (347.95,132.03) and (347.36,131.44) .. (347.36,130.71) -- cycle ;
%Shape: Circle [id:dp595606131894355] 
\draw  [fill={rgb, 255:red, 0; green, 0; blue, 0 }  ,fill opacity=1 ] (367.36,130.71) .. controls (367.36,129.98) and (367.95,129.39) .. (368.68,129.39) .. controls (369.41,129.39) and (370,129.98) .. (370,130.71) .. controls (370,131.44) and (369.41,132.03) .. (368.68,132.03) .. controls (367.95,132.03) and (367.36,131.44) .. (367.36,130.71) -- cycle ;
%Shape: Circle [id:dp8069903318325975] 
\draw  [fill={rgb, 255:red, 0; green, 0; blue, 0 }  ,fill opacity=1 ] (387.36,130.71) .. controls (387.36,129.98) and (387.95,129.39) .. (388.68,129.39) .. controls (389.41,129.39) and (390,129.98) .. (390,130.71) .. controls (390,131.44) and (389.41,132.03) .. (388.68,132.03) .. controls (387.95,132.03) and (387.36,131.44) .. (387.36,130.71) -- cycle ;
%Shape: Circle [id:dp9951891885617644] 
\draw  [fill={rgb, 255:red, 0; green, 0; blue, 0 }  ,fill opacity=1 ] (407.36,130.71) .. controls (407.36,129.98) and (407.95,129.39) .. (408.68,129.39) .. controls (409.41,129.39) and (410,129.98) .. (410,130.71) .. controls (410,131.44) and (409.41,132.03) .. (408.68,132.03) .. controls (407.95,132.03) and (407.36,131.44) .. (407.36,130.71) -- cycle ;
%Shape: Circle [id:dp864740109796157] 
\draw  [fill={rgb, 255:red, 0; green, 0; blue, 0 }  ,fill opacity=1 ] (427.36,130.71) .. controls (427.36,129.98) and (427.95,129.39) .. (428.68,129.39) .. controls (429.41,129.39) and (430,129.98) .. (430,130.71) .. controls (430,131.44) and (429.41,132.03) .. (428.68,132.03) .. controls (427.95,132.03) and (427.36,131.44) .. (427.36,130.71) -- cycle ;
%Shape: Circle [id:dp1855981519890354] 
\draw  [fill={rgb, 255:red, 0; green, 0; blue, 0 }  ,fill opacity=1 ] (347.36,110.71) .. controls (347.36,109.98) and (347.95,109.39) .. (348.68,109.39) .. controls (349.41,109.39) and (350,109.98) .. (350,110.71) .. controls (350,111.44) and (349.41,112.03) .. (348.68,112.03) .. controls (347.95,112.03) and (347.36,111.44) .. (347.36,110.71) -- cycle ;
%Shape: Circle [id:dp782870462149582] 
\draw  [fill={rgb, 255:red, 0; green, 0; blue, 0 }  ,fill opacity=1 ] (367.36,110.71) .. controls (367.36,109.98) and (367.95,109.39) .. (368.68,109.39) .. controls (369.41,109.39) and (370,109.98) .. (370,110.71) .. controls (370,111.44) and (369.41,112.03) .. (368.68,112.03) .. controls (367.95,112.03) and (367.36,111.44) .. (367.36,110.71) -- cycle ;
%Shape: Circle [id:dp6495294225976462] 
\draw  [fill={rgb, 255:red, 0; green, 0; blue, 0 }  ,fill opacity=1 ] (387.36,110.71) .. controls (387.36,109.98) and (387.95,109.39) .. (388.68,109.39) .. controls (389.41,109.39) and (390,109.98) .. (390,110.71) .. controls (390,111.44) and (389.41,112.03) .. (388.68,112.03) .. controls (387.95,112.03) and (387.36,111.44) .. (387.36,110.71) -- cycle ;
%Shape: Circle [id:dp6336486226161604] 
\draw  [fill={rgb, 255:red, 0; green, 0; blue, 0 }  ,fill opacity=1 ] (407.36,110.71) .. controls (407.36,109.98) and (407.95,109.39) .. (408.68,109.39) .. controls (409.41,109.39) and (410,109.98) .. (410,110.71) .. controls (410,111.44) and (409.41,112.03) .. (408.68,112.03) .. controls (407.95,112.03) and (407.36,111.44) .. (407.36,110.71) -- cycle ;
%Shape: Circle [id:dp9844947228844179] 
\draw  [fill={rgb, 255:red, 0; green, 0; blue, 0 }  ,fill opacity=1 ] (347.36,90.71) .. controls (347.36,89.98) and (347.95,89.39) .. (348.68,89.39) .. controls (349.41,89.39) and (350,89.98) .. (350,90.71) .. controls (350,91.44) and (349.41,92.03) .. (348.68,92.03) .. controls (347.95,92.03) and (347.36,91.44) .. (347.36,90.71) -- cycle ;
%Shape: Circle [id:dp28896942970854556] 
\draw  [fill={rgb, 255:red, 0; green, 0; blue, 0 }  ,fill opacity=1 ] (367.36,90.71) .. controls (367.36,89.98) and (367.95,89.39) .. (368.68,89.39) .. controls (369.41,89.39) and (370,89.98) .. (370,90.71) .. controls (370,91.44) and (369.41,92.03) .. (368.68,92.03) .. controls (367.95,92.03) and (367.36,91.44) .. (367.36,90.71) -- cycle ;
%Shape: Circle [id:dp687120703935955] 
\draw  [fill={rgb, 255:red, 0; green, 0; blue, 0 }  ,fill opacity=1 ] (387.36,90.71) .. controls (387.36,89.98) and (387.95,89.39) .. (388.68,89.39) .. controls (389.41,89.39) and (390,89.98) .. (390,90.71) .. controls (390,91.44) and (389.41,92.03) .. (388.68,92.03) .. controls (387.95,92.03) and (387.36,91.44) .. (387.36,90.71) -- cycle ;
%Shape: Circle [id:dp548558741903369] 
\draw  [fill={rgb, 255:red, 0; green, 0; blue, 0 }  ,fill opacity=1 ] (347.36,70.71) .. controls (347.36,69.98) and (347.95,69.39) .. (348.68,69.39) .. controls (349.41,69.39) and (350,69.98) .. (350,70.71) .. controls (350,71.44) and (349.41,72.03) .. (348.68,72.03) .. controls (347.95,72.03) and (347.36,71.44) .. (347.36,70.71) -- cycle ;
%Shape: Circle [id:dp8207137702228025] 
\draw  [fill={rgb, 255:red, 0; green, 0; blue, 0 }  ,fill opacity=1 ] (367.36,70.71) .. controls (367.36,69.98) and (367.95,69.39) .. (368.68,69.39) .. controls (369.41,69.39) and (370,69.98) .. (370,70.71) .. controls (370,71.44) and (369.41,72.03) .. (368.68,72.03) .. controls (367.95,72.03) and (367.36,71.44) .. (367.36,70.71) -- cycle ;
%Shape: Circle [id:dp45384030511774065] 
\draw  [fill={rgb, 255:red, 0; green, 0; blue, 0 }  ,fill opacity=1 ] (347.36,50.71) .. controls (347.36,49.98) and (347.95,49.39) .. (348.68,49.39) .. controls (349.41,49.39) and (350,49.98) .. (350,50.71) .. controls (350,51.44) and (349.41,52.03) .. (348.68,52.03) .. controls (347.95,52.03) and (347.36,51.44) .. (347.36,50.71) -- cycle ;
%Shape: Circle [id:dp47228429719324017] 
\draw  [fill={rgb, 255:red, 0; green, 0; blue, 0 }  ,fill opacity=1 ] (348.36,273.32) .. controls (348.36,272.59) and (348.95,272) .. (349.68,272) .. controls (350.41,272) and (351,272.59) .. (351,273.32) .. controls (351,274.05) and (350.41,274.64) .. (349.68,274.64) .. controls (348.95,274.64) and (348.36,274.05) .. (348.36,273.32) -- cycle ;
%Shape: Circle [id:dp15658232834990027] 
\draw  [fill={rgb, 255:red, 0; green, 0; blue, 0 }  ,fill opacity=1 ] (368.36,273.32) .. controls (368.36,272.59) and (368.95,272) .. (369.68,272) .. controls (370.41,272) and (371,272.59) .. (371,273.32) .. controls (371,274.05) and (370.41,274.64) .. (369.68,274.64) .. controls (368.95,274.64) and (368.36,274.05) .. (368.36,273.32) -- cycle ;
%Shape: Circle [id:dp28420915399100743] 
\draw  [fill={rgb, 255:red, 0; green, 0; blue, 0 }  ,fill opacity=1 ] (388.36,273.32) .. controls (388.36,272.59) and (388.95,272) .. (389.68,272) .. controls (390.41,272) and (391,272.59) .. (391,273.32) .. controls (391,274.05) and (390.41,274.64) .. (389.68,274.64) .. controls (388.95,274.64) and (388.36,274.05) .. (388.36,273.32) -- cycle ;
%Shape: Circle [id:dp41657182534181625] 
\draw  [fill={rgb, 255:red, 0; green, 0; blue, 0 }  ,fill opacity=1 ] (408.36,273.32) .. controls (408.36,272.59) and (408.95,272) .. (409.68,272) .. controls (410.41,272) and (411,272.59) .. (411,273.32) .. controls (411,274.05) and (410.41,274.64) .. (409.68,274.64) .. controls (408.95,274.64) and (408.36,274.05) .. (408.36,273.32) -- cycle ;
%Shape: Circle [id:dp8273194219012512] 
\draw  [fill={rgb, 255:red, 0; green, 0; blue, 0 }  ,fill opacity=1 ] (428.36,273.32) .. controls (428.36,272.59) and (428.95,272) .. (429.68,272) .. controls (430.41,272) and (431,272.59) .. (431,273.32) .. controls (431,274.05) and (430.41,274.64) .. (429.68,274.64) .. controls (428.95,274.64) and (428.36,274.05) .. (428.36,273.32) -- cycle ;
%Shape: Circle [id:dp15189767003405308] 
\draw  [fill={rgb, 255:red, 0; green, 0; blue, 0 }  ,fill opacity=1 ] (348.36,253.32) .. controls (348.36,252.59) and (348.95,252) .. (349.68,252) .. controls (350.41,252) and (351,252.59) .. (351,253.32) .. controls (351,254.05) and (350.41,254.64) .. (349.68,254.64) .. controls (348.95,254.64) and (348.36,254.05) .. (348.36,253.32) -- cycle ;
%Shape: Circle [id:dp008939418710347202] 
\draw  [fill={rgb, 255:red, 0; green, 0; blue, 0 }  ,fill opacity=1 ] (368.36,253.32) .. controls (368.36,252.59) and (368.95,252) .. (369.68,252) .. controls (370.41,252) and (371,252.59) .. (371,253.32) .. controls (371,254.05) and (370.41,254.64) .. (369.68,254.64) .. controls (368.95,254.64) and (368.36,254.05) .. (368.36,253.32) -- cycle ;
%Shape: Circle [id:dp8578611022286349] 
\draw  [fill={rgb, 255:red, 0; green, 0; blue, 0 }  ,fill opacity=1 ] (388.36,253.32) .. controls (388.36,252.59) and (388.95,252) .. (389.68,252) .. controls (390.41,252) and (391,252.59) .. (391,253.32) .. controls (391,254.05) and (390.41,254.64) .. (389.68,254.64) .. controls (388.95,254.64) and (388.36,254.05) .. (388.36,253.32) -- cycle ;
%Shape: Circle [id:dp8076806802493605] 
\draw  [fill={rgb, 255:red, 0; green, 0; blue, 0 }  ,fill opacity=1 ] (408.36,253.32) .. controls (408.36,252.59) and (408.95,252) .. (409.68,252) .. controls (410.41,252) and (411,252.59) .. (411,253.32) .. controls (411,254.05) and (410.41,254.64) .. (409.68,254.64) .. controls (408.95,254.64) and (408.36,254.05) .. (408.36,253.32) -- cycle ;
%Shape: Circle [id:dp4292874014597736] 
\draw  [fill={rgb, 255:red, 0; green, 0; blue, 0 }  ,fill opacity=1 ] (348.36,233.32) .. controls (348.36,232.59) and (348.95,232) .. (349.68,232) .. controls (350.41,232) and (351,232.59) .. (351,233.32) .. controls (351,234.05) and (350.41,234.64) .. (349.68,234.64) .. controls (348.95,234.64) and (348.36,234.05) .. (348.36,233.32) -- cycle ;
%Shape: Circle [id:dp2514501826590426] 
\draw  [fill={rgb, 255:red, 0; green, 0; blue, 0 }  ,fill opacity=1 ] (368.36,233.32) .. controls (368.36,232.59) and (368.95,232) .. (369.68,232) .. controls (370.41,232) and (371,232.59) .. (371,233.32) .. controls (371,234.05) and (370.41,234.64) .. (369.68,234.64) .. controls (368.95,234.64) and (368.36,234.05) .. (368.36,233.32) -- cycle ;
%Shape: Circle [id:dp5892451824044421] 
\draw  [fill={rgb, 255:red, 0; green, 0; blue, 0 }  ,fill opacity=1 ] (388.36,233.32) .. controls (388.36,232.59) and (388.95,232) .. (389.68,232) .. controls (390.41,232) and (391,232.59) .. (391,233.32) .. controls (391,234.05) and (390.41,234.64) .. (389.68,234.64) .. controls (388.95,234.64) and (388.36,234.05) .. (388.36,233.32) -- cycle ;
%Shape: Circle [id:dp28822006691276103] 
\draw  [fill={rgb, 255:red, 0; green, 0; blue, 0 }  ,fill opacity=1 ] (348.36,213.32) .. controls (348.36,212.59) and (348.95,212) .. (349.68,212) .. controls (350.41,212) and (351,212.59) .. (351,213.32) .. controls (351,214.05) and (350.41,214.64) .. (349.68,214.64) .. controls (348.95,214.64) and (348.36,214.05) .. (348.36,213.32) -- cycle ;
%Shape: Circle [id:dp5911143386286732] 
\draw  [fill={rgb, 255:red, 0; green, 0; blue, 0 }  ,fill opacity=1 ] (368.36,213.32) .. controls (368.36,212.59) and (368.95,212) .. (369.68,212) .. controls (370.41,212) and (371,212.59) .. (371,213.32) .. controls (371,214.05) and (370.41,214.64) .. (369.68,214.64) .. controls (368.95,214.64) and (368.36,214.05) .. (368.36,213.32) -- cycle ;
%Shape: Circle [id:dp37910190180298675] 
\draw  [fill={rgb, 255:red, 0; green, 0; blue, 0 }  ,fill opacity=1 ] (348.36,193.32) .. controls (348.36,192.59) and (348.95,192) .. (349.68,192) .. controls (350.41,192) and (351,192.59) .. (351,193.32) .. controls (351,194.05) and (350.41,194.64) .. (349.68,194.64) .. controls (348.95,194.64) and (348.36,194.05) .. (348.36,193.32) -- cycle ;
%Straight Lines [id:da06572683827898784] 
\draw  [dash pattern={on 0.84pt off 2.51pt}]  (126.67,122.67) -- (126.67,112.67) ;
%Straight Lines [id:da5659363890302795] 
\draw  [dash pattern={on 0.84pt off 2.51pt}]  (156.67,112.67) -- (156.67,122.67) ;
%Straight Lines [id:da19652833260714653] 
\draw    (126.67,102.67) -- (156.67,102.67) ;
%Straight Lines [id:da9534240407257167] 
\draw [color={rgb, 255:red, 74; green, 144; blue, 226 }  ,draw opacity=1 ]   (122.31,102.01) -- (222.51,101.9) ;
%Straight Lines [id:da4208057349794393] 
\draw    (216.67,112.67) -- (236.67,72.67) ;
%Straight Lines [id:da9558765202559439] 
\draw    (216.67,112.67) -- (206.67,122.67) ;
%Straight Lines [id:da9333277800717543] 
\draw    (216.67,112.67) -- (216.67,122.67) ;
%Straight Lines [id:da4680017746224673] 
\draw  [dash pattern={on 0.84pt off 2.51pt}]  (206.67,122.67) -- (196.67,132.67) ;
%Straight Lines [id:da7456242154267548] 
\draw  [dash pattern={on 0.84pt off 2.51pt}]  (216.67,132.67) -- (216.67,122.67) ;
%Straight Lines [id:da37484670675003173] 
\draw  [dash pattern={on 0.84pt off 2.51pt}]  (246.67,52.67) -- (236.67,72.67) ;
%Straight Lines [id:da4939548219315889] 
\draw [color={rgb, 255:red, 0; green, 0; blue, 0 }  ,draw opacity=1 ]   (388.68,110.71) -- (428.68,130.71) ;
%Straight Lines [id:da9110744668711243] 
\draw [color={rgb, 255:red, 0; green, 0; blue, 0 }  ,draw opacity=1 ]   (369.68,233.32) -- (369.68,253.32) ;
%Straight Lines [id:da42495466068572996] 
\draw [color={rgb, 255:red, 208; green, 2; blue, 27 }  ,draw opacity=1 ]   (389.68,253.32) -- (369.68,233.32) ;
%Straight Lines [id:da279315655982034] 
\draw [color={rgb, 255:red, 0; green, 0; blue, 0 }  ,draw opacity=1 ]   (369.68,253.32) -- (349.68,273.32) ;
%Straight Lines [id:da7780959517815836] 
\draw    (156.67,102.67) -- (156.67,112.67) ;
%Straight Lines [id:da09683577095448392] 
\draw  [dash pattern={on 0.84pt off 2.51pt}]  (176.67,82.67) -- (166.67,92.67) ;
%Straight Lines [id:da1861092563003034] 
\draw  [dash pattern={on 0.84pt off 2.51pt}]  (225.33,204) -- (215.33,204) ;
%Straight Lines [id:da3096833063258191] 
\draw    (135.33,234) -- (175.33,274) ;
%Straight Lines [id:da9466779934077147] 
\draw    (205.33,204) -- (185.33,224) ;
%Straight Lines [id:da08079126197585274] 
\draw [color={rgb, 255:red, 74; green, 144; blue, 226 }  ,draw opacity=1 ]   (243.74,81.9) -- (222.51,101.9) ;
%Straight Lines [id:da8653438300207696] 
\draw [color={rgb, 255:red, 74; green, 144; blue, 226 }  ,draw opacity=1 ]   (156.24,254.18) -- (156.33,268.1) ;
%Straight Lines [id:da6010554901621515] 
\draw    (205.33,194) -- (205.33,204) ;
%Straight Lines [id:da9256295365652518] 
\draw    (185.33,224) -- (185.33,234) ;
%Straight Lines [id:da154278049889363] 
\draw [color={rgb, 255:red, 74; green, 144; blue, 226 }  ,draw opacity=1 ]   (135.97,254.28) -- (156.24,254.18) ;
%Straight Lines [id:da18547009684523563] 
\draw [color={rgb, 255:red, 74; green, 144; blue, 226 }  ,draw opacity=1 ]   (218.15,192.1) -- (156.24,254.18) ;
%Straight Lines [id:da09700519615261705] 
\draw  [dash pattern={on 0.84pt off 2.51pt}]  (125.33,224) -- (135.33,234) ;
%Straight Lines [id:da33931291066818836] 
\draw  [dash pattern={on 0.84pt off 2.51pt}]  (165.33,224) -- (175.33,224) ;
%Straight Lines [id:da5332693320091259] 
\draw  [dash pattern={on 0.84pt off 2.51pt}]  (205.61,188.38) -- (205.33,194) ;
%Straight Lines [id:da8900449135784287] 
\draw  [dash pattern={on 0.84pt off 2.51pt}]  (175.33,274) -- (179.79,278.72) ;
%Straight Lines [id:da6671983934613184] 
\draw [color={rgb, 255:red, 0; green, 0; blue, 0 }  ,draw opacity=1 ]   (348.68,130.71) -- (368.68,110.71) ;
%Straight Lines [id:da06996598487475325] 
\draw [color={rgb, 255:red, 0; green, 0; blue, 0 }  ,draw opacity=1 ]   (389.68,233.32) -- (389.68,253.32) ;
%Straight Lines [id:da8992147569384877] 
\draw [color={rgb, 255:red, 74; green, 144; blue, 226 }  ,draw opacity=1 ]   (222.51,101.9) -- (222.51,116.76) ;
%Straight Lines [id:da4720685038650061] 
\draw [color={rgb, 255:red, 0; green, 0; blue, 0 }  ,draw opacity=1 ]   (348.68,130.71) -- (368.68,130.71) ;
%Straight Lines [id:da3742425034870265] 
\draw [color={rgb, 255:red, 0; green, 0; blue, 0 }  ,draw opacity=1 ]   (368.68,110.71) -- (388.68,110.71) ;
%Straight Lines [id:da46554527377849864] 
\draw [color={rgb, 255:red, 208; green, 2; blue, 27 }  ,draw opacity=1 ]   (368.68,130.71) -- (368.68,110.71) ;
%Straight Lines [id:da5460659791351182] 
\draw [color={rgb, 255:red, 0; green, 0; blue, 0 }  ,draw opacity=1 ]   (368.68,130.71) -- (388.68,110.71) ;
%Straight Lines [id:da19646378694814504] 
\draw [color={rgb, 255:red, 0; green, 0; blue, 0 }  ,draw opacity=1 ]   (369.68,233.32) -- (389.68,233.32) ;
%Straight Lines [id:da3516363492678055] 
\draw [color={rgb, 255:red, 0; green, 0; blue, 0 }  ,draw opacity=1 ]   (369.68,253.32) -- (389.68,253.32) ;
%Straight Lines [id:da5750813920443597] 
\draw    (116.67,92.67) -- (126.67,102.67) ;
%Straight Lines [id:da4837966927990671] 
\draw    (166.67,92.67) -- (156.67,102.67) ;
%Straight Lines [id:da37993511569194605] 
\draw    (126.67,102.67) -- (126.67,112.67) ;
%Straight Lines [id:da6898898905466998] 
\draw  [dash pattern={on 0.84pt off 2.51pt}]  (106.67,82.67) -- (116.67,92.67) ;
%Straight Lines [id:da44611868727001525] 
\draw  [dash pattern={on 0.84pt off 2.51pt}]  (185.27,240.33) -- (185.33,234) ;
%Straight Lines [id:da6941753685311391] 
\draw    (205.33,204) -- (215.33,204) ;
%Straight Lines [id:da8436425878173417] 
\draw    (175.33,224) -- (185.33,224) ;
%Straight Lines [id:da50543216030067] 
\draw  [dash pattern={on 0.84pt off 2.51pt}]  (145.33,244) -- (155.33,254) ;
%Straight Lines [id:da13803284313778563] 
\draw  [dash pattern={on 0.84pt off 2.51pt}]  (145.33,244) -- (155.33,254) ;

% Text Node
\draw (403.36,148) node [anchor=north west][inner sep=0.75pt]   [align=left] {$ $};
% Text Node
\draw (404.36,290.61) node [anchor=north west][inner sep=0.75pt]   [align=left] {$ $};
% Text Node
\draw (8.69,228.9) node [anchor=north west][inner sep=0.75pt]   [align=left] {(Hb)};
% Text Node
\draw (7.24,95.45) node [anchor=north west][inner sep=0.75pt]   [align=left] {(Ha)};

\end{tikzpicture}

\end{center}

\caption{Local pictures for the tropical tangency components and partial dual motifs for shapes (Ha), (Hb).}\label{fig:H}

\end{figure}
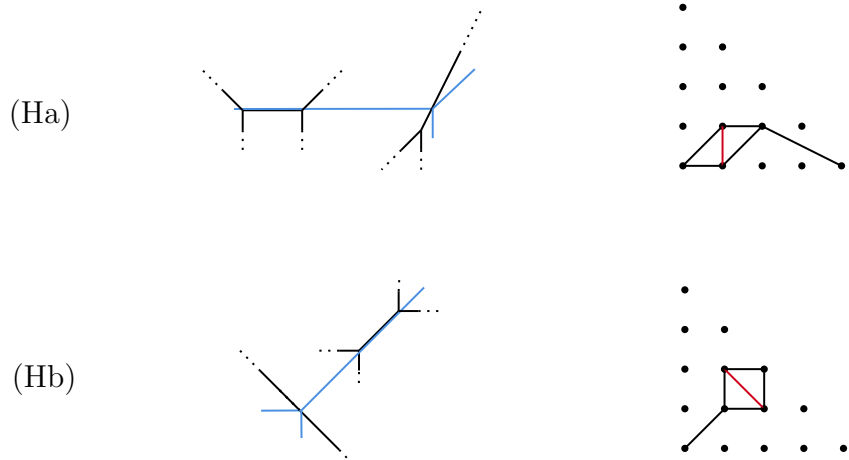

Bitangent shape (H) is a ray whose vertex corresponds to the only liftable tropical bitangent in the class. Figure \ref{fig:H} shows the local pictures of the tropicalization of the tangency points and the partial dual motifs for the two cases (Ha) and (Hb). The red edge could be shifted upwards resp.\ antidiagonally, as could the vertex of its adjacent right resp.\ upper triangle.

\begin{figure}
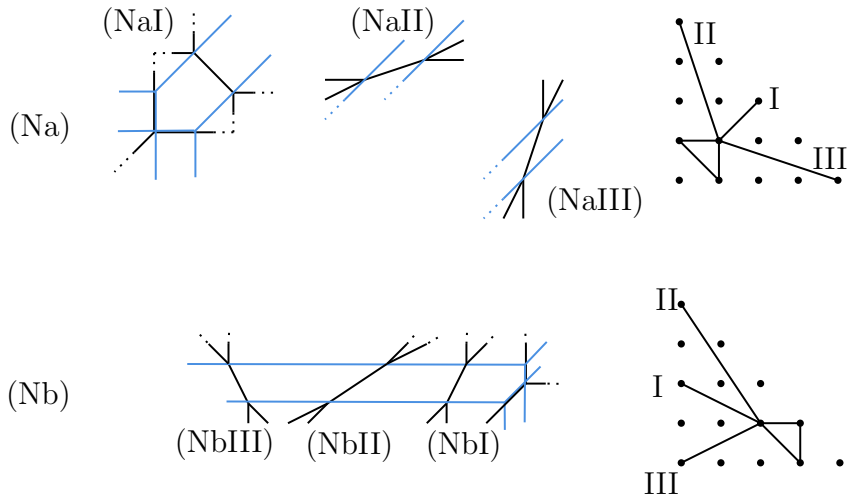

\begin{center}

\tikzset{every picture/.style={line width=0.75pt}} %set default line width to 0.75pt        



\caption{Local pictures of tropical tangency components and partial dual motifs for types (NaI), (NaII), (NaIII), (NbI), (NbII), (NbIII).}\label{fig:N}
\end{center}
\end{figure}

Bitangent shape (N) looks like a  reversed tropical line with two rays cut off. The end points of the two cut rays are the two liftable members. As for shape (G) (see Figure \ref{fig:G}), one tropical tangency is an overlap with an edge (although the two overlapping edges differ for the two liftable tropical bitangents, different from the situation in (G)), and one tropical tangency component wanders over an edge --- which can have three different possible slopes (I), (II) and (III) --- as we deform one liftable tropical tangency component to the other. Figure \ref{fig:N} shows local pictures of the tropicalization of the  tangency points and partial dual motifs, and combines the three options (I), (II) and (III) in one picture each for the two $\mathbb{S}_3/\mathbb{S}_2$-orbits (Na) and (Nb).

\begin{figure}
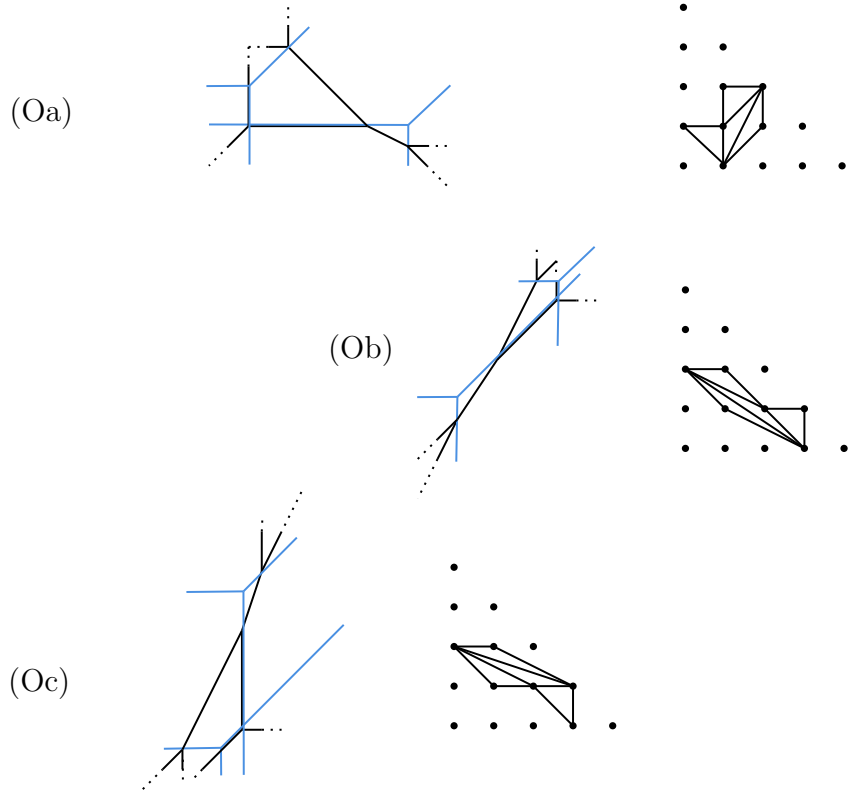

\begin{center}

\tikzset{every picture/.style={line width=0.75pt}} %set default line width to 0.75pt        



\end{center}
\caption{Local pictures for tropical tangency components and dual motifs for types (Oa), (Ob) and (Oc).}\label{fig:O}

\end{figure}

Shape (O) has two liftable tropical bitangents, the local pictures and dual motifs for the three symmetric cases (Oa), (Ob) and (Oc) are depicted in \ref{fig:O}.

\begin{figure}
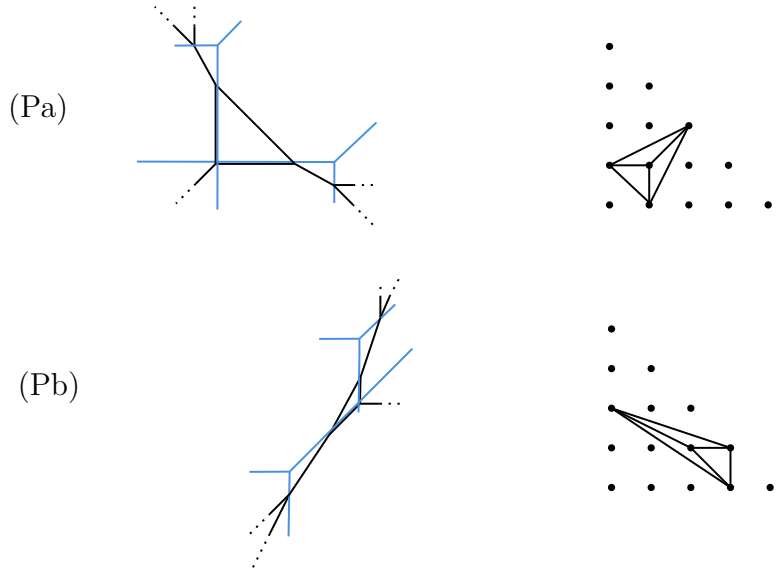

\begin{center}

\tikzset{every picture/.style={line width=0.75pt}} %set default line width to 0.75pt        

%

\end{center}
\caption{Local pictures of the tropical tangency components and dual motifs for shapes (Pa) and (Pb).}\label{fig:P}
\end{figure}

Shape (P) has two liftable tropical bitangents. Figure \ref{fig:P} shows local pictures of the tropical tangency components and dual motifs for the two symmetric cases (Pa) and (Pb).

\begin{figure}
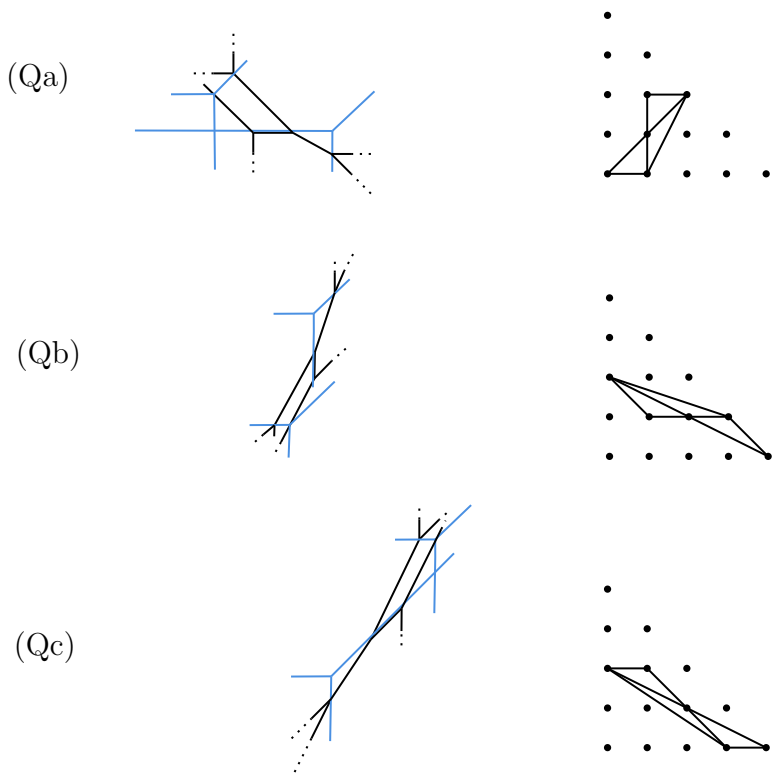

\begin{center}

\tikzset{every picture/.style={line width=0.75pt}} %set default line width to 0.75pt        

%

\end{center}
\caption{Local pictures of tropical tangency components and dual motifs for shapes (Qa), (Qb) and (Qc).}\label{fig:Q}
\end{figure}

Shape (Q) has two liftable tropical bitangents. Figure \ref{fig:Q} shows the local pictures of the tropical  tangency components and the dual motifs for the three symmetric cases (Qa), (Qb) and (Qc).

\begin{figure}
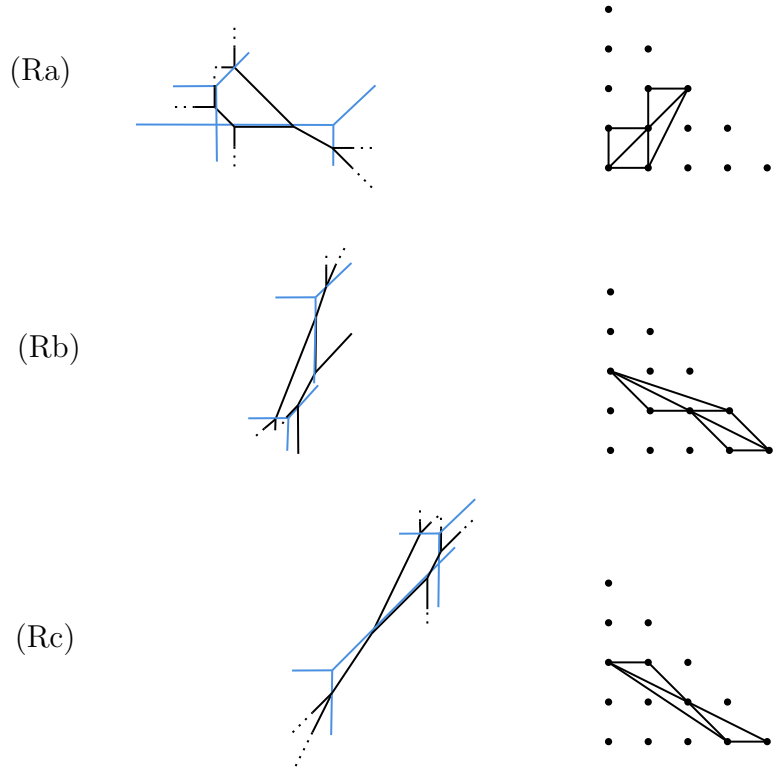

\begin{center}

\tikzset{every picture/.style={line width=0.75pt}} %set default line width to 0.75pt        

%

\end{center}
\caption{Local tangencies and dual motifs for shapes (Ra), (Rb) and (Rc).}\label{fig:R}
\end{figure}

Figure \ref{fig:R} shows the local tropical tangency components and dual motifs for the three symmetric cases (Ra), (Rb) and (Rc) that we have for shape (R). Each has two liftable tropical bitangents.

\begin{figure}
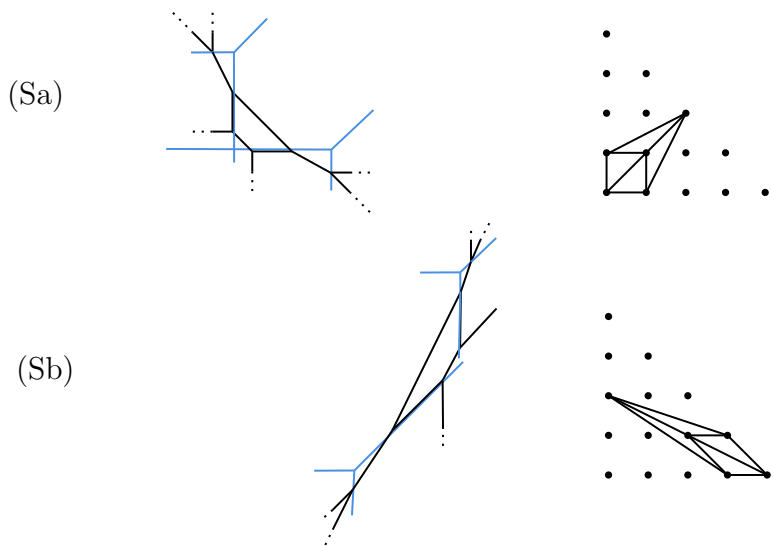

\begin{center}

\tikzset{every picture/.style={line width=0.75pt}} %set default line width to 0.75pt        

%

\end{center}
\caption{Local tangencies of liftable tropical bitangents and dual motifs for shapes (Sa) and (Sb).}\label{fig:S}
\end{figure}

In shape (S), there are two liftable tropical bitangents, shown in Figure \ref{fig:S} together with the dual motifs for the two symmetric cases (Sa) and (Sb).

\begin{figure}
\begin{center}

\tikzset{every picture/.style={line width=0.75pt}} %set default line width to 0.75pt        

%

\end{center}

\caption{Local pictures of tropical tangency components and partial dual motifs for shapes (TaI0, (TaII), (TaIII), (TbI), (TbII) and (TbIII).}\label{fig:T}

\end{figure}

Shape (T) is an unbounded 2-dimensional cell show two vertices each lift with multiplicity two. As for shape (G) and (N), there are three possibilities for the slope of the edge along which one of the tropical tangency component wanders, which we denote by (I), (II) and (III). Figure \ref{fig:T} combines these possibilities into one picture for (Ta) and one for (Tb), which are the two orbits we have to take into account.

\begin{figure}
\begin{center}

\tikzset{every picture/.style={line width=0.75pt}} %set default line width to 0.75pt        



\caption{Local pictures of tropical tangency components and partial dual motifs for (UaI), (UaII), (UaIII), (UbI), (UbII), (UbIII), (UcI0, (UcII) and (UcIII).}\label{fig:U}
\end{center}
\end{figure}

Similar to shape (G), (N) and (T), shape (U) comes with three different possibilities for slopes for the wandering tropical tangency component, see Figure \ref{fig:U}.

\begin{figure}
\begin{center}

\tikzset{every picture/.style={line width=0.75pt}} %set default line width to 0.75pt        

%

\end{center}
\caption{Local pictures of tropical tangency components and partial dual motifs for shapes (VaI0, (VaII), (VaIII), (VbI), (VbII) and (VbIII).}\label{fig:V}
\end{figure}

The behaviour of the tropicalization of the  tangency points for shape (V) is the same as for shape (N), it is only the possible deformations that differ, but this has no effect on Qtypes, see Figure \ref{fig:V}.

We now consider the bitangent shapes (W), (Y), (BB), (CC) and (EE) which each have four liftable members. For some of these shapes, tangencies can arise due to edges for which different slopes are possible. We label these different slopes again using Roman numbers (I), (II), (III).
Bitangent shape (W) is a parallelogram whose four vertices are the liftable members. The other shapes can be viewed as deformations of the parallelogram, where edges of the tropicalized quartic ``cut off'' pieces.
Figure \ref{fig:WI} shows the local pieces of the tropicalized quartic together with the four liftable tropical bitangents in blue for the cases (WaI) and (WbI). Next to it, the partial dual motifs are depicted. For (WaII), (WaIII), (WbIII) and (WcII), the partial dual motifs are depicted in Figure \ref{fig:W}. As the local pictures of the tropical tangency components and liftable tropical bitangents are similar to the ones shown in Figure \ref{fig:WI}, we do not include those.

\begin{figure}
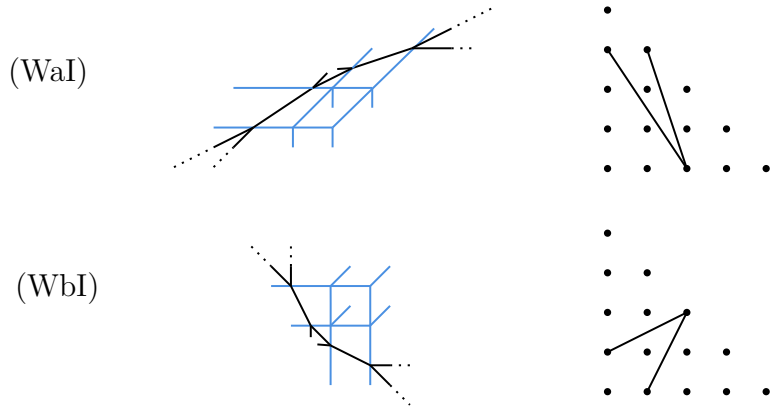

\begin{center}

\tikzset{every picture/.style={line width=0.75pt}} %set default line width to 0.75pt        



\end{center}
\caption{Local pictures of tropical tangency components, liftable tropical bitangents and partial dual motifs for shapes (WaI) and (WbI).}\label{fig:WI}
\end{figure}

\begin{figure}
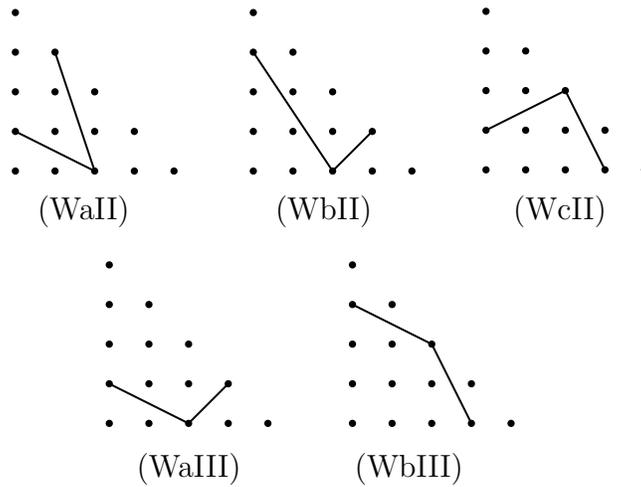

\begin{center}

\tikzset{every picture/.style={line width=0.75pt}} %set default line width to 0.75pt        

%

\end{center}

\caption{Partial dual motifs for the shapes (WaII), (WbII), (WcII), (WaIII), (WBIII). }\label{fig:W}
\end{figure}

Concerning shape (Y), the dual motif for (YaI) equals the one for (WbII), (YbI) equals (WcII), (YcI) equals (WaII), (YaII) and (YaIII) equals (WaIII), and (YbII) equals (WbIII).  The shape (Y) only differs from (W) by ``cutting off'' a vertex of the parallelogram --- one liftable tropical bitangent now has its vertex on an edge of the tropicalized quartic. Pictures of the partial dual  motifs can thus be found in Figure \ref{fig:W}. Figure \ref{fig:Y} depicts the $4$ liftable tropical bitangents in each case together with a local picture of the tropicalized quartic.

\begin{figure}
\begin{center}

\tikzset{every picture/.style={line width=0.75pt}} %set default line width to 0.75pt        

\begin{tikzpicture}[x=0.75pt,y=0.75pt,yscale=-1,xscale=1]
%uncomment if require: \path (0,784); %set diagram left start at 0, and has height of 784

%Straight Lines [id:da4844215818870313] 
\draw [color={rgb, 255:red, 0; green, 0; blue, 0 }  ,draw opacity=1 ]   (190,110) -- (160,80) ;
%Straight Lines [id:da12554032030916007] 
\draw [color={rgb, 255:red, 0; green, 0; blue, 0 }  ,draw opacity=1 ]   (190,120) -- (190,110) ;
%Straight Lines [id:da4359851805948435] 
\draw [color={rgb, 255:red, 0; green, 0; blue, 0 }  ,draw opacity=1 ]   (200,110) -- (190,110) ;
%Straight Lines [id:da13996763859544692] 
\draw [color={rgb, 255:red, 0; green, 0; blue, 0 }  ,draw opacity=1 ]   (160,80) -- (150,80) ;
%Straight Lines [id:da7277472227092766] 
\draw [color={rgb, 255:red, 0; green, 0; blue, 0 }  ,draw opacity=1 ]   (160,80) -- (160,70) ;
%Straight Lines [id:da592152558439977] 
\draw [color={rgb, 255:red, 0; green, 0; blue, 0 }  ,draw opacity=1 ]   (110,100) -- (130,90) ;
%Straight Lines [id:da5435014520999334] 
\draw [color={rgb, 255:red, 0; green, 0; blue, 0 }  ,draw opacity=1 ]   (110,100) -- (120,90) ;
%Straight Lines [id:da6220477960514808] 
\draw [color={rgb, 255:red, 0; green, 0; blue, 0 }  ,draw opacity=1 ]   (80,120) -- (110,100) ;
%Straight Lines [id:da47914995783450975] 
\draw [color={rgb, 255:red, 0; green, 0; blue, 0 }  ,draw opacity=1 ]   (70,130) -- (80,120) ;
%Straight Lines [id:da20620685137628303] 
\draw [color={rgb, 255:red, 0; green, 0; blue, 0 }  ,draw opacity=1 ]   (60,130) -- (80,120) ;
%Straight Lines [id:da9468606897152186] 
\draw [color={rgb, 255:red, 74; green, 144; blue, 226 }  ,draw opacity=1 ]   (90,100) -- (180,100) ;
%Straight Lines [id:da22333164782127357] 
\draw [color={rgb, 255:red, 74; green, 144; blue, 226 }  ,draw opacity=1 ]   (120,120) -- (170,70) ;
%Straight Lines [id:da9243799011816428] 
\draw [color={rgb, 255:red, 74; green, 144; blue, 226 }  ,draw opacity=1 ]   (140,110) -- (140,100) ;
%Straight Lines [id:da5824160892856032] 
\draw [color={rgb, 255:red, 74; green, 144; blue, 226 }  ,draw opacity=1 ]   (70,120) -- (180,120) ;
%Straight Lines [id:da9403144126510139] 
\draw [color={rgb, 255:red, 74; green, 144; blue, 226 }  ,draw opacity=1 ]   (120,130) -- (120,120) ;
%Straight Lines [id:da7526368259301713] 
\draw [color={rgb, 255:red, 74; green, 144; blue, 226 }  ,draw opacity=1 ]   (180,100) -- (190,90) ;
%Straight Lines [id:da11672154866472562] 
\draw [color={rgb, 255:red, 74; green, 144; blue, 226 }  ,draw opacity=1 ]   (180,110) -- (180,100) ;
%Straight Lines [id:da08673540051262763] 
\draw [color={rgb, 255:red, 74; green, 144; blue, 226 }  ,draw opacity=1 ]   (180,120) -- (200,100) ;
%Straight Lines [id:da5442988930650196] 
\draw [color={rgb, 255:red, 74; green, 144; blue, 226 }  ,draw opacity=1 ]   (180,130) -- (180,120) ;
%Straight Lines [id:da14478876707678678] 
\draw [color={rgb, 255:red, 0; green, 0; blue, 0 }  ,draw opacity=1 ]   (240,110) -- (230,90) ;
%Straight Lines [id:da36772493270495443] 
\draw [color={rgb, 255:red, 0; green, 0; blue, 0 }  ,draw opacity=1 ]   (230,90) -- (220,80) ;
%Straight Lines [id:da09537428871892406] 
\draw [color={rgb, 255:red, 0; green, 0; blue, 0 }  ,draw opacity=1 ]   (250,120) -- (240,110) ;
%Straight Lines [id:da6141602470368769] 
\draw [color={rgb, 255:red, 0; green, 0; blue, 0 }  ,draw opacity=1 ]   (230,90) -- (230,80) ;
%Straight Lines [id:da6080682091392537] 
\draw [color={rgb, 255:red, 0; green, 0; blue, 0 }  ,draw opacity=1 ]   (240,120) -- (240,110) ;
%Straight Lines [id:da8099045752779371] 
\draw [color={rgb, 255:red, 0; green, 0; blue, 0 }  ,draw opacity=1 ]   (270,120) -- (310,100) ;
%Straight Lines [id:da33460199298903537] 
\draw [color={rgb, 255:red, 0; green, 0; blue, 0 }  ,draw opacity=1 ]   (260,130) -- (270,120) ;
%Straight Lines [id:da3903008752080197] 
\draw [color={rgb, 255:red, 0; green, 0; blue, 0 }  ,draw opacity=1 ]   (260,120) -- (270,120) ;
%Straight Lines [id:da9360437180150829] 
\draw [color={rgb, 255:red, 0; green, 0; blue, 0 }  ,draw opacity=1 ]   (310,100) -- (320,100) ;
%Straight Lines [id:da5648667719655713] 
\draw [color={rgb, 255:red, 0; green, 0; blue, 0 }  ,draw opacity=1 ]   (310,100) -- (320,90) ;
%Straight Lines [id:da9306746119753091] 
\draw [color={rgb, 255:red, 74; green, 144; blue, 226 }  ,draw opacity=1 ]   (220,90) -- (310,90) ;
%Straight Lines [id:da14700543830497337] 
\draw [color={rgb, 255:red, 74; green, 144; blue, 226 }  ,draw opacity=1 ]   (310,110) -- (310,90) ;
%Straight Lines [id:da13921183194069586] 
\draw [color={rgb, 255:red, 74; green, 144; blue, 226 }  ,draw opacity=1 ]   (310,90) -- (320,80) ;
%Straight Lines [id:da04847870924621678] 
\draw [color={rgb, 255:red, 74; green, 144; blue, 226 }  ,draw opacity=1 ]   (270,90) -- (280,80) ;
%Straight Lines [id:da5625993387327548] 
\draw [color={rgb, 255:red, 74; green, 144; blue, 226 }  ,draw opacity=1 ]   (270,130) -- (270,90) ;
%Straight Lines [id:da7042589229865169] 
\draw [color={rgb, 255:red, 74; green, 144; blue, 226 }  ,draw opacity=1 ]   (230,110) -- (290,110) ;
%Straight Lines [id:da7427693421052832] 
\draw [color={rgb, 255:red, 74; green, 144; blue, 226 }  ,draw opacity=1 ]   (290,120) -- (290,110) ;
%Straight Lines [id:da3184300005720233] 
\draw [color={rgb, 255:red, 74; green, 144; blue, 226 }  ,draw opacity=1 ]   (270,110) -- (280,100) ;
%Straight Lines [id:da523071800107579] 
\draw [color={rgb, 255:red, 74; green, 144; blue, 226 }  ,draw opacity=1 ]   (290,110) -- (300,100) ;
%Straight Lines [id:da287839655545561] 
\draw [color={rgb, 255:red, 0; green, 0; blue, 0 }  ,draw opacity=1 ]   (340,160) -- (370,100) ;
%Straight Lines [id:da18426242272887416] 
\draw [color={rgb, 255:red, 0; green, 0; blue, 0 }  ,draw opacity=1 ]   (370,100) -- (380,90) ;
%Straight Lines [id:da5962644105162809] 
\draw [color={rgb, 255:red, 0; green, 0; blue, 0 }  ,draw opacity=1 ]   (330,170) -- (340,160) ;
%Straight Lines [id:da7406329565276324] 
\draw [color={rgb, 255:red, 0; green, 0; blue, 0 }  ,draw opacity=1 ]   (370,100) -- (370,90) ;
%Straight Lines [id:da25072997854513823] 
\draw [color={rgb, 255:red, 0; green, 0; blue, 0 }  ,draw opacity=1 ]   (340,170) -- (340,160) ;
%Straight Lines [id:da3455954414168335] 
\draw [color={rgb, 255:red, 0; green, 0; blue, 0 }  ,draw opacity=1 ]   (420,70) -- (450,60) ;
%Straight Lines [id:da27639984494047054] 
\draw [color={rgb, 255:red, 0; green, 0; blue, 0 }  ,draw opacity=1 ]   (400,80) -- (420,70) ;
%Straight Lines [id:da597689358261988] 
\draw [color={rgb, 255:red, 0; green, 0; blue, 0 }  ,draw opacity=1 ]   (450,60) -- (470,50) ;
%Straight Lines [id:da49801175363862893] 
\draw [color={rgb, 255:red, 0; green, 0; blue, 0 }  ,draw opacity=1 ]   (420,70) -- (400,70) ;
%Straight Lines [id:da40039520853054866] 
\draw [color={rgb, 255:red, 0; green, 0; blue, 0 }  ,draw opacity=1 ]   (470,60) -- (450,60) ;
%Straight Lines [id:da737651793032231] 
\draw [color={rgb, 255:red, 74; green, 144; blue, 226 }  ,draw opacity=1 ]   (350,140) -- (430,60) ;
%Straight Lines [id:da11771444785018592] 
\draw [color={rgb, 255:red, 74; green, 144; blue, 226 }  ,draw opacity=1 ]   (350,160) -- (460,50) ;
%Straight Lines [id:da9022188837196113] 
\draw [color={rgb, 255:red, 74; green, 144; blue, 226 }  ,draw opacity=1 ]   (360,100) -- (410,100) ;
%Straight Lines [id:da882964397475599] 
\draw [color={rgb, 255:red, 74; green, 144; blue, 226 }  ,draw opacity=1 ]   (340,140) -- (350,140) ;
%Straight Lines [id:da8010410180996987] 
\draw [color={rgb, 255:red, 74; green, 144; blue, 226 }  ,draw opacity=1 ]   (350,150) -- (350,140) ;
%Straight Lines [id:da0996640940070157] 
\draw [color={rgb, 255:red, 74; green, 144; blue, 226 }  ,draw opacity=1 ]   (330,160) -- (350,160) ;
%Straight Lines [id:da9023691842205462] 
\draw [color={rgb, 255:red, 74; green, 144; blue, 226 }  ,draw opacity=1 ]   (350,170) -- (350,160) ;
%Straight Lines [id:da7726113136972891] 
\draw [color={rgb, 255:red, 74; green, 144; blue, 226 }  ,draw opacity=1 ]   (410,110) -- (410,100) ;
%Straight Lines [id:da9900447962684791] 
\draw [color={rgb, 255:red, 74; green, 144; blue, 226 }  ,draw opacity=1 ]   (390,110) -- (390,100) ;
%Straight Lines [id:da3119978054090685] 
\draw [color={rgb, 255:red, 0; green, 0; blue, 0 }  ,draw opacity=1 ]   (140,180) -- (130,170) ;
%Straight Lines [id:da19455239144498715] 
\draw [color={rgb, 255:red, 0; green, 0; blue, 0 }  ,draw opacity=1 ]   (140,190) -- (140,180) ;
%Straight Lines [id:da7798137966242913] 
\draw [color={rgb, 255:red, 0; green, 0; blue, 0 }  ,draw opacity=1 ]   (130,170) -- (130,160) ;
%Straight Lines [id:da6021247918598006] 
\draw [color={rgb, 255:red, 0; green, 0; blue, 0 }  ,draw opacity=1 ]   (150,180) -- (140,180) ;
%Straight Lines [id:da524890409770494] 
\draw [color={rgb, 255:red, 0; green, 0; blue, 0 }  ,draw opacity=1 ]   (130,170) -- (120,170) ;
%Straight Lines [id:da6427157369100719] 
\draw [color={rgb, 255:red, 0; green, 0; blue, 0 }  ,draw opacity=1 ]   (110,180) -- (100,190) ;
%Straight Lines [id:da9274913826727754] 
\draw [color={rgb, 255:red, 0; green, 0; blue, 0 }  ,draw opacity=1 ]   (80,230) -- (70,240) ;
%Straight Lines [id:da5071903041769609] 
\draw [color={rgb, 255:red, 0; green, 0; blue, 0 }  ,draw opacity=1 ]   (80,230) -- (80,240) ;
%Straight Lines [id:da5768135107278562] 
\draw [color={rgb, 255:red, 0; green, 0; blue, 0 }  ,draw opacity=1 ]   (100,180) -- (100,190) ;
%Straight Lines [id:da19084612047275318] 
\draw [color={rgb, 255:red, 0; green, 0; blue, 0 }  ,draw opacity=1 ]   (100,190) -- (80,230) ;
%Straight Lines [id:da9478371061098868] 
\draw [color={rgb, 255:red, 74; green, 144; blue, 226 }  ,draw opacity=1 ]   (90,210) -- (140,160) ;
%Straight Lines [id:da6630617813257678] 
\draw [color={rgb, 255:red, 74; green, 144; blue, 226 }  ,draw opacity=1 ]   (90,230) -- (150,170) ;
%Straight Lines [id:da5361310716528913] 
\draw [color={rgb, 255:red, 74; green, 144; blue, 226 }  ,draw opacity=1 ]   (90,190) -- (130,190) ;
%Straight Lines [id:da11566785908031207] 
\draw [color={rgb, 255:red, 74; green, 144; blue, 226 }  ,draw opacity=1 ]   (70,230) -- (90,230) ;
%Straight Lines [id:da9683165477406207] 
\draw [color={rgb, 255:red, 74; green, 144; blue, 226 }  ,draw opacity=1 ]   (90,220) -- (90,210) ;
%Straight Lines [id:da23845307188069753] 
\draw [color={rgb, 255:red, 74; green, 144; blue, 226 }  ,draw opacity=1 ]   (130,200) -- (130,190) ;
%Straight Lines [id:da5352051304990572] 
\draw [color={rgb, 255:red, 74; green, 144; blue, 226 }  ,draw opacity=1 ]   (90,240) -- (90,230) ;
%Straight Lines [id:da6433336014219859] 
\draw [color={rgb, 255:red, 74; green, 144; blue, 226 }  ,draw opacity=1 ]   (80,210) -- (90,210) ;
%Straight Lines [id:da5874348535238525] 
\draw [color={rgb, 255:red, 0; green, 0; blue, 0 }  ,draw opacity=1 ]   (280,210) -- (250,180) ;
%Straight Lines [id:da25813014887505703] 
\draw [color={rgb, 255:red, 0; green, 0; blue, 0 }  ,draw opacity=1 ]   (250,180) -- (250,170) ;
%Straight Lines [id:da9471818730746632] 
\draw [color={rgb, 255:red, 0; green, 0; blue, 0 }  ,draw opacity=1 ]   (280,220) -- (280,210) ;
%Straight Lines [id:da6757109096302101] 
\draw [color={rgb, 255:red, 0; green, 0; blue, 0 }  ,draw opacity=1 ]   (250,180) -- (240,180) ;
%Straight Lines [id:da7053572914917253] 
\draw [color={rgb, 255:red, 0; green, 0; blue, 0 }  ,draw opacity=1 ]   (290,210) -- (280,210) ;
%Straight Lines [id:da47395437938336304] 
\draw [color={rgb, 255:red, 0; green, 0; blue, 0 }  ,draw opacity=1 ]   (200,200) -- (190,220) ;
%Straight Lines [id:da7387727567809482] 
\draw [color={rgb, 255:red, 0; green, 0; blue, 0 }  ,draw opacity=1 ]   (200,190) -- (200,200) ;
%Straight Lines [id:da9697739430057121] 
\draw [color={rgb, 255:red, 0; green, 0; blue, 0 }  ,draw opacity=1 ]   (190,220) -- (190,230) ;
%Straight Lines [id:da6062138544498404] 
\draw [color={rgb, 255:red, 0; green, 0; blue, 0 }  ,draw opacity=1 ]   (210,190) -- (200,200) ;
%Straight Lines [id:da5150745159241761] 
\draw [color={rgb, 255:red, 0; green, 0; blue, 0 }  ,draw opacity=1 ]   (190,220) -- (180,230) ;
%Straight Lines [id:da01799500912218932] 
\draw [color={rgb, 255:red, 74; green, 144; blue, 226 }  ,draw opacity=1 ]   (190,200) -- (270,200) ;
%Straight Lines [id:da6033297965711083] 
\draw [color={rgb, 255:red, 74; green, 144; blue, 226 }  ,draw opacity=1 ]   (210,220) -- (260,170) ;
%Straight Lines [id:da071168969809847] 
\draw [color={rgb, 255:red, 74; green, 144; blue, 226 }  ,draw opacity=1 ]   (270,220) -- (290,200) ;
%Straight Lines [id:da030816085336112686] 
\draw [color={rgb, 255:red, 74; green, 144; blue, 226 }  ,draw opacity=1 ]   (180,220) -- (270,220) ;
%Straight Lines [id:da4263372124586816] 
\draw [color={rgb, 255:red, 74; green, 144; blue, 226 }  ,draw opacity=1 ]   (270,200) -- (280,190) ;
%Straight Lines [id:da3954801206648785] 
\draw [color={rgb, 255:red, 74; green, 144; blue, 226 }  ,draw opacity=1 ]   (210,230) -- (210,220) ;
%Straight Lines [id:da635112216211101] 
\draw [color={rgb, 255:red, 74; green, 144; blue, 226 }  ,draw opacity=1 ]   (230,210) -- (230,200) ;
%Straight Lines [id:da8866026808095928] 
\draw [color={rgb, 255:red, 74; green, 144; blue, 226 }  ,draw opacity=1 ]   (270,230) -- (270,220) ;
%Straight Lines [id:da8688724673078834] 
\draw [color={rgb, 255:red, 74; green, 144; blue, 226 }  ,draw opacity=1 ]   (270,210) -- (270,200) ;
%Straight Lines [id:da07226348768699353] 
\draw [color={rgb, 255:red, 0; green, 0; blue, 0 }  ,draw opacity=1 ]   (360,220) -- (370,200) ;
%Straight Lines [id:da17280519146783502] 
\draw [color={rgb, 255:red, 0; green, 0; blue, 0 }  ,draw opacity=1 ]   (370,200) -- (370,190) ;
%Straight Lines [id:da2782365732892178] 
\draw [color={rgb, 255:red, 0; green, 0; blue, 0 }  ,draw opacity=1 ]   (370,200) -- (380,190) ;
%Straight Lines [id:da8487097837966301] 
\draw [color={rgb, 255:red, 0; green, 0; blue, 0 }  ,draw opacity=1 ]   (350,230) -- (360,220) ;
%Straight Lines [id:da27888771733358686] 
\draw [color={rgb, 255:red, 0; green, 0; blue, 0 }  ,draw opacity=1 ]   (360,230) -- (360,220) ;
%Straight Lines [id:da3014138223343795] 
\draw [color={rgb, 255:red, 0; green, 0; blue, 0 }  ,draw opacity=1 ]   (380,230) -- (420,210) ;
%Straight Lines [id:da24525095716590706] 
\draw [color={rgb, 255:red, 0; green, 0; blue, 0 }  ,draw opacity=1 ]   (420,210) -- (430,200) ;
%Straight Lines [id:da47219910320806546] 
\draw [color={rgb, 255:red, 0; green, 0; blue, 0 }  ,draw opacity=1 ]   (420,210) -- (430,210) ;
%Straight Lines [id:da7588961833742854] 
\draw [color={rgb, 255:red, 0; green, 0; blue, 0 }  ,draw opacity=1 ]   (370,230) -- (380,230) ;
%Straight Lines [id:da1744252505660866] 
\draw [color={rgb, 255:red, 0; green, 0; blue, 0 }  ,draw opacity=1 ]   (370,240) -- (380,230) ;
%Straight Lines [id:da4343391205128707] 
\draw [color={rgb, 255:red, 74; green, 144; blue, 226 }  ,draw opacity=1 ]   (360,200) -- (420,200) ;
%Straight Lines [id:da6839170203073517] 
\draw [color={rgb, 255:red, 74; green, 144; blue, 226 }  ,draw opacity=1 ]   (420,200) -- (430,190) ;
%Straight Lines [id:da8389702498120113] 
\draw [color={rgb, 255:red, 74; green, 144; blue, 226 }  ,draw opacity=1 ]   (420,220) -- (420,200) ;
%Straight Lines [id:da7700670156701555] 
\draw [color={rgb, 255:red, 74; green, 144; blue, 226 }  ,draw opacity=1 ]   (380,240) -- (380,200) ;
%Straight Lines [id:da4451927410282618] 
\draw [color={rgb, 255:red, 74; green, 144; blue, 226 }  ,draw opacity=1 ]   (380,200) -- (390,190) ;
%Straight Lines [id:da7926463492975235] 
\draw [color={rgb, 255:red, 74; green, 144; blue, 226 }  ,draw opacity=1 ]   (400,220) -- (410,210) ;
%Straight Lines [id:da8838547000500679] 
\draw [color={rgb, 255:red, 74; green, 144; blue, 226 }  ,draw opacity=1 ]   (380,220) -- (390,210) ;
%Straight Lines [id:da5467858614486426] 
\draw [color={rgb, 255:red, 74; green, 144; blue, 226 }  ,draw opacity=1 ]   (400,230) -- (400,220) ;
%Straight Lines [id:da3677285243884417] 
\draw [color={rgb, 255:red, 74; green, 144; blue, 226 }  ,draw opacity=1 ]   (350,220) -- (400,220) ;

% Text Node
\draw (107,132) node [anchor=north west][inner sep=0.75pt]   [align=left] {(YaI)};
% Text Node
\draw (251,132) node [anchor=north west][inner sep=0.75pt]   [align=left] {(YbI)};
% Text Node
\draw (391,132) node [anchor=north west][inner sep=0.75pt]   [align=left] {(YcI)};
% Text Node
\draw (107,222) node [anchor=north west][inner sep=0.75pt]   [align=left] {(YaII)};
% Text Node
\draw (221,232) node [anchor=north west][inner sep=0.75pt]   [align=left] {(YaIII)};
% Text Node
\draw (391,232) node [anchor=north west][inner sep=0.75pt]   [align=left] {(YbII)};

\end{tikzpicture}

\end{center}

\caption{Local pictures of liftable tropical bitangent lines for (YaI), (YbI), (YcI), (YaII), (YaIII), (YbII).}\label{fig:Y}

\end{figure}
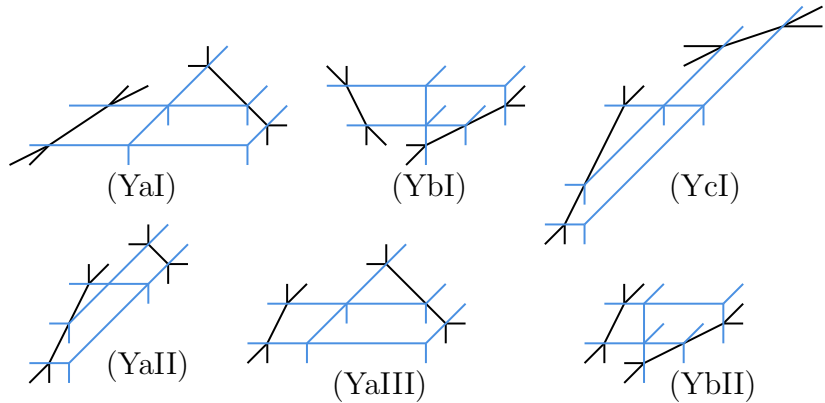

Shape (BB) differs from shape (W) by cutting off two vertices of the parallelogram. The dual motif for (BBa) equals the one for (WaIII), the one for (BBb) equals (WbIII). Again, we refer to Figure \ref{fig:W} for these partial dual motifs. The $4$ liftable tropical bitangents together with local pictures of the tropicalized quartic are found in Figure \ref{fig:BB}.

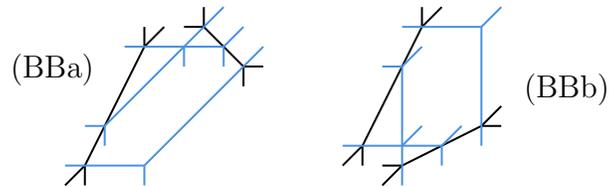
\begin{figure}
\begin{center}

\tikzset{every picture/.style={line width=0.75pt}} %set default line width to 0.75pt        

\begin{tikzpicture}[x=0.75pt,y=0.75pt,yscale=-1,xscale=1]
%uncomment if require: \path (0,784); %set diagram left start at 0, and has height of 784

%Straight Lines [id:da14478876707678678] 
\draw [color={rgb, 255:red, 0; green, 0; blue, 0 }  ,draw opacity=1 ]   (290,90) -- (320,30) ;
%Straight Lines [id:da36772493270495443] 
\draw [color={rgb, 255:red, 0; green, 0; blue, 0 }  ,draw opacity=1 ]   (320,30) -- (330,20) ;
%Straight Lines [id:da09537428871892406] 
\draw [color={rgb, 255:red, 0; green, 0; blue, 0 }  ,draw opacity=1 ]   (280,100) -- (290,90) ;
%Straight Lines [id:da6141602470368769] 
\draw [color={rgb, 255:red, 0; green, 0; blue, 0 }  ,draw opacity=1 ]   (320,30) -- (320,20) ;
%Straight Lines [id:da6080682091392537] 
\draw [color={rgb, 255:red, 0; green, 0; blue, 0 }  ,draw opacity=1 ]   (290,100) -- (290,90) ;
%Straight Lines [id:da8099045752779371] 
\draw [color={rgb, 255:red, 0; green, 0; blue, 0 }  ,draw opacity=1 ]   (310,100) -- (350,80) ;
%Straight Lines [id:da33460199298903537] 
\draw [color={rgb, 255:red, 0; green, 0; blue, 0 }  ,draw opacity=1 ]   (300,110) -- (310,100) ;
%Straight Lines [id:da3903008752080197] 
\draw [color={rgb, 255:red, 0; green, 0; blue, 0 }  ,draw opacity=1 ]   (300,100) -- (310,100) ;
%Straight Lines [id:da9360437180150829] 
\draw [color={rgb, 255:red, 0; green, 0; blue, 0 }  ,draw opacity=1 ]   (350,80) -- (360,80) ;
%Straight Lines [id:da5648667719655713] 
\draw [color={rgb, 255:red, 0; green, 0; blue, 0 }  ,draw opacity=1 ]   (350,80) -- (360,70) ;
%Straight Lines [id:da9306746119753091] 
\draw [color={rgb, 255:red, 74; green, 144; blue, 226 }  ,draw opacity=1 ]   (310,30) -- (350,30) ;
%Straight Lines [id:da14700543830497337] 
\draw [color={rgb, 255:red, 74; green, 144; blue, 226 }  ,draw opacity=1 ]   (350,90) -- (350,30) ;
%Straight Lines [id:da13921183194069586] 
\draw [color={rgb, 255:red, 74; green, 144; blue, 226 }  ,draw opacity=1 ]   (330,90) -- (340,80) ;
%Straight Lines [id:da04847870924621678] 
\draw [color={rgb, 255:red, 74; green, 144; blue, 226 }  ,draw opacity=1 ]   (310,50) -- (320,40) ;
%Straight Lines [id:da5625993387327548] 
\draw [color={rgb, 255:red, 74; green, 144; blue, 226 }  ,draw opacity=1 ]   (310,110) -- (310,50) ;
%Straight Lines [id:da7042589229865169] 
\draw [color={rgb, 255:red, 74; green, 144; blue, 226 }  ,draw opacity=1 ]   (280,90) -- (330,90) ;
%Straight Lines [id:da7427693421052832] 
\draw [color={rgb, 255:red, 74; green, 144; blue, 226 }  ,draw opacity=1 ]   (330,100) -- (330,90) ;
%Straight Lines [id:da3184300005720233] 
\draw [color={rgb, 255:red, 74; green, 144; blue, 226 }  ,draw opacity=1 ]   (310,90) -- (320,80) ;
%Straight Lines [id:da523071800107579] 
\draw [color={rgb, 255:red, 74; green, 144; blue, 226 }  ,draw opacity=1 ]   (350,30) -- (360,20) ;
%Straight Lines [id:da3119978054090685] 
\draw [color={rgb, 255:red, 0; green, 0; blue, 0 }  ,draw opacity=1 ]   (230,50) -- (215.45,35.45) -- (210,30) ;
%Straight Lines [id:da19455239144498715] 
\draw [color={rgb, 255:red, 0; green, 0; blue, 0 }  ,draw opacity=1 ]   (230,60) -- (230,50) ;
%Straight Lines [id:da7798137966242913] 
\draw [color={rgb, 255:red, 0; green, 0; blue, 0 }  ,draw opacity=1 ]   (210,30) -- (210,20) ;
%Straight Lines [id:da6021247918598006] 
\draw [color={rgb, 255:red, 0; green, 0; blue, 0 }  ,draw opacity=1 ]   (240,50) -- (230,50) ;
%Straight Lines [id:da524890409770494] 
\draw [color={rgb, 255:red, 0; green, 0; blue, 0 }  ,draw opacity=1 ]   (210,30) -- (200,30) ;
%Straight Lines [id:da6427157369100719] 
\draw [color={rgb, 255:red, 0; green, 0; blue, 0 }  ,draw opacity=1 ]   (190,30) -- (180,40) ;
%Straight Lines [id:da9274913826727754] 
\draw [color={rgb, 255:red, 0; green, 0; blue, 0 }  ,draw opacity=1 ]   (150,100) -- (140,110) ;
%Straight Lines [id:da5071903041769609] 
\draw [color={rgb, 255:red, 0; green, 0; blue, 0 }  ,draw opacity=1 ]   (150,100) -- (150,110) ;
%Straight Lines [id:da5768135107278562] 
\draw [color={rgb, 255:red, 0; green, 0; blue, 0 }  ,draw opacity=1 ]   (180,30) -- (180,40) ;
%Straight Lines [id:da19084612047275318] 
\draw [color={rgb, 255:red, 0; green, 0; blue, 0 }  ,draw opacity=1 ]   (180,40) -- (150,100) ;
%Straight Lines [id:da9478371061098868] 
\draw [color={rgb, 255:red, 74; green, 144; blue, 226 }  ,draw opacity=1 ]   (160,80) -- (220,20) ;
%Straight Lines [id:da6630617813257678] 
\draw [color={rgb, 255:red, 74; green, 144; blue, 226 }  ,draw opacity=1 ]   (180,100) -- (240,40) ;
%Straight Lines [id:da5361310716528913] 
\draw [color={rgb, 255:red, 74; green, 144; blue, 226 }  ,draw opacity=1 ]   (170,40) -- (220,40) ;
%Straight Lines [id:da11566785908031207] 
\draw [color={rgb, 255:red, 74; green, 144; blue, 226 }  ,draw opacity=1 ]   (140,100) -- (180,100) ;
%Straight Lines [id:da9683165477406207] 
\draw [color={rgb, 255:red, 74; green, 144; blue, 226 }  ,draw opacity=1 ]   (160,90) -- (160,80) ;
%Straight Lines [id:da23845307188069753] 
\draw [color={rgb, 255:red, 74; green, 144; blue, 226 }  ,draw opacity=1 ]   (220,50) -- (220,40) ;
%Straight Lines [id:da5352051304990572] 
\draw [color={rgb, 255:red, 74; green, 144; blue, 226 }  ,draw opacity=1 ]   (180,110) -- (180,100) ;
%Straight Lines [id:da6433336014219859] 
\draw [color={rgb, 255:red, 74; green, 144; blue, 226 }  ,draw opacity=1 ]   (150,80) -- (160,80) ;
%Straight Lines [id:da4070572587026775] 
\draw [color={rgb, 255:red, 74; green, 144; blue, 226 }  ,draw opacity=1 ]   (300,50) -- (310,50) ;
%Straight Lines [id:da7386152835158992] 
\draw [color={rgb, 255:red, 74; green, 144; blue, 226 }  ,draw opacity=1 ]   (200,50) -- (200,40) ;
%Straight Lines [id:da7625767770446974] 
\draw [color={rgb, 255:red, 74; green, 144; blue, 226 }  ,draw opacity=1 ]   (220,40) -- (230,30) ;

% Text Node
\draw (370,52) node [anchor=north west][inner sep=0.75pt]   [align=left] {(BBb)};
% Text Node
\draw (111,42) node [anchor=north west][inner sep=0.75pt]   [align=left] {(BBa)};

\end{tikzpicture}

\end{center}

\caption{Local pictures of liftable tropical  bitangent lines for (BBa) and (BBb).}\label{fig:BB}

\end{figure}

In shape (CC), we also cut off two vertices, but we cut one of those off even further. The dual motif for (CCa) equals (WaIII), the one for (CCb) (WbIII), see Figure \ref{fig:W}. The corresponding local pictures of the tangencies are found in Figure \ref{fig:CC}. (CCaI) and (CCaII) differ by which of the two vertices is cut off further. For (CCb), this is symmetric.

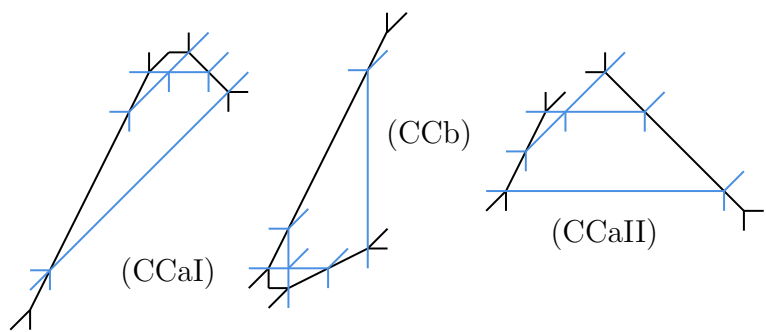
\begin{figure}
\begin{center}

\tikzset{every picture/.style={line width=0.75pt}} %set default line width to 0.75pt        

\begin{tikzpicture}[x=0.75pt,y=0.75pt,yscale=-1,xscale=1]
%uncomment if require: \path (0,784); %set diagram left start at 0, and has height of 784

%Straight Lines [id:da14478876707678678] 
\draw [color={rgb, 255:red, 0; green, 0; blue, 0 }  ,draw opacity=1 ]   (240.29,139) -- (300.29,19) ;
%Straight Lines [id:da36772493270495443] 
\draw [color={rgb, 255:red, 0; green, 0; blue, 0 }  ,draw opacity=1 ]   (300,20) -- (310,10) ;
%Straight Lines [id:da09537428871892406] 
\draw [color={rgb, 255:red, 0; green, 0; blue, 0 }  ,draw opacity=1 ]   (230.29,149) -- (240.29,139) ;
%Straight Lines [id:da6141602470368769] 
\draw [color={rgb, 255:red, 0; green, 0; blue, 0 }  ,draw opacity=1 ]   (300,20) -- (300,10) ;
%Straight Lines [id:da6080682091392537] 
\draw [color={rgb, 255:red, 0; green, 0; blue, 0 }  ,draw opacity=1 ]   (240.29,149) -- (240.29,139) ;
%Straight Lines [id:da8099045752779371] 
\draw [color={rgb, 255:red, 0; green, 0; blue, 0 }  ,draw opacity=1 ]   (250.29,149) -- (290.29,129) ;
%Straight Lines [id:da33460199298903537] 
\draw [color={rgb, 255:red, 0; green, 0; blue, 0 }  ,draw opacity=1 ]   (240.29,159) -- (250.29,149) ;
%Straight Lines [id:da3903008752080197] 
\draw [color={rgb, 255:red, 0; green, 0; blue, 0 }  ,draw opacity=1 ]   (240.29,149) -- (250.29,149) ;
%Straight Lines [id:da9360437180150829] 
\draw [color={rgb, 255:red, 0; green, 0; blue, 0 }  ,draw opacity=1 ]   (290.29,129) -- (300.29,129) ;
%Straight Lines [id:da5648667719655713] 
\draw [color={rgb, 255:red, 0; green, 0; blue, 0 }  ,draw opacity=1 ]   (290.29,129) -- (300.29,119) ;
%Straight Lines [id:da9306746119753091] 
\draw [color={rgb, 255:red, 74; green, 144; blue, 226 }  ,draw opacity=1 ]   (280.29,39) -- (290.29,39) ;
%Straight Lines [id:da14700543830497337] 
\draw [color={rgb, 255:red, 74; green, 144; blue, 226 }  ,draw opacity=1 ]   (290.29,139) -- (290.29,39) ;
%Straight Lines [id:da13921183194069586] 
\draw [color={rgb, 255:red, 74; green, 144; blue, 226 }  ,draw opacity=1 ]   (270.29,139) -- (280.29,129) ;
%Straight Lines [id:da04847870924621678] 
\draw [color={rgb, 255:red, 74; green, 144; blue, 226 }  ,draw opacity=1 ]   (250.29,119) -- (260.29,109) ;
%Straight Lines [id:da5625993387327548] 
\draw [color={rgb, 255:red, 74; green, 144; blue, 226 }  ,draw opacity=1 ]   (250.29,159) -- (250.29,119) ;
%Straight Lines [id:da7042589229865169] 
\draw [color={rgb, 255:red, 74; green, 144; blue, 226 }  ,draw opacity=1 ]   (230.29,139) -- (270.29,139) ;
%Straight Lines [id:da7427693421052832] 
\draw [color={rgb, 255:red, 74; green, 144; blue, 226 }  ,draw opacity=1 ]   (270.29,149) -- (270.29,139) ;
%Straight Lines [id:da3184300005720233] 
\draw [color={rgb, 255:red, 74; green, 144; blue, 226 }  ,draw opacity=1 ]   (250.29,139) -- (260.29,129) ;
%Straight Lines [id:da523071800107579] 
\draw [color={rgb, 255:red, 74; green, 144; blue, 226 }  ,draw opacity=1 ]   (290.29,39) -- (300.29,29) ;
%Straight Lines [id:da3119978054090685] 
\draw [color={rgb, 255:red, 0; green, 0; blue, 0 }  ,draw opacity=1 ]   (220,50) -- (205.45,35.45) -- (200,30) ;
%Straight Lines [id:da19455239144498715] 
\draw [color={rgb, 255:red, 0; green, 0; blue, 0 }  ,draw opacity=1 ]   (220,60) -- (220,50) ;
%Straight Lines [id:da7798137966242913] 
\draw [color={rgb, 255:red, 0; green, 0; blue, 0 }  ,draw opacity=1 ]   (200,30) -- (200,20) ;
%Straight Lines [id:da6021247918598006] 
\draw [color={rgb, 255:red, 0; green, 0; blue, 0 }  ,draw opacity=1 ]   (230,50) -- (220,50) ;
%Straight Lines [id:da524890409770494] 
\draw [color={rgb, 255:red, 0; green, 0; blue, 0 }  ,draw opacity=1 ]   (200,30) -- (190,30) ;
%Straight Lines [id:da6427157369100719] 
\draw [color={rgb, 255:red, 0; green, 0; blue, 0 }  ,draw opacity=1 ]   (190,30) -- (180,40) ;
%Straight Lines [id:da9274913826727754] 
\draw [color={rgb, 255:red, 0; green, 0; blue, 0 }  ,draw opacity=1 ]   (120,160) -- (110,170) ;
%Straight Lines [id:da5071903041769609] 
\draw [color={rgb, 255:red, 0; green, 0; blue, 0 }  ,draw opacity=1 ]   (120,160) -- (120,170) ;
%Straight Lines [id:da5768135107278562] 
\draw [color={rgb, 255:red, 0; green, 0; blue, 0 }  ,draw opacity=1 ]   (180,30) -- (180,40) ;
%Straight Lines [id:da19084612047275318] 
\draw [color={rgb, 255:red, 0; green, 0; blue, 0 }  ,draw opacity=1 ]   (180,40) -- (120,160) ;
%Straight Lines [id:da9478371061098868] 
\draw [color={rgb, 255:red, 74; green, 144; blue, 226 }  ,draw opacity=1 ]   (170,60) -- (210,20) ;
%Straight Lines [id:da6630617813257678] 
\draw [color={rgb, 255:red, 74; green, 144; blue, 226 }  ,draw opacity=1 ]   (120,150) -- (230,40) ;
%Straight Lines [id:da5361310716528913] 
\draw [color={rgb, 255:red, 74; green, 144; blue, 226 }  ,draw opacity=1 ]   (170,40) -- (210,40) ;
%Straight Lines [id:da11566785908031207] 
\draw [color={rgb, 255:red, 74; green, 144; blue, 226 }  ,draw opacity=1 ]   (120,140) -- (130,140) ;
%Straight Lines [id:da9683165477406207] 
\draw [color={rgb, 255:red, 74; green, 144; blue, 226 }  ,draw opacity=1 ]   (170,70) -- (170,60) ;
%Straight Lines [id:da23845307188069753] 
\draw [color={rgb, 255:red, 74; green, 144; blue, 226 }  ,draw opacity=1 ]   (210,50) -- (210,40) ;
%Straight Lines [id:da5352051304990572] 
\draw [color={rgb, 255:red, 74; green, 144; blue, 226 }  ,draw opacity=1 ]   (130,150) -- (130,140) ;
%Straight Lines [id:da6433336014219859] 
\draw [color={rgb, 255:red, 74; green, 144; blue, 226 }  ,draw opacity=1 ]   (160,60) -- (170,60) ;
%Straight Lines [id:da5874348535238525] 
\draw [color={rgb, 255:red, 0; green, 0; blue, 0 }  ,draw opacity=1 ]   (480,110) -- (410,40) ;
%Straight Lines [id:da25813014887505703] 
\draw [color={rgb, 255:red, 0; green, 0; blue, 0 }  ,draw opacity=1 ]   (410,40) -- (410,30) ;
%Straight Lines [id:da9471818730746632] 
\draw [color={rgb, 255:red, 0; green, 0; blue, 0 }  ,draw opacity=1 ]   (480,120) -- (480,110) ;
%Straight Lines [id:da6757109096302101] 
\draw [color={rgb, 255:red, 0; green, 0; blue, 0 }  ,draw opacity=1 ]   (410,40) -- (400,40) ;
%Straight Lines [id:da7053572914917253] 
\draw [color={rgb, 255:red, 0; green, 0; blue, 0 }  ,draw opacity=1 ]   (490,110) -- (480,110) ;
%Straight Lines [id:da47395437938336304] 
\draw [color={rgb, 255:red, 0; green, 0; blue, 0 }  ,draw opacity=1 ]   (380,60) -- (360,100) ;
%Straight Lines [id:da7387727567809482] 
\draw [color={rgb, 255:red, 0; green, 0; blue, 0 }  ,draw opacity=1 ]   (380,50) -- (380,60) ;
%Straight Lines [id:da9697739430057121] 
\draw [color={rgb, 255:red, 0; green, 0; blue, 0 }  ,draw opacity=1 ]   (360,100) -- (360,110) ;
%Straight Lines [id:da6062138544498404] 
\draw [color={rgb, 255:red, 0; green, 0; blue, 0 }  ,draw opacity=1 ]   (390,50) -- (380,60) ;
%Straight Lines [id:da5150745159241761] 
\draw [color={rgb, 255:red, 0; green, 0; blue, 0 }  ,draw opacity=1 ]   (360,100) -- (350,110) ;
%Straight Lines [id:da01799500912218932] 
\draw [color={rgb, 255:red, 74; green, 144; blue, 226 }  ,draw opacity=1 ]   (370,60) -- (430,60) ;
%Straight Lines [id:da6033297965711083] 
\draw [color={rgb, 255:red, 74; green, 144; blue, 226 }  ,draw opacity=1 ]   (370,80) -- (420,30) ;
%Straight Lines [id:da071168969809847] 
\draw [color={rgb, 255:red, 74; green, 144; blue, 226 }  ,draw opacity=1 ]   (470,100) -- (480,90) ;
%Straight Lines [id:da030816085336112686] 
\draw [color={rgb, 255:red, 74; green, 144; blue, 226 }  ,draw opacity=1 ]   (350,100) -- (470,100) ;
%Straight Lines [id:da4263372124586816] 
\draw [color={rgb, 255:red, 74; green, 144; blue, 226 }  ,draw opacity=1 ]   (430,60) -- (440,50) ;
%Straight Lines [id:da3954801206648785] 
\draw [color={rgb, 255:red, 74; green, 144; blue, 226 }  ,draw opacity=1 ]   (370,90) -- (370,80) ;
%Straight Lines [id:da635112216211101] 
\draw [color={rgb, 255:red, 74; green, 144; blue, 226 }  ,draw opacity=1 ]   (390,70) -- (390,60) ;
%Straight Lines [id:da8866026808095928] 
\draw [color={rgb, 255:red, 74; green, 144; blue, 226 }  ,draw opacity=1 ]   (470,110) -- (470,100) ;
%Straight Lines [id:da8688724673078834] 
\draw [color={rgb, 255:red, 74; green, 144; blue, 226 }  ,draw opacity=1 ]   (430,70) -- (430,60) ;
%Straight Lines [id:da4070572587026775] 
\draw [color={rgb, 255:red, 74; green, 144; blue, 226 }  ,draw opacity=1 ]   (240.29,119) -- (250.29,119) ;
%Straight Lines [id:da7386152835158992] 
\draw [color={rgb, 255:red, 74; green, 144; blue, 226 }  ,draw opacity=1 ]   (190,50) -- (190,40) ;
%Straight Lines [id:da706268968929206] 
\draw [color={rgb, 255:red, 74; green, 144; blue, 226 }  ,draw opacity=1 ]   (360,80) -- (370,80) ;
%Straight Lines [id:da7625767770446974] 
\draw [color={rgb, 255:red, 74; green, 144; blue, 226 }  ,draw opacity=1 ]   (210,40) -- (220,30) ;

% Text Node
\draw (298,62) node [anchor=north west][inner sep=0.75pt]   [align=left] {(CCb)};
% Text Node
\draw (164,132) node [anchor=north west][inner sep=0.75pt]   [align=left] {(CCaI)};
% Text Node
\draw (381,112) node [anchor=north west][inner sep=0.75pt]   [align=left] {(CCaII)};

\end{tikzpicture}

\end{center}

\caption{Local pictures of liftable tropical bitangent lines for (CCaI), (CCaII), (CCb).}\label{fig:CC}

\end{figure}

In shape (EE), we only cut off one vertex, but we cut this one off even further again. The dual motif for (EEaI) and (EEaII) equals the one for (WaIII) (those two cases only differ by the choice which vertex of the parallelogram we cut off), the dual motif for (EEaIII) equals the one for  (WaII), (EEbI) equals (WbIII), (EEbII) equals (WbII) and (EEcI) equals (WcII), see Figure \ref{fig:W}. The corresponding local pictures of the tangencies appear in Figure \ref{fig:EE}.

\begin{figure}
\begin{center}

\tikzset{every picture/.style={line width=0.75pt}} %set default line width to 0.75pt        

\begin{tikzpicture}[x=0.75pt,y=0.75pt,yscale=-1,xscale=1]
%uncomment if require: \path (0,784); %set diagram left start at 0, and has height of 784

%Straight Lines [id:da4844215818870313] 
\draw [color={rgb, 255:red, 0; green, 0; blue, 0 }  ,draw opacity=1 ]   (210,100) -- (170,60) ;
%Straight Lines [id:da12554032030916007] 
\draw [color={rgb, 255:red, 0; green, 0; blue, 0 }  ,draw opacity=1 ]   (210,110) -- (210,100) ;
%Straight Lines [id:da4359851805948435] 
\draw [color={rgb, 255:red, 0; green, 0; blue, 0 }  ,draw opacity=1 ]   (220,100) -- (210,100) ;
%Straight Lines [id:da13996763859544692] 
\draw [color={rgb, 255:red, 0; green, 0; blue, 0 }  ,draw opacity=1 ]   (170,60) -- (160,60) ;
%Straight Lines [id:da7277472227092766] 
\draw [color={rgb, 255:red, 0; green, 0; blue, 0 }  ,draw opacity=1 ]   (170,60) -- (170,50) ;
%Straight Lines [id:da592152558439977] 
\draw [color={rgb, 255:red, 0; green, 0; blue, 0 }  ,draw opacity=1 ]   (110,70) -- (130,60) ;
%Straight Lines [id:da5435014520999334] 
\draw [color={rgb, 255:red, 0; green, 0; blue, 0 }  ,draw opacity=1 ]   (110,70) -- (120,60) ;
%Straight Lines [id:da6220477960514808] 
\draw [color={rgb, 255:red, 0; green, 0; blue, 0 }  ,draw opacity=1 ]   (80,90) -- (110,70) ;
%Straight Lines [id:da47914995783450975] 
\draw [color={rgb, 255:red, 0; green, 0; blue, 0 }  ,draw opacity=1 ]   (70,100) -- (80,90) ;
%Straight Lines [id:da20620685137628303] 
\draw [color={rgb, 255:red, 0; green, 0; blue, 0 }  ,draw opacity=1 ]   (60,100) -- (80,90) ;
%Straight Lines [id:da9468606897152186] 
\draw [color={rgb, 255:red, 74; green, 144; blue, 226 }  ,draw opacity=1 ]   (90,70) -- (180,70) ;
%Straight Lines [id:da22333164782127357] 
\draw [color={rgb, 255:red, 74; green, 144; blue, 226 }  ,draw opacity=1 ]   (140,90) -- (180,50) ;
%Straight Lines [id:da9243799011816428] 
\draw [color={rgb, 255:red, 74; green, 144; blue, 226 }  ,draw opacity=1 ]   (160,80) -- (160,70) ;
%Straight Lines [id:da5824160892856032] 
\draw [color={rgb, 255:red, 74; green, 144; blue, 226 }  ,draw opacity=1 ]   (70,90) -- (200,90) ;
%Straight Lines [id:da9403144126510139] 
\draw [color={rgb, 255:red, 74; green, 144; blue, 226 }  ,draw opacity=1 ]   (140,100) -- (140,90) ;
%Straight Lines [id:da7526368259301713] 
\draw [color={rgb, 255:red, 74; green, 144; blue, 226 }  ,draw opacity=1 ]   (180,70) -- (190,60) ;
%Straight Lines [id:da11672154866472562] 
\draw [color={rgb, 255:red, 74; green, 144; blue, 226 }  ,draw opacity=1 ]   (180,80) -- (180,70) ;
%Straight Lines [id:da08673540051262763] 
\draw [color={rgb, 255:red, 74; green, 144; blue, 226 }  ,draw opacity=1 ]   (200,90) -- (210,80) ;
%Straight Lines [id:da5442988930650196] 
\draw [color={rgb, 255:red, 74; green, 144; blue, 226 }  ,draw opacity=1 ]   (200,100) -- (200,90) ;
%Straight Lines [id:da14478876707678678] 
\draw [color={rgb, 255:red, 0; green, 0; blue, 0 }  ,draw opacity=1 ]   (270,130) -- (310,50) ;
%Straight Lines [id:da36772493270495443] 
\draw [color={rgb, 255:red, 0; green, 0; blue, 0 }  ,draw opacity=1 ]   (310,50) -- (320,40) ;
%Straight Lines [id:da09537428871892406] 
\draw [color={rgb, 255:red, 0; green, 0; blue, 0 }  ,draw opacity=1 ]   (260,140) -- (270,130) ;
%Straight Lines [id:da6141602470368769] 
\draw [color={rgb, 255:red, 0; green, 0; blue, 0 }  ,draw opacity=1 ]   (310,50) -- (310,40) ;
%Straight Lines [id:da6080682091392537] 
\draw [color={rgb, 255:red, 0; green, 0; blue, 0 }  ,draw opacity=1 ]   (270,140) -- (270,130) ;
%Straight Lines [id:da8099045752779371] 
\draw [color={rgb, 255:red, 0; green, 0; blue, 0 }  ,draw opacity=1 ]   (280,150) -- (300,140) ;
%Straight Lines [id:da33460199298903537] 
\draw [color={rgb, 255:red, 0; green, 0; blue, 0 }  ,draw opacity=1 ]   (270,160) -- (280,150) ;
%Straight Lines [id:da3903008752080197] 
\draw [color={rgb, 255:red, 0; green, 0; blue, 0 }  ,draw opacity=1 ]   (270,150) -- (280,150) ;
%Straight Lines [id:da9360437180150829] 
\draw [color={rgb, 255:red, 0; green, 0; blue, 0 }  ,draw opacity=1 ]   (300,140) -- (310,140) ;
%Straight Lines [id:da5648667719655713] 
\draw [color={rgb, 255:red, 0; green, 0; blue, 0 }  ,draw opacity=1 ]   (300,140) -- (310,130) ;
%Straight Lines [id:da9306746119753091] 
\draw [color={rgb, 255:red, 74; green, 144; blue, 226 }  ,draw opacity=1 ]   (290,70) -- (300,70) ;
%Straight Lines [id:da14700543830497337] 
\draw [color={rgb, 255:red, 74; green, 144; blue, 226 }  ,draw opacity=1 ]   (300,140) -- (300,70) ;
%Straight Lines [id:da13921183194069586] 
\draw [color={rgb, 255:red, 74; green, 144; blue, 226 }  ,draw opacity=1 ]   (300,130) -- (310,120) ;
%Straight Lines [id:da04847870924621678] 
\draw [color={rgb, 255:red, 74; green, 144; blue, 226 }  ,draw opacity=1 ]   (280,110) -- (290,100) ;
%Straight Lines [id:da5625993387327548] 
\draw [color={rgb, 255:red, 74; green, 144; blue, 226 }  ,draw opacity=1 ]   (280,160) -- (280,110) ;
%Straight Lines [id:da7042589229865169] 
\draw [color={rgb, 255:red, 74; green, 144; blue, 226 }  ,draw opacity=1 ]   (260,130) -- (300,130) ;
%Straight Lines [id:da7427693421052832] 
\draw [color={rgb, 255:red, 74; green, 144; blue, 226 }  ,draw opacity=1 ]   (300,150) -- (300,140) ;
%Straight Lines [id:da3184300005720233] 
\draw [color={rgb, 255:red, 74; green, 144; blue, 226 }  ,draw opacity=1 ]   (280,130) -- (290,120) ;
%Straight Lines [id:da523071800107579] 
\draw [color={rgb, 255:red, 74; green, 144; blue, 226 }  ,draw opacity=1 ]   (300,70) -- (310,60) ;
%Straight Lines [id:da287839655545561] 
\draw [color={rgb, 255:red, 0; green, 0; blue, 0 }  ,draw opacity=1 ]   (330,180) -- (370,100) ;
%Straight Lines [id:da18426242272887416] 
\draw [color={rgb, 255:red, 0; green, 0; blue, 0 }  ,draw opacity=1 ]   (370,100) -- (380,90) ;
%Straight Lines [id:da5962644105162809] 
\draw [color={rgb, 255:red, 0; green, 0; blue, 0 }  ,draw opacity=1 ]   (320,190) -- (323.7,186.3) -- (330,180) ;
%Straight Lines [id:da7406329565276324] 
\draw [color={rgb, 255:red, 0; green, 0; blue, 0 }  ,draw opacity=1 ]   (370,100) -- (370,90) ;
%Straight Lines [id:da25072997854513823] 
\draw [color={rgb, 255:red, 0; green, 0; blue, 0 }  ,draw opacity=1 ]   (330,190) -- (330,180) ;
%Straight Lines [id:da3455954414168335] 
\draw [color={rgb, 255:red, 0; green, 0; blue, 0 }  ,draw opacity=1 ]   (410,70) -- (440,60) ;
%Straight Lines [id:da27639984494047054] 
\draw [color={rgb, 255:red, 0; green, 0; blue, 0 }  ,draw opacity=1 ]   (390,80) -- (410,70) ;
%Straight Lines [id:da597689358261988] 
\draw [color={rgb, 255:red, 0; green, 0; blue, 0 }  ,draw opacity=1 ]   (440,60) -- (460,50) ;
%Straight Lines [id:da49801175363862893] 
\draw [color={rgb, 255:red, 0; green, 0; blue, 0 }  ,draw opacity=1 ]   (410,70) -- (390,70) ;
%Straight Lines [id:da40039520853054866] 
\draw [color={rgb, 255:red, 0; green, 0; blue, 0 }  ,draw opacity=1 ]   (460,60) -- (440,60) ;
%Straight Lines [id:da737651793032231] 
\draw [color={rgb, 255:red, 74; green, 144; blue, 226 }  ,draw opacity=1 ]   (350,140) -- (430,60) ;
%Straight Lines [id:da11771444785018592] 
\draw [color={rgb, 255:red, 74; green, 144; blue, 226 }  ,draw opacity=1 ]   (340,160) -- (450,50) ;
%Straight Lines [id:da9022188837196113] 
\draw [color={rgb, 255:red, 74; green, 144; blue, 226 }  ,draw opacity=1 ]   (360,100) -- (400,100) ;
%Straight Lines [id:da882964397475599] 
\draw [color={rgb, 255:red, 74; green, 144; blue, 226 }  ,draw opacity=1 ]   (340,140) -- (350,140) ;
%Straight Lines [id:da8010410180996987] 
\draw [color={rgb, 255:red, 74; green, 144; blue, 226 }  ,draw opacity=1 ]   (350,150) -- (350,140) ;
%Straight Lines [id:da0996640940070157] 
\draw [color={rgb, 255:red, 74; green, 144; blue, 226 }  ,draw opacity=1 ]   (330,160) -- (340,160) ;
%Straight Lines [id:da9023691842205462] 
\draw [color={rgb, 255:red, 74; green, 144; blue, 226 }  ,draw opacity=1 ]   (340,170) -- (340,160) ;
%Straight Lines [id:da7726113136972891] 
\draw [color={rgb, 255:red, 74; green, 144; blue, 226 }  ,draw opacity=1 ]   (400,110) -- (400,100) ;
%Straight Lines [id:da9900447962684791] 
\draw [color={rgb, 255:red, 74; green, 144; blue, 226 }  ,draw opacity=1 ]   (390,110) -- (390,100) ;
%Straight Lines [id:da3119978054090685] 
\draw [color={rgb, 255:red, 0; green, 0; blue, 0 }  ,draw opacity=1 ]   (160,150) -- (150,140) ;
%Straight Lines [id:da19455239144498715] 
\draw [color={rgb, 255:red, 0; green, 0; blue, 0 }  ,draw opacity=1 ]   (160,160) -- (160,150) ;
%Straight Lines [id:da7798137966242913] 
\draw [color={rgb, 255:red, 0; green, 0; blue, 0 }  ,draw opacity=1 ]   (150,140) -- (150,130) ;
%Straight Lines [id:da6021247918598006] 
\draw [color={rgb, 255:red, 0; green, 0; blue, 0 }  ,draw opacity=1 ]   (170,150) -- (160,150) ;
%Straight Lines [id:da524890409770494] 
\draw [color={rgb, 255:red, 0; green, 0; blue, 0 }  ,draw opacity=1 ]   (150,140) -- (140,140) ;
%Straight Lines [id:da6427157369100719] 
\draw [color={rgb, 255:red, 0; green, 0; blue, 0 }  ,draw opacity=1 ]   (130,150) -- (120,160) ;
%Straight Lines [id:da9274913826727754] 
\draw [color={rgb, 255:red, 0; green, 0; blue, 0 }  ,draw opacity=1 ]   (80,240) -- (70,250) ;
%Straight Lines [id:da5071903041769609] 
\draw [color={rgb, 255:red, 0; green, 0; blue, 0 }  ,draw opacity=1 ]   (80,240) -- (80,250) ;
%Straight Lines [id:da5768135107278562] 
\draw [color={rgb, 255:red, 0; green, 0; blue, 0 }  ,draw opacity=1 ]   (120,150) -- (120,160) ;
%Straight Lines [id:da19084612047275318] 
\draw [color={rgb, 255:red, 0; green, 0; blue, 0 }  ,draw opacity=1 ]   (120,160) -- (80,240) ;
%Straight Lines [id:da9478371061098868] 
\draw [color={rgb, 255:red, 74; green, 144; blue, 226 }  ,draw opacity=1 ]   (110,180) -- (160,130) ;
%Straight Lines [id:da6630617813257678] 
\draw [color={rgb, 255:red, 74; green, 144; blue, 226 }  ,draw opacity=1 ]   (90,220) -- (170,140) ;
%Straight Lines [id:da5361310716528913] 
\draw [color={rgb, 255:red, 74; green, 144; blue, 226 }  ,draw opacity=1 ]   (110,160) -- (150,160) ;
%Straight Lines [id:da11566785908031207] 
\draw [color={rgb, 255:red, 74; green, 144; blue, 226 }  ,draw opacity=1 ]   (80,220) -- (90,220) ;
%Straight Lines [id:da9683165477406207] 
\draw [color={rgb, 255:red, 74; green, 144; blue, 226 }  ,draw opacity=1 ]   (110,190) -- (110,180) ;
%Straight Lines [id:da23845307188069753] 
\draw [color={rgb, 255:red, 74; green, 144; blue, 226 }  ,draw opacity=1 ]   (150,170) -- (150,160) ;
%Straight Lines [id:da5352051304990572] 
\draw [color={rgb, 255:red, 74; green, 144; blue, 226 }  ,draw opacity=1 ]   (90,230) -- (90,220) ;
%Straight Lines [id:da6433336014219859] 
\draw [color={rgb, 255:red, 74; green, 144; blue, 226 }  ,draw opacity=1 ]   (100,180) -- (110,180) ;
%Straight Lines [id:da5874348535238525] 
\draw [color={rgb, 255:red, 0; green, 0; blue, 0 }  ,draw opacity=1 ]   (300,230) -- (250,180) ;
%Straight Lines [id:da25813014887505703] 
\draw [color={rgb, 255:red, 0; green, 0; blue, 0 }  ,draw opacity=1 ]   (250,180) -- (250,170) ;
%Straight Lines [id:da9471818730746632] 
\draw [color={rgb, 255:red, 0; green, 0; blue, 0 }  ,draw opacity=1 ]   (300,240) -- (300,230) ;
%Straight Lines [id:da6757109096302101] 
\draw [color={rgb, 255:red, 0; green, 0; blue, 0 }  ,draw opacity=1 ]   (250,180) -- (240,180) ;
%Straight Lines [id:da7053572914917253] 
\draw [color={rgb, 255:red, 0; green, 0; blue, 0 }  ,draw opacity=1 ]   (310,230) -- (300,230) ;
%Straight Lines [id:da47395437938336304] 
\draw [color={rgb, 255:red, 0; green, 0; blue, 0 }  ,draw opacity=1 ]   (200,200) -- (190,220) ;
%Straight Lines [id:da7387727567809482] 
\draw [color={rgb, 255:red, 0; green, 0; blue, 0 }  ,draw opacity=1 ]   (200,190) -- (200,200) ;
%Straight Lines [id:da9697739430057121] 
\draw [color={rgb, 255:red, 0; green, 0; blue, 0 }  ,draw opacity=1 ]   (190,220) -- (190,230) ;
%Straight Lines [id:da6062138544498404] 
\draw [color={rgb, 255:red, 0; green, 0; blue, 0 }  ,draw opacity=1 ]   (210,190) -- (200,200) ;
%Straight Lines [id:da5150745159241761] 
\draw [color={rgb, 255:red, 0; green, 0; blue, 0 }  ,draw opacity=1 ]   (190,220) -- (180,230) ;
%Straight Lines [id:da01799500912218932] 
\draw [color={rgb, 255:red, 74; green, 144; blue, 226 }  ,draw opacity=1 ]   (190,200) -- (270,200) ;
%Straight Lines [id:da6033297965711083] 
\draw [color={rgb, 255:red, 74; green, 144; blue, 226 }  ,draw opacity=1 ]   (210,220) -- (260,170) ;
%Straight Lines [id:da071168969809847] 
\draw [color={rgb, 255:red, 74; green, 144; blue, 226 }  ,draw opacity=1 ]   (290,220) -- (300,210) ;
%Straight Lines [id:da030816085336112686] 
\draw [color={rgb, 255:red, 74; green, 144; blue, 226 }  ,draw opacity=1 ]   (180,220) -- (290,220) ;
%Straight Lines [id:da4263372124586816] 
\draw [color={rgb, 255:red, 74; green, 144; blue, 226 }  ,draw opacity=1 ]   (270,200) -- (280,190) ;
%Straight Lines [id:da3954801206648785] 
\draw [color={rgb, 255:red, 74; green, 144; blue, 226 }  ,draw opacity=1 ]   (210,230) -- (210,220) ;
%Straight Lines [id:da635112216211101] 
\draw [color={rgb, 255:red, 74; green, 144; blue, 226 }  ,draw opacity=1 ]   (230,210) -- (230,200) ;
%Straight Lines [id:da8866026808095928] 
\draw [color={rgb, 255:red, 74; green, 144; blue, 226 }  ,draw opacity=1 ]   (290,230) -- (290,220) ;
%Straight Lines [id:da8688724673078834] 
\draw [color={rgb, 255:red, 74; green, 144; blue, 226 }  ,draw opacity=1 ]   (270,210) -- (270,200) ;
%Straight Lines [id:da07226348768699353] 
\draw [color={rgb, 255:red, 0; green, 0; blue, 0 }  ,draw opacity=1 ]   (360,220) -- (370,200) ;
%Straight Lines [id:da17280519146783502] 
\draw [color={rgb, 255:red, 0; green, 0; blue, 0 }  ,draw opacity=1 ]   (370,200) -- (370,190) ;
%Straight Lines [id:da2782365732892178] 
\draw [color={rgb, 255:red, 0; green, 0; blue, 0 }  ,draw opacity=1 ]   (370,200) -- (380,190) ;
%Straight Lines [id:da8487097837966301] 
\draw [color={rgb, 255:red, 0; green, 0; blue, 0 }  ,draw opacity=1 ]   (350,230) -- (360,220) ;
%Straight Lines [id:da27888771733358686] 
\draw [color={rgb, 255:red, 0; green, 0; blue, 0 }  ,draw opacity=1 ]   (360,230) -- (360,220) ;
%Straight Lines [id:da3014138223343795] 
\draw [color={rgb, 255:red, 0; green, 0; blue, 0 }  ,draw opacity=1 ]   (380,230) -- (460,190) ;
%Straight Lines [id:da24525095716590706] 
\draw [color={rgb, 255:red, 0; green, 0; blue, 0 }  ,draw opacity=1 ]   (460,190) -- (470,180) ;
%Straight Lines [id:da47219910320806546] 
\draw [color={rgb, 255:red, 0; green, 0; blue, 0 }  ,draw opacity=1 ]   (460,190) -- (470,190) ;
%Straight Lines [id:da7588961833742854] 
\draw [color={rgb, 255:red, 0; green, 0; blue, 0 }  ,draw opacity=1 ]   (370,230) -- (380,230) ;
%Straight Lines [id:da1744252505660866] 
\draw [color={rgb, 255:red, 0; green, 0; blue, 0 }  ,draw opacity=1 ]   (370,240) -- (380,230) ;
%Straight Lines [id:da4343391205128707] 
\draw [color={rgb, 255:red, 74; green, 144; blue, 226 }  ,draw opacity=1 ]   (360,200) -- (440,200) ;
%Straight Lines [id:da6839170203073517] 
\draw [color={rgb, 255:red, 74; green, 144; blue, 226 }  ,draw opacity=1 ]   (440,200) -- (450,190) ;
%Straight Lines [id:da8389702498120113] 
\draw [color={rgb, 255:red, 74; green, 144; blue, 226 }  ,draw opacity=1 ]   (440,210) -- (440,207.34) -- (440,200) ;
%Straight Lines [id:da7700670156701555] 
\draw [color={rgb, 255:red, 74; green, 144; blue, 226 }  ,draw opacity=1 ]   (380,240) -- (380,200) ;
%Straight Lines [id:da4451927410282618] 
\draw [color={rgb, 255:red, 74; green, 144; blue, 226 }  ,draw opacity=1 ]   (380,200) -- (390,190) ;
%Straight Lines [id:da7926463492975235] 
\draw [color={rgb, 255:red, 74; green, 144; blue, 226 }  ,draw opacity=1 ]   (400,220) -- (410,210) ;
%Straight Lines [id:da8838547000500679] 
\draw [color={rgb, 255:red, 74; green, 144; blue, 226 }  ,draw opacity=1 ]   (380,220) -- (390,210) ;
%Straight Lines [id:da5467858614486426] 
\draw [color={rgb, 255:red, 74; green, 144; blue, 226 }  ,draw opacity=1 ]   (400,230) -- (400,220) ;
%Straight Lines [id:da3677285243884417] 
\draw [color={rgb, 255:red, 74; green, 144; blue, 226 }  ,draw opacity=1 ]   (350,220) -- (400,220) ;
%Straight Lines [id:da4070572587026775] 
\draw [color={rgb, 255:red, 74; green, 144; blue, 226 }  ,draw opacity=1 ]   (270,110) -- (280,110) ;
%Straight Lines [id:da7386152835158992] 
\draw [color={rgb, 255:red, 74; green, 144; blue, 226 }  ,draw opacity=1 ]   (130,170) -- (130,160) ;

% Text Node
\draw (107,102) node [anchor=north west][inner sep=0.75pt]   [align=left] {(EEbII)};
% Text Node
\draw (225,82) node [anchor=north west][inner sep=0.75pt]   [align=left] {(EEbI)};
% Text Node
\draw (391,132) node [anchor=north west][inner sep=0.75pt]   [align=left] {(EEaIII)};
% Text Node
\draw (125,192) node [anchor=north west][inner sep=0.75pt]   [align=left] {(EEaI)};
% Text Node
\draw (221,232) node [anchor=north west][inner sep=0.75pt]   [align=left] {(EEaII)};
% Text Node
\draw (391,232) node [anchor=north west][inner sep=0.75pt]   [align=left] {(EEcI)};

\end{tikzpicture}

\end{center}

\caption{Local pictures of liftable tropical bitangent lines for (EEaI), (EEaII), (EEaIII), (EEbI), (EEbII), (EEcI).}\label{fig:EE}

\end{figure}
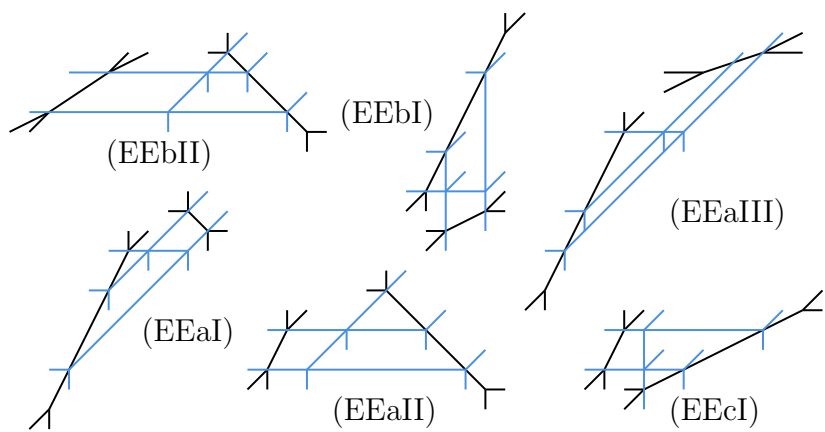

Finally, Figure \ref{fig-II} shows the dual motifs for the bitangent shapes (IIa), (IIb) and (IIc). For (IIa), there is also a local picture of the tropicalized quartic with the liftable tropical bitangent lines. The two left ones have lifting multiplicity one, the right one lifting multiplicity two. The local pictures for (IIb) and (IIc) are similar. For the two left tropical bitangents, the two tangency points tropicalize to the same point, namely a vertex of the tropicalized quartic, where the tropical line and the tropicalized quartic meet with intersection multiplicity four.

\begin{figure}
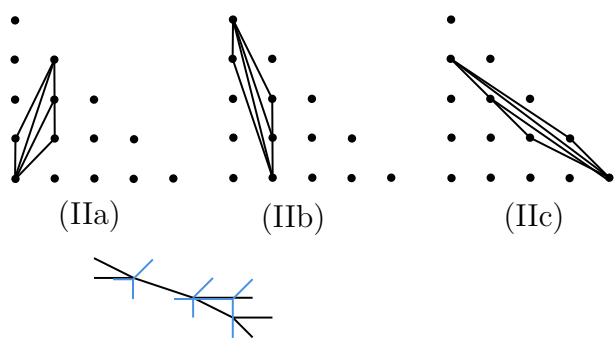

\begin{center}

\tikzset{every picture/.style={line width=0.75pt}} %set default line width to 0.75pt        



\end{center}

\caption{Dual Motifs for (IIa), (IIb) and (IIc). For (IIa), a local picture involving the three liftable tropical bitangent lines and their tropical tangency components.}\label{fig-II}
\end{figure}

\subsection[Details on Theorem \ref{thm-2H}  including case-by-case analysis ]{\for{toc}{Details on Theorem \ref{thm-2H}  including case-by-case analysis}\except{toc}{Details on Theorem \ref{thm-2H} and its proof including the case-by-case analysis of bitangent shapes}}\label{app:2H}

The following Theorem spells out the  details of the case-by-case analysis of bitangent shapes for Theorem \ref{thm-2H}:

\begin{theorem}\label{thm-2Hdetails}

If a tropical bitangent class is of shape (A), (B), (C), (D), (E), (F), (G), (H), (Na), (Oa), (Pa), (Qa), (Qc), (Ra), (Sa), (IIa), (IIb), (T), (Ua), (Uc), (Va), (W), (YcI), (YaII), (CCaII), or  (EE), 
then the total contribution of its $\Qtype$ is $2\mathbb{H}$.

\end{theorem}

\begin{proof}
The strategy of the proof is to use Lemmas \ref{lem-edgeonedgetangency}, \ref{lem-movehorizontal} and \ref{lem-Hb} to argue that the four lifts come in pairs contribution $\langle \pm a\rangle$ for some $a$, which sums up to $2\mathbb{H}$.
This works if the lifts exist in the field, else we use Remark \ref{rem:degree2extension} and Example \ref{ex:degree4extension} to show that nevertheless, we obtain $2\mathbb{H}$ as total contribution.

If it is of shape (A) or (B), there are two tropical tangency components in the interior of edges, of which at least one has to be vertical or horizontal. Using Lemma \ref{lem-edgeonedgetangency}, we can pair up the lifts in such a way that in total we obtain  $2\mathbb{H}$.
The argument is analogous for shape (C), as locally the tropical tangency components are in the interior of edges here, too.
%For (D), if the tangency point in the interior of a horizontal or vertical edge, we can pair up in the same way. 
%
%For (Db) this argument does not work, but a computation shows that the contribution of the other tangency point, which is not in the interior of the diagonal edge, equals $a_{21}$ resp.\ $-a_{21}$. Using Lemma \ref{lem-edgeonedgediagonal}, we can now pair up and conclude that the total contribution is again $2\mathbb{H}$.
Shape (D) was already discussed in detail in the proof of Theorem \ref{thm-2H}.

For (E),(G), one of the tropical tangency components has to be on the horizontal or vertical ray, and we can use Lemma \ref{lem-edgeonedgetangency} resp.\ Lemma \ref{lem-movehorizontal}, depending on which tropical tangency component it is, to pair up.

For (Fa) and (Fb), we have a tropical tangency component in the interior of a horizontal or vertical edge, and we can use Lemma \ref{lem-edgeonedgetangency}  to pair up.
For (Fc), we have two liftable tropical bitangents. Both have a segment of intersection with $\Trop(Q)$ on the diagonal ray, so by Lemma \ref{lem-edgeonedgediagonal}, we get the same contribution. The second tropical tangency component differs: for the first, it is, as in (Ec), an intersection with the horizontal ray, for the second, it is an intersection with the vertex. A computation shows that the contributions are negative of each other, so we can pair up and obtain $2\mathbb{H}$.

For (Ha), one tropical tangency component is in the interior of a horizontal edge. For (Hb), we can pair up using Lemma \ref{lem-Hb} and Lemma \ref{lem-edgeonedgediagonal}.

For (Na), (Oa), (Pa), (Ra), (Sa) we can pair up using the tropical tangency components on the horizontal resp.\ vertical ray.
For (Qa), for one liftable point we can use the tropical tangency component on the horizontal edge to pair up, for the other, the tropical tangency component at the vertex of the tropical line as in Lemma \ref{lem-Hb}.
For (Qc), a computation shows that we obtain $\langle 1 \rangle+\langle 1\rangle+\langle -1\rangle+\langle -1\rangle= 2\mathbb{H}$.
For (IIa) and (IIb), we can use a computation to show that we obtain $2\mathbb{H}$.
For (Ta), Lemma \ref{lem-Hb} shows that we can pair up.
For (Tb), we can use Lemma \ref{lem-movehorizontal} to pair up. 
For (Ua), we can use Lemma \ref{lem-Hb} and Lemma \ref{lem-edgeonedgetangency} to pair up.

For (Uc), a computation shows that the lifts of the right tropical tangency components yield a contribution of $+2$ for the lower liftable tropical bitangent, and $-2$ for the upper liftable tropical bitangent. Combining with Lemma \ref{lem-movehorizontal}, we can pair up.

For (Va), we can use the tropical tangency components on the horizontal and vertical edge to pair up, using Lemma \ref{lem-edgeonedgetangency}.
For (W), one of the tropical tangency component has to be on a horizontal or vertical ray, and thus we can use Lemma \ref{lem-movehorizontal} to pair up. 
For (EEaII), (EEbI), (EEbII), (EEcI, and (EEcII)  we can use Lemma \ref{lem-movehorizontal} to pair up. For (YcI), (YaII), (CCaII), (EEaI) and (EEaIII), we can pair up two liftable tropical bitangents using Lemma \ref{lem-movehorizontal}. For the remaining two, a computation shows that we can again lift up, similar to (Fc).

\end{proof}

Following this classification we can prove the version of this Theorem stated in the introduction:

\begin{proof}[Proof of Theorem \ref{thm:2Hintro}]

The compact bitangent shapes are (A), (B), (C), (D), (E), (F), (G) and (W).
These all yield  $2\mathbb{H}$ as $\Qtype$ by Theorem \ref{thm-2Hdetails}.
\end{proof}

\subsection{Exceptional $\GW$-multiplicities}\label{app-tableQtypes}\label{subsec-Qtypetable}

Here we discuss the $\GW$-multiplicities of the special bitangent shapes which do not yield $2\mathbb{H}$ by Theorem \ref{thm-2H} resp.\ \ref{thm-2Hdetails}.

The shapes (Nb), (Ob), (Oc), (Pb), (Qb), (Rb), (Rc), (Sb), (Ub), (Vb), (IIc), (YbI), (YbII) and (CCb) all have similar behaviour of the tropical tangency components as described in the classification in Appendix \ref{app:classification}.
The $\GW$-multiplicity of a bitangent class of one of these shapes equals
$$ \langle	1 \rangle +\langle 1 \rangle+\mathbb{H}.$$

The following table shows the $\GW$-multiplicities of the remaining exceptional cases.

\hspace{0.2cm}

\begin{tabular}{|c|c|c|c|c|}
\hline
bitangent & &&&\\ 
\hline
%(II)c &  $\langle -2 \rangle $ &$\langle 2\rangle$ &$\langle 1\rangle$ & $\langle 1\rangle$ \\
%(F)c &  $\langle	(-1)^{k}  a_{10}a_{k3-k}^{2-k} a_{k+12-k}^{3-k} \rangle $ &$\langle " \rangle$ & $\langle -2 " \rangle$ & $\langle -2 "\rangle$  \\
%(Nb) ... & $\langle	1 \rangle $ &$\langle 1 \rangle$ & $\mathbb{H}$ & \\

(YaI) & $\langle a_{20}a_{31}a_{30}a_{03}\rangle$ &$\langle  2\rangle$ & $\mathbb{H}$ &  \\
%(Y)bI & $\langle 1\rangle $ &$\langle  -1 \rangle$ & $\langle 1 \rangle$ & $\langle  1\rangle$ \\
%(Y)cI  & $\langle 2 a_{01}a_{12}a_{20}a_{13}\rangle $ &$\langle  - a_{01}a_{12}a_{20}a_{13} \rangle$ & $\langle a_{01}a_{21}\rangle$ & $\langle- "\rangle$ \\
%(Y)aII  & $\langle 2 a_{01}a_{21}\rangle $ &$\langle  -a_{01}a_{21} \rangle$ & $\langle 2a_{01}a_{20}a_{31}a_{30} \rangle$ & $\langle- "\rangle$ \\
(YaIII)  & $\langle -a_{01}a_{20}a_{31}a_{30} \rangle$ &$\langle 2 \rangle$ &  $\mathbb{H}$&  \\
%(Y)bII  & $\langle 1 \rangle $ &$\langle  -1 \rangle$ & $\langle 1 \rangle$ & $\langle 2 \rangle$ \\
(BBa)  & $\langle  -a_{01}a_{20}a_{31}a_{30} \rangle$ &$\langle 2 \rangle$ & $\mathbb{H}$ &  \\
(BBb)  & $\langle 1 \rangle $ &$\langle  1 \rangle$ & $\langle  1 \rangle$ & $\langle 1 \rangle$ \\
(CCaI)  & $\langle -2a_{20}a_{31}a_{30}a_{01} \rangle$  &$\langle 2 \rangle $ & $\mathbb{H}$ & \\
%(CCaII)  & $\langle 1 \rangle $ &$\langle  2a_{01}a_{21} \rangle$ & $\langle  -a_{01}a_{21} \rangle$ & $\langle -1 \rangle$ \\
%(CCb)  & $\langle 1 \rangle $ &$\langle  2 \rangle$ & $\langle  2 \rangle$ & $\langle -2 \rangle$ \\
%(EEaI)  & $\langle -a_{01}a_{21} \rangle $ &$\langle  2a_{01}a_{21} \rangle$ & $\langle  2a_{01}a_{20}a_{31}a_{30} \rangle$ & $\langle -a_{20}a_{01}a_{30}a_{31} \rangle$ \\
%(EEaIII)  & $\langle a_{01}a_{21} \rangle $ &$\langle  -2a_{01}a_{21} \rangle$ & $\langle  2a_{01}a_{20}a_{12}a_{13}  \rangle$ & $\langle -a_{01}a_{20}a_{12}a_{13} \rangle$ \\

\hline
\end{tabular}

\end{document}